\documentclass[a4paper,10pt]{amsart}
\usepackage[utf8]{inputenc}
\usepackage[american]{babel}
\usepackage{header}

\usetikzlibrary{backgrounds, calc, patterns}
\usepackage{enumitem}
\usepackage{fixltx2e}
\usepackage{caption}

\newcommand{\cells}{\!\diamond}

\newcommand{\cyl}{\mathsf{Cyl}}
\newcommand{\ffjc}{Fibered Farrell--Jones Conjecture}
\newcommand{\ffjcw}{Fibered Farrell--Jones Conjecture with wreath products}
\newcommand{\fic}{Fibered Isomorphism Conjecture}
\newcommand{\ficw}{Fibered Isomorphism Conjecture with wreath products}
\newcommand{\finprod}{\sideset{}{^{\mathrm{fin}}}\prod}
\newcommand{\gen}[1]{\langle #1 \rangle}
\newcommand{\halftimes}{\leftthreetimes}

\newcommand{\rep}{\cR ep}
\newcommand{\skel}[2]{\mathrm{sk}_{#2}#1}

\DeclareMathOperator{\coker}{coker}

\DeclareMathOperator{\inv}{inv}
\DeclareMathOperator{\map}{map}
\DeclareMathOperator{\inc}{inc}
\DeclareMathOperator{\swan}{Sw}
\DeclareMathOperator{\swana}{Sw^A}
\DeclareMathOperator{\swindle}{sw}
\DeclareMathOperator{\tel}{Tel}
\DeclareMathOperator{\transfer}{tr}
\DeclareMathOperator{\vcd}{vcd}

\title[Farrell--Hsiang method for algebraic K-theory of spaces]{On the Farrell--Jones Conjecture for algebraic K-theory of spaces: the Farrell--Hsiang method}

\author{Mark Ullmann}
\address{FU Berlin, Institut f\"ur Mathematik\\
	Arnimallee 7 \\ 14195 Berlin \\ Germany}
\email{mark.ullmann@math.fu-berlin.de}

\author{Christoph Winges}
\address{Hausdorff Research Institute for Mathematics (HIM)\\
	Poppelsdorfer Allee 45 \\ 53115 Bonn \\ Germany}
\email{christoph.winges@wwu.de}


\begin{document}

\begin{abstract}
 We prove the Farrell--Jones Conjecture for algebraic $K$-theory of spaces for virtually poly-$\ZZ$-groups.
 For this, we transfer the ``Farrell--Hsiang method'' from the linear case to categories of equivariant, controlled retractive spaces.
\end{abstract}

\maketitle

\section{Introduction}

The classification of high-dimensional manifolds and the understanding of their automorphism groups is a long-standing question in algebraic topology. The classification of high-dimensional manifolds turned out to be intimately related to the algebraic $K$- and $L$-theory of group rings \cite{Wall-Surgery}, while the understanding of automorphism groups has a deep connection to pseudo-isotopy theory and Waldhausen's algebraic $K$-theory of spaces \cite{WJR2013, WW-Autos3}.

Since the 1970s, a lot of progress was made to calculate the algebraic $K$- and $L$-theory of group rings. This culminated in what is now called the \emph{Farrell--Jones Conjecture}, first stated in \cite{FJ1993}. For algebraic $K$-theory, it predicts that the algebraic $K$-theory of $R[G]$ can be computed from the algebraic $K$-theory of a certain set of subgroups by a homological recipe. The Farrell--Jones conjectures for algebraic $K$- and $L$-theory of group rings have been a highly active research area over the last 20 years. Though many cases of the conjectures are known by now, they remain open in general.

In the language of \cite{DL1998}, the conjecture takes the following form: Let $\FF$ be a homotopy invariant functor from spaces to spectra. Let $X$ be a connected CW-complex with fundamental group $G$. Then we obtain an induced functor $\FF_X$ from the orbit category $Or(G)$ to the category of spectra which sends $G/H$ to $\FF(\tilde{X} \times_G G/H)$. By the methods of \cite{DL1998}, $\FF_X$ gives rise to a $G$-homology theory $H_*^G(-; \FF_X)$. The \emph{Farrell--Jones Conjecture} for $\FF_X$ predicts that the \emph{assembly map}
\begin{equation}\label{eq:fjc}
 H_n^G(E_{\cV\cC yc} G; \FF_X) \to H_n^G(G/G;\FF_X) \cong \pi_n\FF(X),
\end{equation}
which is induced by the projection map $E_{\cV\cC yc} G \to G/G$ from the classifying space for virtually cyclic subgroups to a point, is an isomorphism for all $n \in \ZZ$. If we choose $\FF$ to be non-connective algebraic $K$-theory $\KK^{-\infty}(\ZZ[\pi_1(-)])$ or Ranicki's ultimate lower $L$-theory $\LL^{-\infty}(\ZZ[\pi_1(-)])$, we obtain the $K$- respectively $L$-theoretic Farrell--Jones conjecture for group rings.

In this article, we consider the case that $\FF(-) = \Aa^{-\infty}(-)$, a non-connective delooping of Waldhausen's algebraic $K$-theory of spaces. We obtain the following result:

\begin{theorem}\label{intro_poly-z-groups}
 Let $X$ be a connected CW-complex with fundamental group $G$. If $G$ is a virtually poly-$\ZZ$-group, then the assembly map
 \begin{equation*}
  H_n^G(E_{\cV\cC yc} G; \Aa^{-\infty}_X) \to \pi_n\Aa^{-\infty}(X)
 \end{equation*}
 is an isomorphism for all $n \in \ZZ$.
\end{theorem}

In addition to the algebraic $K$- and $L$-theory of group rings, Farrell and Jones also stated conjecture \eqref{eq:fjc} for pseudo-isotopy. They went on to prove the pseudo-isotopy version of the conjecture for spaces whose fundamental group is a cocompact lattice in an almost connected Lie group, assuming that it holds for spaces whose fundamental group is virtually poly-$\ZZ$. However, the announced proof of this special case was never published, cf.~\cite[Remark 7.1]{BFL2014}.
%It is likely that there exists a version of the Waldhausen--Jahren--Rognes theorem \cite{WJR2013} which translates Theorem \ref{intro_poly-z-groups} from $A$-theory to pseudo-isotopy theory, thus filling this gap in the published literature.
Since the isomorphism conjecture in $A$--theory is in fact equivalent to the (topological, PL and smooth) pseudo-isotopy version, cf.~\cite[Lemma~3.3]{ELPUW}, the present article closes this gap in the published literature.

We also give a description of the $A$-theory of spaces with finite fundamental group, which is similar to Lemma~4.1 of \cite{BL2007}. Call a finite group $D$ a \emph{Dress group} if there are primes $p$ and $q$ and a normal series $P \unlhd C \unlhd D$ such that $P$ is a $p$-group, $C/P$ is cyclic and $D/C$ is a $q$-group.

\begin{theorem}\label{intro_finitegroups}
 Let $X$ be a connected CW-complex with finite fundamental group $G$. Let $\cD$ denote the family of Dress subgroups of $G$. Then the assembly map
 \begin{equation*}
  H_n^G(E_{\cD} G; \Aa^{-\infty}_X) \to \pi_n\Aa^{-\infty}(X)
 \end{equation*}
 is an isomorphism for all $n \in \ZZ$.
\end{theorem}

%Theorems \ref{intro_poly-z-groups} and \ref{intro_finitegroups} are not the first results of this type. For the algebraic $K$-theory and $L$-theory of group rings, the last years have seen dramatic progress on the Farrell--Jones Conjecture. To name some important results, the conjecture has been shown to hold for word-hyperbolic groups \cite{BLR2008}, $\mathrm{CAT(0)}$-groups \cite{BL2012a, Wegner2012}, lattices in almost connected Lie groups \cite{BFL2014, KLR-FJCForLatticesInLieGroups}, subgroups of $GL_n(\ZZ)$ \cite{BLRR2014} and solvable groups \cite{W-FJCSolvableGroups}.
Theorems \ref{intro_poly-z-groups} and \ref{intro_finitegroups} are not the first results of this type. For the algebraic $K$-theory and $L$-theory of group rings, the last years have seen dramatic progress on the Farrell--Jones Conjecture. To name some important results, the conjecture has been shown to hold for word-hyperbolic groups \cite{BLR2008}, $\mathrm{CAT(0)}$-groups \cite{BL2012a, Wegner2012}, lattices in almost connected Lie groups \cite{BFL2014, KLR-FJCForLatticesInLieGroups}, subgroups of $GL_n(\ZZ)$ \cite{BLRR2014} and $GL_n(\QQ)$ as well as $GL_n(F(t))$ for any finite field $F$ \cite{Rueping2016}, solvable groups \cite{W-FJCSolvableGroups}, and mapping class groups \cite{BB-MCG}.

The proofs of these results make heavy use of a set of ideas known as ``controlled algebra'', which go back to work of Connell--Hollingsworth \cite{CH1969} and Quinn \cite{Quinn1979}. It was shown in \cite{BFJR2004} that the methods of controlled algebra can be used to produce explicit models for the (equivariant) assembly map $H_n^G(E_{VCyc} G; \KK^{-\infty}_R) \to K_n(R[G])$. Precursors of this model appeared for example in \cite{PW-homology} and \cite{ACFP1994}. All recent proofs of the Farrell--Jones Conjecture use this setup, and rely on at least one of two sufficient criteria to prove the conjecture: the notions of \emph{transfer reducibility} and \emph{being a Farrell--Hsiang group} (cf.~\cite{Bartels-OnProofs}).

For the algebraic $K$-theory of spaces, a.k.a.~$A$-theory, Vogell used the ideas of controlled algebra in the setting of retractive spaces to describe an $A$-theory assembly map \cite{Vogell1990, CPV1998}.
These models were recast by Weiss in \cite{Weiss2002} to repair some problems with the original approach.

In this article, we promote Weiss' categories of controlled retractive spaces to the equivariant setting (even though our notions of weak equivalence are closer to those of Carlsson--Pedersen--Vogell). We give a self-contained discussion of the categories of equivariant, controlled retractive spaces. We prove a number of theorems modelled after those of \cite{BFJR2004}, and produce a model for the equivariant $A$-theory assembly map. One particular feature of our treatment lies in the fact that we can re-use a considerable amount of results from \cite{BFJR2004} and subsequent work.

One obtains a category (of controlled retractive spaces) whose $K$-theory vanishes if and only if the assembly map for $G$ is an isomorphism. In the linear case, proofs of the Farrell--Jones Conjecture proceed by using the notions of transfer reducibility or Farrell--Hsiang group to show that the $K$-theory of a similar ``obstruction category'' is trivial.
We adapt ``Farrell--Hsiang method'' to our setting:

\begin{definition}\label{intro_dfhgroup}
 Let $G$ be a group and $S$ be a finite, symmetric generating set of $G$. Let $\cF$ be a family of subgroups of $G$.
 
 Call $(G,S)$ a \emph{Dress--Farrell--Hsiang group with respect to $\cF$} if there exists $N \in \NN$ such that for every $\epsilon > 0$ there is an epimorphism $\pi \colon G \twoheadrightarrow F$ onto a finite group $F$
 such that the following holds: For every Dress group $D \leq F$, there are a $\overline{D} := \pi^{-1}(D)$-simplicial complex $E_D$ of dimension at most $N$ whose isotropy groups lie in $\cF$, and a $\overline{D}$-equivariant map $\phi_D \colon G \to E_D$ such that $d^{\ell^1}(\phi_D(g),\phi_D(g')) \leq \epsilon$ whenever $g^{-1}g' \in S$.
\end{definition}

\begin{theorem}\label{intro_fhmethod}
 Let $X$ be a connected CW-complex with fundamental group $G$. Let $\mathcal{F}$ be a family of subgroups of $G$. If $G$ is a Dress--Farrell--Hsiang group with respect to $\mathcal{F}$, then the assembly map
 \begin{equation*}
  H_n^G(E_{\mathcal{F}} G; \mathbb{A}^{-\infty}_X) \to \pi_n\Aa^{-\infty}(X)
 \end{equation*}
 is an isomorphism for all $n \in \mathbb{Z}$.
\end{theorem}

Theorem \ref{intro_finitegroups} follows immediately from this result. Theorem \ref{intro_poly-z-groups} is deduced in Section \ref{sec_applications} following the strategy of \cite{BFL2014}, using previous results from \cite{W-TransferReducibilityOfFHGroups} that all relevant instances of Farrell--Hsiang groups are actually Dress--Farrell--Hsiang.

Using the framework we develop here, the proof of the Farrell--Jones Conjecture for transfer reducible groups can also be adapted to the $A$--theory setting, see \cite{ELPUW}.

\subsection{Structure of the article}
Let us outline the structure of this article.

In the first half of this article, we set up the technical background for our constructions.
In Section~\ref{sec_controlledcws}, we define the notion of a coarse structure and
explain how a coarse structure $\fZ$ gives rise to the notion of a
\emph{controlled space} relative to a base space $W$.
In Section~\ref{sec_controlledretractivespaces}, we use these notions to construct
the category of \emph{controlled retractive spaces} $\cR(W,\fZ)$ relative to a base
space $W$.  Every subspace $A$ of $\fZ$ gives rise to a class of weak
equivalences $h^A$ on $\cR(W,\fZ)$; if $A$ is empty, this gives a notion of
homotopy equivalence.  We show that the category $\cR(W,\fZ)$ together with
the weak equivalences $h^A$ is a Waldhausen category which has a Cylinder
Functor and satisfies the Saturation Axiom and the Cylinder Axiom.  As usual,
we need some finiteness condition to make algebraic $K$-theory nontrivial, so
we define subcategories of finite, homotopy finite and finitely
dominated objects. In fact, we work $G$-equivariantly and obtain in particular a Waldhausen category
of finite, $G$-equivariant controlled retractive spaces $\cR^G_f(W, \fZ)$.

In Section~\ref{sec_comparisonthms}, we compare the different finiteness
conditions and show that the resulting (connective) algebraic $K$-theory
differs at most in degree $0$.  We show that we have a version of
Waldhausen's Fibration Theorem which applies in our situation,
even though the $h^A$-equivalences do not satisfy the Extension Axiom.
We call this the Modified Fibration Theorem and prove it as
Proposition~\ref{prop_modifiedfibrationtheorem}. (Such a statement was already used in \cite{Weiss2002}.)
We use this to construct homotopy fiber sequences which compare homotopy equivalences and 
$h^A$-equivalences and show an excision result, the ``coarse Mayer--Vietoris Theorem'' \ref{thm_connectivemayervietoris}.
These results still have a certain ``defect'' in degree $0$, which will be corrected in the next section.
The section concludes with a criterion for the vanishing of the algebraic $K$-theory of the categories $\cR^G_f(W, \fZ)$.

In Section~\ref{sec_nonconnectiveA}, we define a delooping of the algebraic $K$-theory space $K(\cR^G_f(W,\fZ))$
to obtain the \emph{non-connective algebraic $K$-theory spectrum} $\KK^{-\infty}(\cR^G_f(W,\fZ))$.
We establish non-connective versions of the homotopy fiber sequences and ``coarse Mayer--Vietoris Theorem'' from the previous section.
In particular, this repairs the ``defect'' in degree $0$ of the connective case.

The second half of the article discusses the Farrell--Jones Conjecture for $A$-theory.
Section \ref{sec_assembly} constructs a model for the assembly map.
As in the linear case, there exists for any $G$-CW-complex $X$ a coarse structure $\JJ(X)$ which,
together with a certain class of weak equivalences $h^{\infty}$, makes $\KK^{-\infty}(\cR^G_f(W, \JJ(-)), h^{\infty})$
into a $G$-homology theory. If $W$ is a free $G$-CW-complex, we identify its coefficients with
$\Aa^{-\infty}(H \backslash W)$. Here, $\Aa^{-\infty}(V)$ is a non-connective delooping of Waldhausen's algebraic
$K$-theory of spaces $A(V)$, which we define using the results of Section~\ref{sec_nonconnectiveA}.
Applying the $G$-homology theory to the map $E_\cF G \rightarrow G/G$ gives the assembly map.
We conclude with a criterion when this assembly map is a weak equivalence.

In Section~\ref{sec_outline}, we recall the \fic{} for $A$-theory. We define the notion of a
Dress-Farrell-Hsiang group with respect to a family $\cF$.
Theorem~\ref{thm_fhmethod} states that the \fic{} is true for this class of
groups. Imitating \cite{BL2012b}, we show how the theorem follows once we know
Corollary \ref{cor:uniformtransfer:deloop} and Theorem \ref{thm_squeezing}.
Theorem \ref{intro_fhmethod} is a special case of Theorem \ref{thm_fhmethod}.

In Section~\ref{sec_swangroup}, we introduce the \emph{$A$-theoretic Swan group}
and show that it acts on the $K$-theory of the categories $\cR^G_f(W,\JJ(X))$.
We prove an analog of Swan's Induction Theorem as Theorem~\ref{thm_inductionswana}.
This is used to construct a ``transfer map'' in Section~\ref{sec_transfer}.
Section~\ref{sec_squeezing} contains a proof of the ``Squeezing Theorem'' \ref{thm_squeezing}.

Section~\ref{sec_applications} is devoted to applications. 
We prove Theorem~\ref{intro_finitegroups} and proceed to show Theorem \ref{intro_poly-z-groups} following the strategy of \cite{BFL2014}.
We state the ``\ffjcw{} in $A$-theory'', establish the usual inheritance properties and generalize Theorem \ref{thm_fhmethod}
to cover this case as well. We conclude with the proof that virtually poly-$\ZZ$-groups satisfy the \ffjcw{} in $A$-theory.

\subsection*{Acknowledgments}
This work was supported by the DFG: The first author was supported by project C1 of CRC 647 in Berlin; the second author was supported by project B5 of CRC 878 in M\"unster.
The results of Section \ref{sec_swangroup} are based on parts of the second author's PhD thesis written under the supervision of Arthur Bartels.
The second author would like to thank Arthur Bartels and Michael Weiss for helpful discussions.
Finally, we would like to thank Malte Pieper for comments on a draft of this article.

\setcounter{tocdepth}{1}
\tableofcontents

\section{Controlled equivariant CW-complexes}\label{sec_controlledcws}
Throughout this article, $G$ will denote a discrete group and $W$ will denote a $G$-space.

\begin{definition}\label{def_coarsestructure}
 Let $Z$ be a $G$-space which is Hausdorff.
 A set of \emph{morphism control conditions} $\fC$ is a collection of $G$-invariant subsets of $Z \times Z$
 with the following properties:
 \begin{enumerate}[label={(C{\arabic*})}]
  \item \label{cont:item:diag} Every $C \in \fC$ contains the diagonal $\Delta(Z) := \{ (z,z) \mid z \in Z \}$.
  \item \label{cont:item:symmetric} Every $C \in \fC$ is symmetric.
  \item \label{cont:item:union} For all $C, C' \in \fC$ there is some $C'' \in \fC$ such that $C \cup C' \subset C''$.
  \item \label{cont:item:composition} For all $C, C' \in \fC$ there is some $C'' \in \fC$ such that $C' \circ C \subset C''$,
   where the composition $C' \circ C$ is defined as
   \begin{equation*}
    C' \circ C := \{ (z'',z) \mid \exists z' \colon (z',z) \in C, (z'',z') \in C' \}.
   \end{equation*}
 \end{enumerate}
 A set of \emph{object support conditions} $\fS$ is a collection of $G$-invariant subsets of $Z$
 with the following properties:
 \begin{enumerate}[label={(S\arabic*})]
   \item \label{support:item:union} For all $S, S' \in \fS$ there is some $S'' \in \fS$ such that $S \cup S' \subset S''$.
 \end{enumerate}
 The triple $\fZ = (Z,\fC,\fS)$ is called a \emph{coarse structure}.
\end{definition}

Note that conditions~\ref{cont:item:diag} and \ref{cont:item:composition} imply condition~\ref{cont:item:union}.
\begin{example}[{see \cite[Sections 2.3.2 \& 2.3.3]{BFJR2004} and \cite[Section 3.2]{BLR2008}}]\label{example:control-structures}\
 \begin{enumerate}
  \item Let $Z$ be a $G$-space. The \emph{trivial object support condition} is $\fS_{triv}(Z) = \{ Z \}$.
   The \emph{trivial morphism control condition} is given by $\fC_{triv}(Z) := \{ Z \times Z \}$.
   Together, these form the \emph{trivial coarse structure} $\fT(Z)$.
  \item Let $X$ be a $G$-space. The \emph{$G$-compact support condition} is defined to be
   \begin{equation*}
    \fC_{G\text{-cpt}}(X) := \{ K \subset X \mid K \text{ is $G$-compact.} \}.
   \end{equation*}
  \item Let $M$ be a metric space with isometric $G$-action; metrics are allowed to map to the extended real line $\RR \cup \{ \infty \}$.
   The \emph{bounded morphism control condition} is defined to be
   \begin{equation*}
    \begin{split}
     \fC_{bdd}(M) := \{ B \in \cP(M \times M) \mid &\text{ There is some $R > 0$ such that} \\
      &d(m_1,m_2) \leq R \text{ for all } (m_1,m_2) \in B. \}.
    \end{split}
   \end{equation*}
   Together with the trivial object support condition on $M$, we obtain the \emph{bounded coarse structure}
   \begin{equation*}
    \fB(M) := (M, \fC_{bdd}(M), \fS_{triv}(M)).
   \end{equation*}
  \item Let $X$ be a $G$-space. The \emph{$G$-continuous control condition} $ \fC_{G\text{cc}}(X)$ is the set of all subsets $C \subset (X \times [1,\infty[) \times (X \times [1,\infty[)$ which satisfy the following:
   \begin{enumerate}
    \item For every $x \in X$ and every $G_x$-invariant open neighbourhood $U$ of $(x,\infty)$ in $X \times [1,\infty]$,
     there exists a $G_x$-invariant open neighbourhood $V \subset U$ of $(x,\infty)$ in $X \times [1,\infty]$
     such that $((X \times [1,\infty[) \setminus U) \times V) \cap C = \varnothing$.
    \item Let $p_{[1,\infty[} \colon X \times [1,\infty[ \to [1,\infty[$ be the projection map
     and equip $[1,\infty[$ with the Euclidean metric. Then the set $(p_{[1,\infty[} \times p_{[1,\infty[})(C)$
     is a member of $\fC_{bdd}([1,\infty[)$.
    \item $C$ is symmetric, $G$-invariant and contains the diagonal.
   \end{enumerate}
   \item If $\fC_1$, $\fC_2$ are sets of control conditions on the same space $Z$, then
    \begin{equation*}
     \fC_1 \Cap \fC_2 := \{ C_1 \cap C_2 \mid C_1 \in \fC_1, C_2 \in \fC_2 \}
    \end{equation*}
    is again a set of control conditions. We refer to this construction as ``pointwise intersection''.
  \end{enumerate}
    There are further constructions which allow us to produce new coarse structures out of these, see \cite[2.3.1]{BFJR2004}. We will introduce these on the way as we need them, see for example Definitions \ref{def:pullback-coarse-structure}, \ref{def:delooping-coarse-structure} and \ref{def_continuouscontrol}.
\end{example} 

Let $Y$ be a $G$-CW-complex relative to $W$. The \emph{structural inclusion} of the relative $G$-CW-complex
$(Y,W)$ will usually be denoted by $s \colon W \to Y$. If we speak about a ``cell'' of $Y$, this will always mean
a (non-equivariant) open, relative cell. The closure of a cell $e$ will be denoted by $\overline{e}$, and $\partial e$ will always be
the boundary of the cell $e$, i.e., the image of any attaching map for $e$.
Let $\cells_k Y$ denote the set of $k$-cells of $Y$. Set $\cells Y := \bigcup_k \cells_k Y$.
If $e \in \cells Y$ is a cell in $Y$, we define $\gen{e} \subset Y$
to be the smallest non-equivariant subcomplex of $Y$ (relative $W$) which contains $e$.
For a subgroup $H \leq G$, let $\gen{e}_H \subset Y$ denote the smallest $H$-CW-subcomplex of $Y$ which contains $e$.  Similary, we define $\gen{S}$, $\gen{S}_H$ for any subset $S \subseteq Y$.

A non-equivariant version of the following definition was already considered in \cite{Weiss2002}.

\begin{definition}\label{def_controlledmap}
 Let $\fZ = (Z,\fC,\fS)$ be a coarse structure. Let $Y$ be a $G$-CW-complex relative to $W$.
 A \emph{control map} for $Y$ is an equivariant function $\kappa \colon \cells Y \to Z$.
 
 A $G$-CW-complex $Y$ relative $W$ together with a control map $\kappa$ is called a \emph{labelled $G$-CW-complex relative $W$}.
 
 Let $(Y_1,\kappa_1)$ and $(Y_2,\kappa_2)$ be labelled $G$-CW-complexes relative $W$. A \emph{$\fZ$-controlled map}
 $f \colon Y_1 \to Y_2$ is an equivariant, cellular map (relative $W$) such that for all $k \in \NN$
 there is some $C \in \fC$ for which
 \begin{equation*}
  (\kappa_2 \times \kappa_1)( \{ (e_2,e_1) \mid e_1 \in \cells_k Y_1, e_2 \in \cells Y_2, \gen{f(e_1)} \cap e_2 \neq \varnothing \} ) \subset C
 \end{equation*}
 holds.
 
 A \emph{$\fZ$-controlled $G$-CW-complex relative $W$} is a labelled $G$-CW-complex $(Y,\kappa)$ such that
 the identity map on $Y$ is a $\fZ$-controlled map and for all $k \in \NN$ there is some $S \in \fS$ such that
 \begin{equation*}
  \kappa(\cells_k Y) \subset S.
 \end{equation*}
 We abbreviate the terminology to \emph{controlled map} and \emph{controlled $G$-CW-complex} if the coarse structure $\fZ$ is understood.
\end{definition}

\begin{remark}
 The $\fZ$-control condition for a labelled $G$-CW-complex $(Y,\kappa)$ is a statement about attaching maps.
 Since $\fC$ is closed unter composition and taking finite unions, the control condition in Definition \ref{def_controlledmap} is equivalent to requiring that
 for each $k$, there is some $C_k \in \fC$ such that for every $k$-cell $e$ and every cell $e'$ intersecting
 the closed cell $\overline{e}$ non-trivially, we have $(\kappa(e'),\kappa(e)) \in C_k$.
 
 Moreover, if $C_k$ witnesses $\fZ$-controlledness for a given complex $Y$, we may assume that $C_k \subset C_{k+1}$,
 and the same holds for the support conditions.
 In particular, if $Y$ is finite-dimensional, there are a single support condition $S$ and a single control condition $C$
 witnessing that $Y$ is $\fZ$-controlled.
\end{remark}

Let $(Y,\kappa)$ be a labelled $G$-CW-complex relative $W$. Define the \emph{relative cylinder} $Y \halftimes [0,1]$ by the pushout
\begin{equation*}
 \commsquare{W \times [0,1]}{W}{Y \times [0,1]}{Y \halftimes [0,1]}{}{s \times \id_{[0,1]}}{}{}
\end{equation*}
The projection map $p \colon Y \halftimes [0,1] \to Y$ induces a function $\cells p \colon \cells(Y \halftimes [0,1]) \to \cells Y$,
so $\kappa \circ \cells p$ is a control map for $Y \halftimes [0,1]$.  This turns $Y \halftimes [0,1]$ into a labelled $G$-CW complex relative to $W$.
If $(Y,\kappa)$ is $\fZ$-controlled, then $(Y \halftimes [0,1], \kappa \circ \cells p)$ is also $\fZ$-controlled.

\begin{definition}\label{def_controlledhomotopy}
 Let $(Y_1,\kappa_1)$ and $(Y_2,\kappa_2)$ be labelled $G$-CW-complexes relative $W$.
 A \emph{$\fZ$-controlled homotopy} is a $\fZ$-controlled map $H \colon Y_1 \halftimes [0,1] \to Y_2$.
 Two maps $f_0, f_1 \colon Y_1 \to Y_2$ are \emph{$\fZ$-controlled homotopic}, $f_0 \simeq_\fZ f_1$, if there is a $\fZ$-controlled homotopy
 whose restriction to $Y_1 \times \{ 0 \}$ and $Y_1 \times \{ 1 \}$ equals $f_0$ and $f_1$, respectively.
 
 A $\fZ$-controlled map $f \colon Y_1 \to Y_2$ is a \emph{$\fZ$-controlled homotopy equivalence}
 if there is a $\fZ$-controlled map $\overline{f} \colon Y_2 \to Y_1$ such that $\overline{f}f \simeq_\fZ \id_{Y_1}$
 and $f\overline{f} \simeq_{\fZ} \id_{Y_2}$.
\end{definition}

Suppose $(Y,\kappa)$ is a labelled $G$-CW-complex relative to $W$, and that $B \subset Y$ is a $G$-invariant subcomplex of $Y$ which contains $W$.
Then $B$ naturally becomes a labelled $G$-CW-complex relative to $W$ by restricting $\kappa$ to $B$. If we do not say anything different,
we will always think about subcomplexes as labelled $G$-CW-complexes this way.

\begin{proposition}[{$\fZ$-controlled homotopy extension property \cite[Section 1.A.6]{WW1998}}]\label{prop_controlledhep}
 Let $(Y,\kappa)$ be a $\fZ$-controlled $G$-CW-complex relative $W$, and let $B \subset Y$ be a $G$-invariant subcomplex.
 Let $Y_1$ be a $\fZ$-controlled $G$-CW-complex relative $W$, and suppose that $h \colon Y \times \{ 0 \} \cup B \halftimes [0,1] \to Y_1$
 is a $\fZ$-controlled map.
 
 Then there is a $\fZ$-controlled map $Y \halftimes [0,1] \to Y_1$ extending $h$.
\end{proposition}
\begin{proof}
 The proof follows the usual pattern. Subject to a choice of deformation retraction of $D^n \times [0,1]$ to $D^n \times \{ 0 \} \cup \partial D^n \times [0,1]$,  we can define a $G$-equivariant deformation retraction of $(\skel{Y}{n} \cup B) \halftimes [0,1]$ onto $\skel{Y}{n} \times \{ 0 \} \cup B \halftimes [0,1]$ by composing with the characteristic maps of equivariant $n$-cells. The resulting deformation retraction can be chosen to be constant on all points which do not lie on an $n$-cell of $Y$ which is not in $B$. It is $\fZ$-controlled because points on a given cell are moved at most into the image of the attaching sphere of the same cell (and attaching maps are controlled).
 
 We obtain a deformation retraction of $Y \halftimes [0,1]$ onto $Y \times \{ 0 \} \cup B \halftimes [0,1]$ by stacking the homotopies defined in the first step. This produces another $\fZ$-controlled homotopy since, for each $n$, all but finitely many of the stacked homotopies are constant on the $n$-skeleton. The endpoint of this homotopy is a retraction $r \colon Y \halftimes [0,1] \to Y \times \{ 0 \} \cup B \halftimes [0,1]$, so we may define an extension of $h$ by $H := h \circ r$.
\end{proof}

\section{Categories of controlled retractive spaces}\label{sec_controlledretractivespaces}
The primary objective of the following discussion is to form a Waldhausen category of controlled $G$-CW-complexes relative $W$. This will enable us to study the controlled $A$-theory of $W$ in the sequel. Since the terminology can be considered standard by now, we will freely use the notions of category with cofibrations \cite[page 320]{Waldhausen1985}, Waldhausen category \cite[page 326]{Waldhausen1985} (where it is called ``category with cofibrations and weak equivalences'') and cylinder functor \cite[page 348]{Waldhausen1985}. The Saturation Axiom and Extension Axiom \cite[page 327]{Waldhausen1985} as well as the Cylinder Axiom \cite[page 349]{Waldhausen1985} will play a role.

\begin{definition}\label{def_controlledretractivespaces}
 Let $\fZ = (Z,\fC,\fS)$ be a coarse structure. The \emph{category $\cR^G(W,\fZ)$ of $\fZ$-controlled retractive spaces over $W$}
 is the category whose objects are $\fZ$-controlled, free $G$-CW-complexes $Y$ relative to $W$
 which come equipped with an equivariant retraction $r \colon Y \to W$ (i.e., $r \circ s = \id_W$, where $s$ denotes the structural inclusion $W \to Y$).
 Morphisms in this category are $\fZ$-controlled maps over and under $W$.
\end{definition}

We write $(Y, s_Y, r_Y)$ or $Y \leftrightarrows W$ for objects of $\cR^G(W,\fZ)$, if we want to emphasize the section and retraction or the base space.

\begin{remark}
 We would like to emphasize a few points about Definition \ref{def_controlledretractivespaces} which may be easy to overlook.
 
 By definition, the morphisms in $\cR^G(W,\fZ)$ are all \emph{cellular} maps, and we will never consider maps which are not cellular. This is important for inductive arguments, and also provides us with mapping cylinders.

 Requiring the relative $G$-CW-complexes $(Y,W)$ to have a retraction $Y \rightarrow W$ and morphisms to respect this retraction provides $\cR^G(W, \fZ)$ with a basepoint.
 However, the homotopy equivalences we define later are inherited from the category of relative $G$-CW-complexes. This means that some arguments will lead us to consider maps which do not have to respect the retraction.
 We will use the word \emph{morphism} if a map respects the retractions, and speak about \emph{maps} if the retractions do not need to be preserved.
\end{remark}

\begin{definition}[Finiteness conditions]\label{def_finiteness}
 Let $(Y \leftrightarrows W,\kappa)$ be a $\fZ$-controlled retractive space over $W$.
 
 We call $Y$ \emph{finite} if it is finite-dimensional, the image of $Y \setminus W$ under the retraction meets the orbits of only finitely many path components of $W$, and for all $z \in Z$ there is some open neighbourhood $U$ of $z$ such that $\kappa^{-1}(U)$ is finite.
 
 An object $Y$ is \emph{homotopy finite} if there is a finite object $F$ and a morphism $F \to Y$ which is a $\fZ$-controlled homotopy equivalence.
 
 We call $Y$ \emph{finitely dominated} if there are a finite $\fZ$-controlled $G$-CW-complex $D$ relative $W$, a $\fZ$-controlled morphism $p \colon D \to Y$ and a $\fZ$-controlled map $i \colon Y \to D$ such that $pi \simeq_\fZ \id_Y$ as $\fZ$-controlled maps.
 
 Let us denote the full subcategories of finite/homotopy finite/finitely dominated $\fZ$-controlled retractive spaces by $\cR^G_f(W,\fZ)$, $\cR^G_{hf}(W,\fZ)$ and $\cR^G_{fd}(W,\fZ)$.
\end{definition}

\begin{remark}
 Note that the cells of finite objects can only be labelled with points whose isotropy group is finite. In fact, all control spaces we consider are free.
\end{remark}

\subsection{$\cR^G(W,\fZ)$ as a Waldhausen category}

Observe that $\cR^G(W,\fZ)$ is canonically pointed by the zero object $* = (W \rightleftarrows W, \varnothing)$, and that $*$ is finite. Let $co\cR^G(W,\fZ)$ be the subcategory of all morphisms which are isomorphic to the inclusion of a $G$-invariant subcomplex. We call such morphisms \emph{cofibrations} and denote them by ``$\rightarrowtail$''. Since isomorphisms are controlled, the controlled homotopy extension property (CHEP) holds with respect to cofibrations as a consequence of Proposition \ref{prop_controlledhep}. As observed in \cite[Section 1.A.6]{WW1998}, the CHEP is key to showing that $\cR^G(W,\fZ)$ is a Waldhausen category. In the remainder of this section, we elaborate on this remark, and also introduce a more general notion of weak equivalences in $\cR^G(W,\fZ)$, which is inspired by \cite{CPV1998}.

\begin{lemma}\label{lem_controlledpushouts}
 The subcategory $co\cR^G(W,\fZ)$ is a subcategory of cofibrations for $\cR^G(W,\fZ)$.
 If in a diagram $Y_2 \leftarrow Y_0 \rightarrowtail Y_1$ all three objects are finite, then so is the pushout $Y_1 \cup_{Y_0} Y_2$.
\end{lemma}
\begin{proof}
 The unique morphism $* \to Y$ (given by the structural inclusion) is clearly in $co\cR^G(W,\fZ)$, and the same holds true for any isomorphism.  This shows Waldhausen's first two axioms.  We are left to show that cofibrations admit cobase changes.  Clearly isomorphisms do, so we can restrict ourselves to inclusions of $G$-invariant subcomplexes.
 
 Let a diagram of the form $Y_2 \leftarrow Y_0 \rightarrowtail Y_1$ be given, where $Y_0$ is a subcomplex of $Y_1$.
 The pushout $Y := Y_1 \cup_{Y_0} Y_2$ exists in the category of $G$-CW-complexes
 relative $W$, and the resulting map $Y_2 \to Y$ is the inclusion of a subcomplex.
 By the universal property of the pushout, we obtain a structural retraction $Y \to W$.
 
 We observe that $\cells Y = \cells Y_2 \sqcup (\cells Y_1 \setminus \cells Y_0)$.
 This allows us to define a control map $\kappa \colon \cells Y \to Z$ by setting
 \begin{equation*}
  \kappa(e) := \begin{cases}
                \kappa_2(e) & e \in \cells Y_2, \\
                \kappa_1(e) & e \in \cells Y_1 \setminus Y_0.
               \end{cases}
 \end{equation*}
 Then $(Y,\kappa)$ is an object in $\cR^G(W,\fZ)$ because the map $Y_0 \to Y_2$ is controlled.
 If $Y_1$ and $Y_2$ are finite, then so is $Y$.
\end{proof}

Setting $co\cR^G_f(W,\fZ) := \cR^G_f(W,\fZ) \cap co\cR^G(W,\fZ)$, Lemma \ref{lem_controlledpushouts}
shows that both $\cR^G(W,\fZ)$ and $\cR^G_f(W,\fZ)$ are categories with cofibrations.

The pushout of homotopy finite or finitely dominated objects is also homotopy finite respectively finitely dominated,
and both $\cR^G_{hf}(W,\fZ)$ and $\cR^G_{fd}(W,\fZ)$ are therefore also categories with cofibrations.
However, the proof requires us to know more about the Waldhausen category structure of $\cR^G(W,\fZ)$.
It will be given in Lemma \ref{lem_finitelydominatedpushouts}.

\begin{definition}[Cofinal subcomplexes]\label{def:thickenings:cofinal:subcomplexes}
 Let $A \subset Z$ be a $G$-invariant subspace. A \emph{$\fZ$-thickening} $A$ is a set of the form
 \begin{equation*}
  A^C := \{ z \in Z \mid (z,a) \in C, a \in A \}
 \end{equation*}
 for some $C \in \fC$.
 
 Let additionally $(Y,\kappa)$ be a labelled $G$-CW-complex relative $W$. A subcomplex $Y' \subset Y$ is called \emph{cofinal away from $A$} if for every $k \in \NN$ there is
 some $\fZ$-thickening $A^C$ of $A$ such that $\kappa^{-1}(Z \setminus A^C) \cap \cells_k Y \subset \cells_k Y'$.
\end{definition}

In the following discussion, we tacitly assume the next lemma.

\begin{lemma}\label{lemma:cofinal-subcomplexes}
 Let $Y, Y_1, Y_2$ be labelled $G$-CW-complexes.
 \begin{enumerate}
  \item If $Y' \subset Y$ and $Y'' \subset Y$ are cofinal subcomplexes away from $A$, then so is $Y' \cap Y''$.
  \item If $Y'' \subset Y' \subset Y$ are inclusions of subcomplexes, $Y' \subset Y$ is cofinal away from $A$
   and $Y'' \subset Y'$ is cofinal away from $A$, then $Y'' \subset Y$ is cofinal away from $A$.
  \item Let $f \colon (Y_1,\kappa_1) \to (Y_2,\kappa_2)$ be a controlled map, and let $Y_2' \subset Y_2$
   be cofinal away from $A$. Let $f^*Y_2'$ be the largest subcomplex of $Y_1$ whose image under $f$ is contained in
   $Y_2'$. Then $f^*Y_2' \subset Y_1$ is cofinal away from $A$.
  \item \label{item:subcomplex-cofinal-subcomplex}
   Suppose that $Y$ is $\fZ$-controlled. Let $B \subset Y$ be a subcomplex, $B' \subset B$ a cofinal subcomplex.
   Then there is a cofinal subcomplex $Y' \subset Y$ with $Y' \cap B = B'$.
 \end{enumerate}
\end{lemma}

\begin{definition}[Partially defined maps]\label{def_partiallydefinedmaps}
 Let $(Y_1,\kappa_1)$, $(Y_2,\kappa_2)$ and $(Y_3,\kappa_3)$ be labelled $G$-CW-complexes.
 A \emph{partially defined $\fZ$-controlled map (away from $A$)} $Y_1 \to^A Y_2$ is a pair $(Y_1',f_1)$ where $Y_1' \subset Y_1$
 is cofinal away from $A$ and $f_1 \colon Y_1' \to Y_2$ is a controlled map.
 
 For two partially defined controlled maps $(Y_1',f_1) \colon Y_1 \to^A Y_2$ and $(Y_2',f_2) \colon Y_2 \to^A Y_3$,
 their \emph{composition} $(Y_2',f_2) \circ^A (Y_1',f_1)$ is the partially defined controlled map
 $(f_1^*Y_2',f_2 \circ f_1|_{f_1^*Y_2'}) \colon Y_1 \rightarrow^A Y_3$.
\end{definition}

Composition of partially defined maps is well-defined and associative.

\begin{definition}
 Let $(Y_1',f_0), (Y_1'',f_1) \colon Y_1 \to^A Y_2$ be partially defined controlled maps. Then $(Y_1',f_0)$ and $(Y_1'',f_1)$ are \emph{controlled homotopic away from $A$}, $(Y_1',f_0) \simeq^A (Y_1'',f_1)$, if there is a cofinal subcomplex $Y_1''' \subset Y_1' \cap Y_1''$ and a controlled homotopy $H \colon Y_1''' \halftimes [0,1] \to Y_2$ from $f_0\big|_{Y_1'''}$ to $f_1\big|_{Y_1'''}$.
\end{definition}

\begin{lemma}\label{lemma:restricting-controlled-homotopies}
 Let $H \colon Y_1 \halftimes [0,1] \to Y_2$ be a controlled homotopy, and suppose that $Y_2' \subset Y_2$ is cofinal away from $A \subset Z$.
 
 Then there is a cofinal subcomplex $Y_1' \subset Y_1$ away from $A$ such that $H$ restricts to a controlled homotopy $Y_1' \halftimes [0,1] \to Y_2'$.
\end{lemma}
\begin{proof}
 We construct $Y_1'$ by induction over the skeleta. Assume we have constructed $\skel{Y_1'}{n-1} \subset Y_1$ such that $\skel{Y_1'}{n-1} \subset \skel{Y_1}{n-1}$ is cofinal and such that $H(\skel{Y_1'}{n-1} \halftimes [0,1]) \subset Y_2'$. Define
 \begin{equation*}
  I_n := \{ e \in \cells_n Y_1 \mid \partial e \subset \skel{Y_1'}{n-1}, H( \gen{e} \halftimes [0,1] ) \subset Y_2' \}.
 \end{equation*}
 Then $\skel{Y_1'}{n} := \skel{Y_1'}{n-1} \cup \bigcup_{e \in I_n} \gen{e}$ is a $G$-invariant subcomplex such that
 \begin{equation*}
  H(\skel{Y_1'}{n} \halftimes [0,1]) \subset Y_2'.
 \end{equation*}
 So we only have to show that $\skel{Y_1'}{n} \subset \skel{Y_1}{n}$ is cofinal. There are control conditions $C_1$, $C_1'$, $C$ and $C_2'$ with the following properties:
 \begin{enumerate}
  \item For all $e, e' \in \cells \skel{Y_1}{n}$ with $e' \subset \gen{e}$ we have $(\kappa_1(e'),\kappa_1(e)) \in C_1$.
  \item $\kappa_1^{-1}(Z \setminus A^{C_1'} ) \cap \cells \skel{Y_1}{n-1} \subset \cells \skel{Y_1'}{n-1}$.
  \item For all $e \in \cells_n Y_1$ and all $e' \in \cells \gen{ H(\gen{e} \halftimes [0,1]) }$, we have $(\kappa_2(e'),\kappa_1(e)) \in C$.
  \item $\kappa_2^{-1}(Z \setminus A^{C_2'} ) \cap \cells \skel{Y_2}{n+1} \subset \cells \skel{Y_2'}{n+1}$.
 \end{enumerate}
 Suppose $e \in \cells_n Y_1$ such that $e \notin I_n$. If $\partial e \nsubseteq \skel{Y_1'}{n-1}$, then $\kappa_1(e) \in A^{C_1' \circ C_1}$. If $H( \gen{e} \halftimes [0,1] ) \nsubseteq Y_2'$, then $\kappa_1(e) \in A^{C_2' \circ C}$. Hence, $\kappa_1( \cells_n Y_1 \setminus I_n) \subset A^{C_1' \circ C_1 \cup C_2' \circ C}$. This proves that $\skel{Y_1'}{n} \subset \skel{Y_1}{n}$ is cofinal away from $A$.
 
 Defining $Y_1' := \bigcup_n \skel{Y_1'}{n}$ finishes the proof.
\end{proof}

\begin{lemma}
 Let $(Y_1^0,f_0), (Y_1^1,f_1) \colon Y_1 \to^A Y_2$ be partially defined controlled maps such that $(Y_1^0,f_0) \simeq^A (Y_1^1,f_1)$.
 \begin{enumerate}
  \item For every partially defined controlled map $(Y_0',\alpha) \colon Y_0 \to^A Y_1$ we have $(Y_1^0,f_0) \circ^A (Y_0',\alpha) \simeq^A (Y_1^1,f_1) \circ^A (Y_0',\alpha)$
  \item For every partially defined controlled map $(Y_2',\beta) \colon Y_2 \to^A Y_3$ we have $(Y_2',\beta) \circ^A (Y_1^0,f_0) \simeq ^A (Y_2',\beta) \circ^A (Y_1^1,f_1)$
 \end{enumerate}
\end{lemma}
\begin{proof}
 Consider the second claim. Since $(Y_1^0,f_0) \simeq^A (Y_1^1,f_1)$, there are a cofinal subcomplex $Y_1' \subset Y_1$ and a controlled homotopy $H \colon Y_1' \halftimes [0,1] \to Y_2$ from $f_0|_{Y_1'}$ to $f_1|_{Y_1'}$. Consider $Y_2' \subset Y_2$. By Lemma \ref{lemma:restricting-controlled-homotopies}, there is a cofinal subcomplex $Y_1'' \subset Y_1'$ such that $H$ restricts to a homotopy $H' \colon Y_1'' \halftimes [0,1] \to Y_2'$. Then $\beta \circ H'$ is the desired homotopy. The other claim is similar, but easier.
\end{proof}

\begin{definition}[Homotopy equivalences away from $A$] \label{def_controlledhomotopyequivalences}
 A controlled map $f \colon Y_1 \to Y_2$ between controlled $G$-CW-complexes relative $W$ is a \emph{controlled homotopy equivalence away from $A$} if there is a partially defined controlled map $(Y_2',\overline{f}) \colon Y_2 \to^A Y_1$ such that $f \circ^A (Y_2',\overline{f}) \simeq^A \id_{Y_2}$ and $(Y_2',\overline{f}) \circ^A f \simeq^A \id_{Y_1}$.

 Such maps are called \emph{$h^A$-equivalences}. If $A = \emptyset$, we abbreviate this to \emph{$h$-equivalences}.
 We denote by $h^A\cR^G(W,\fZ)$ the collection of all morphisms in $\cR^G(W,\fZ)$ which are controlled homotopy equivalences away from $A$.
\end{definition}

\begin{remark}
 Note that maps in $h^A\cR^G(W,\fZ)$ are \emph{morphisms}, hence required to respect the retractions, while in general partially defined maps and partially defined homotopy equivalences do not need to respect the retractions.  This means that homotopy inverses of morphisms in $h^A\cR^G(W,\fZ)$ do not need to lie in $h^A\cR^G(W,\fZ)$. Cf.~Section 2.1 of~\cite{Waldhausen1985}, where weak equivalences are defined in a similar way.
  
 The following results will be proven using \emph{maps}.  We obtain results about $h^A\cR^G(W,\fZ)$ because it is the intersection of the $h^A$-equivalences with the morphisms.
 
 Note also that $h$-equivalences are $h^A$-equivalences for any choice of $A$.
\end{remark}

The collection $h^A\cR^G(W,\fZ)$ is closed under composition of morphisms, and identity morphisms are controlled homotopy equivalences away from $A$, hence $h^A\cR_G(W,\fZ)$ is a subcategory of $\cR^G(W,\fZ)$.
Moreover, this subcategory satisfies the Saturation Axiom, i.e., whenever $f_1$ and $f_2$ are composable morphisms, and two out of $f_1$, $f_2$ and $f_2f_1$ are $h^A$-equivalences, so is the third.

We also need to discuss the cylinder functor on $\cR^G_f(W,\fZ)$ before we are ready to continue.
Let $f \colon Y_1 \to Y_2$ be a controlled map of controlled $G$-CW-complexes relative $W$.
Then we define $\cyl(f)$ by the pushout
\begin{equation*}
 \commsquare{Y_1 \times \{1\} = Y_1}{Y_2}{Y_1 \halftimes [0,1]}{\cyl(f)}{f}{}{}{}
\end{equation*}
of $G$-CW-complexes relative $W$. We choose the canonical cofibration $Y_2 \rightarrowtail \cyl(f)$ as the back inclusion of the cylinder, and let $Y_1 = Y_1 \times \{0\} \rightarrowtail Y_1\ \halftimes [0,1] \to \cyl(f)$ be the front inclusion.  The back projection $\cyl(f) \to Y_2$ is induced by $\id_{Y_2}$ and $f$ via the  projection $Y_1 \halftimes [0,1] \rightarrow Y_1$ and the universal property of the pushout.  If $f$ is a morphism in $\cR^G(W,\fZ)$, we can equip $\cyl(f)$ with the induced structural retraction to obtain a retractive space $\cyl_W(f)$.  Then the above diagram becomes a pushout in $\cR^G(W,\fZ)$, and the front inclusion, back inclusion and back projection are morphisms in $\cR^G(W,\fZ)$.

If we use the construction of the pushout given in the proof of Lemma \ref{lem_controlledpushouts} and the fact that $Y_1 \rightarrowtail Y_1 \halftimes [0,1]$ is the inclusion of a subcomplex,
it is clear that this defines a functor from the category of arrows in $\cR^G(W,\fZ)$ to the category of diagrams of the shape
\begin{equation}\label{eq_cylinderdiagram}
 \begin{tikzpicture}
  \matrix (m) [matrix of math nodes, column sep=2em, row sep=2em, text depth=.5em, text height=1em]
  {Y_1 & \cyl_W(f) & Y_2 \\ & Y_2 & \\};
  \path[->]
   (m-1-1) edge node[below left]{$f$} (m-2-2)
   (m-1-2) edge (m-2-2)
   (m-1-3) edge node[below right]{$=$}(m-2-2);
  \path[>->]
   (m-1-1) edge (m-1-2)
   (m-1-3) edge (m-1-2);
 \end{tikzpicture}
\end{equation}
in $\cR^G(W,\fZ)$. Observe also, that the back projection is a controlled homotopy equivalence: The usual deformation retraction of $\cyl(f)$ onto $Y_2$ is a controlled homotopy. We can choose $\cyl_W(* \rightarrow A) = A$, which is needed for the following lemma.

\begin{lemma}\label{lemma:cylinder-functor}
  $\cyl_W(-)$ gives a \emph{cylinder functor} on $\cR^G(W,\fZ)$ which satisfies the Cylinder Axiom with respect to $h$-equivalences.
\end{lemma}

We are heading towards the following Proposition.

\begin{proposition}[Gluing Lemma] \label{prop:gluing-lemma}
  Assume we have the following commutative diagram in $\cR^G(W,\fZ)$:
  \begin{equation}\label{diag:gluing-lemma-pushout}
  \begin{tikzpicture}
   \matrix (m) [matrix of math nodes, column sep=2em, row sep=2em, text depth=.5em, text height=1em]
    {X_2 & X_0 & X_1 \\ Y_2 & Y_0 & Y_1 \\};
   \path[>->]
     (m-1-2) edge node[above] {$x_1$} (m-1-3)
     (m-2-2) edge node[above] {$y_1$} (m-2-3);
   \path[->] 
    (m-1-2) edge node[above] {$x_2$} (m-1-1)
    (m-2-2) edge node[above] {$y_2$} (m-2-1)
    (m-1-1) edge node[left]{$\sim^A$} node[right]{$f_2$} (m-2-1)
    (m-1-2) edge node[left]{$\sim^A$} node[right]{$f_0$} (m-2-2)
    (m-1-3) edge node[left]{$\sim^A$} node[right]{$f_1$} (m-2-3);
   \end{tikzpicture}
 \end{equation}
 Assume $x_1$, $y_1$ are cofibrations and the $f_i$ are $h^A$-equivalences.

 Then the induced map on the pushouts $f \colon X_2 \cup_{X_0} X_1 \rightarrow Y_2 \cup_{Y_0} Y_1$ is an $h^A$-equivalence.
\end{proposition}

The uncontrolled version of Proposition \ref{prop:gluing-lemma} is well-known. Our proof will follow the strategy pursued in \cite[pages 33--59]{Kamps-Porter;abstract-homotopy}, who give a detailed argument which relies only on the homotopy extension property.

For the purpose of the proof, we introduce the following notation: If $f$ and $g$ are partially defined maps $X \to^A Y$ whose restrictions to some cofinal subcomplex of $X$ are equal, we write $f =^A g$.

For the proof of the Gluing Lemma, we need the following auxiliary results.
 
\begin{lemma} \label{lemma:left-ha-inverse}
  Let $j_i \colon B \rightarrowtail Y_i$, $i=1,2$ be cofibrations.  Let $f \colon Y_1 \rightarrow Y_2$ be an $h^A$-equivalence which satisfies $f j_1 = j_2$. 
  
  Then there is a partially defined map $(Y_2',g) \colon Y_2 \rightarrow^A Y_1$ with $(Y_2',g) \circ^A j_2 =^A j_1$ and a homotopy $H \colon (Y_2',g) \circ^A f \simeq^A \id_{Y_1}$ away from $A$ with $H \circ^A (j_1 \times [0,1]) =^A j_1 \times [0,1]$.

  Furthermore, $f \circ^A (Y_2',g)$ is also homotopic to the inclusion via a homotopy under $B$, i.e., $f$ is an ``$h^A$-equivalence under $B$''.
\end{lemma}
\begin{proof}
 This is very similar to the standard proofs in the uncontrolled case (e.g.~\cite[\S 6.5]{May1999}). In our situation, one needs to take into account that maps and homotopies are only defined on cofinal subcomplexes.
\end{proof}

\begin{lemma}[left inverses for $h^A$-equivalences, relative case] \label{lemma:hA-equivalence-of-pairs}
  Assume we have the following diagram:
  \begin{equation}\label{diag:square-cofibrations}
    \begin{tikzpicture}
      \matrix (m) [matrix of math nodes, column sep=2em, row sep=2em, text depth=.5em, text height=1em]
      {B_1 & B_2\\
      Y_1  & Y_2 \\};
      \path[>->]
      (m-1-1) edge node[left] {$i_1$} (m-2-1)
      (m-1-2) edge node[left] {$i_2$} (m-2-2);
      \path[->]
      (m-1-1) edge node[above] {$b$} (m-1-2)
      (m-2-1) edge node[above]{$f$} (m-2-2);
    \end{tikzpicture}
  \end{equation}
  Assume that $b$ is an $h^A$-equivalence with inverse $(B_2',b') \colon B_2 \rightarrow^A B_1$ and homotopy $H_{B_1}\colon (B_2',b') \circ^A b \simeq^A \id_{B_1}$.  (We do not need to specify the other homotopy.)

  If $f$ is an $h^A$-equivalence, then there is a partially defined map $(Y_2',f') \colon Y_2 \rightarrow^A Y_1$ and a homotopy $H_{Y_1} \colon (Y_2',f') \circ^A f \simeq^A \id_{Y_1}$ such that $i_1 \circ^A (B_2',b') =^A (Y_2',f') \circ^A i_2$ and $H_{Y_1} \circ^A (i_1 \times [0,1]) =^A i_1 \circ^A H_{B_1}$. (``$f$ has a left $h^A$-inverse relative to $B_i$.'')
\end{lemma}

\begin{proof}
  Cf.~\cite[I.7.3]{Kamps-Porter;abstract-homotopy}. For the purpose of this proof, we omit the domains of partially defined maps from the notation.
  Let $g$ be an $h^A$-inverse for $f$.  The map $g \circ^A i_2$ is homotopic away from $A$ to $g \circ^A i_2 \circ^A b \circ^A b'$ and hence to $i_1 \circ^A b'$.  As $i_2$ is a cofibration, $g$ is homotopic to a map $g'\colon Y_2 \rightarrow^A Y_1$ such that $g' \circ^A i_2 =^A i_1\circ^A b'$.  Now $i_1 \circ^A H_{B_1}$ is a homotopy away from $A$ from $i_1 \circ^A b' \circ^A b =^A g' \circ^A f \circ^A i_1$ to $i_1$.  As $i_1$ is a cofibration, homotopy extension gives a homotopy $K$, extending $H_{B_1}$, from $g' \circ^A f$ to a map $l$.

  Then $l \circ^A i_1 =^A i_1$, hence Lemma~\ref{lemma:left-ha-inverse} provides a left $h^A$-inverse $l'$ of $l$ under $B_1$.  Define $f' := l' \circ^A g'$.  Then, as a composition of $h^A$-equivalences, $f'$ is itself an $h^A$-equivalence, and $f' \circ^A i_2 =^A i_1$.

  We have homotopies $f' \circ^A f =^A l'\circ^A g' \circ^A f \simeq^A_{K} l' \circ^A l \simeq^A \id$. Restricting along $i_1$, this is the concatenation of the homotopy $H_{B_1}$ and the constant homotopy.  There is a cofinal subcomplex $B_1' \subset B_1$ such that we get a map $B_1' \halftimes [0,1] \halftimes [0,1] \rightarrow^A Y_1$ by projecting to the first two factors and then applying $H_{B_1}$.  The homotopies above extend this to a map $Y_1' \halftimes [0,1] \halftimes 0 \cup Y_1 \halftimes \{0,1\} \halftimes [0,1] \rightarrow^A Y_1$, defined on some cofinal subcomplex $Y_1' \subset Y_1$.  We may assume that $B_1' \subset Y_1'$.  The CHEP~\ref{prop_controlledhep} then gives the homotopy $H_{Y_1}$.
\end{proof}

\begin{remark}
  We cannot make special assumptions about the cofinal subcomplex on which $f'$ is defined.  In particular, it could happen that $Y_2' \cap B_2 \not\subseteq B_2'$.  We need to take care of this situation in the proof of Lemma~\ref{lemma:gluing-lemma-special-case} below.
\end{remark}

\begin{lemma}[Gluing Lemma, special case]\label{lemma:gluing-lemma-special-case}
  Assume in~\eqref{diag:gluing-lemma-pushout} additionally that $x_2$ and $y_2$ are cofibrations.  Then the conclusion of the Proposition holds, i.e., the map $f$ on the pushout is an $h^A$-equivalence.
\end{lemma}

\begin{proof}
  We can assume $x_i$, $y_i$ are cellular inclusions, because they are so up to isomorphism.

  Pick an $h^A$-inverse $(Y_0',g_0)$ of $f_0$ and a homotopy $H_0 \colon (Y_0',g_0) \circ^A f_0 \simeq^A \id_{X_0}$. By Lemma~\ref{lemma:hA-equivalence-of-pairs}, there are $h^A$-left inverses $(Y_i',g_i)$ of $f_i$ and homotopies $H_i \colon (Y_i',g_i) \circ^A f_i \simeq^A \id_{X_i}$, such that $(Y_i',g_i) \circ^A y_i =^A x_i \circ^A (Y_0',g_0)$ and $H_i \circ^A (x_i \times [0,1]) =^A y_i \circ^A H_0$, $i=1,2$.

  Choose a cofinal subcomplex $Y_0''$ of $Y_0$ such that the following diagram commutes:
  \begin{equation*}
  \begin{tikzpicture}
   \matrix (m) [matrix of math nodes, column sep=2em, row sep=2em, text depth=.5em, text height=1em]
    { Y'_2 & Y''_0 & Y'_1 \\ X_2 & X_0 & X_1 \\};
   \path[>->]
     (m-1-2) edge node[above] {$y_1$} (m-1-3)
     (m-2-2) edge node[above] {$x_1$} (m-2-3)
    (m-1-2) edge node[above] {$y_2$} (m-1-1)
    (m-2-2) edge node[above] {$x_2$} (m-2-1);
   \path[->] 
    (m-1-1) edge node[left]{$\sim^A$} node[right]{$g_2$} (m-2-1)
    (m-1-2) edge node[left]{$\sim^A$} node[right]{$g_0$} (m-2-2)
    (m-1-3) edge node[left]{$\sim^A$} node[right]{$g_1$} (m-2-3);
   \end{tikzpicture}
 \end{equation*}
 However, $Y_1' \cup_{Y_0''} Y_2'$ does not need to be a subcomplex of $Y_1 \cup_{Y_0} Y_2$, as the $Y_0''$ provided by Lemma~\ref{lemma:hA-equivalence-of-pairs} could be too small.  But by part~\eqref{item:subcomplex-cofinal-subcomplex} of Lemma~\ref{lemma:cofinal-subcomplexes} we can restrict further to cofinal subcomplexes $Y_i''$, $i=1,2$, such that $Y_i'' \cap Y_0 = Y_0''$.  Then $Y'' := Y_1'' \cup_{Y_0''} Y_2''$ is canonically isomorphic to the cofinal subcomplex $Y_2''\cup Y_1''$ of $Y_1 \cup_{Y_0} Y_2$.  Thus we get a partially defined map $(Y'', g) \colon Y_1 \cup_{Y_0} Y_2 \rightarrow^A X_1 \cup_{X_0} X_2$.  
 
 By the same argument, we get a partially defined homotopy from $(Y'',g) \circ^A f$ to the inclusion.

 Repeating the argument with $g_i$ instead of $f_i$, we get a partially defined map $l \colon X_1 \cup_{X_0} X_2 \rightarrow^A Y_1 \cup_{Y_0} Y_2$ with $l \circ^A g \simeq^A \id_{Y_2 \cup_{Y_0} Y_1}$.  It follows that $f \circ^A g \simeq^A \id_{Y_2 \cup_{Y_0} Y_1}$, hence $f$ is an $h^A$-equivalence.
\end{proof}

\begin{lemma}
  \label{lemma:pushout-along-cofibration}
  Assume that~\eqref{diag:square-cofibrations} is a pushout square and that $b$ is an $h$-equivalence.  Then $f$ is an $h$-equivalence.
\end{lemma}

\begin{proof}
  We can factor $b$ into $B_1 \rightarrowtail \cyl (b) \rightarrow B_2$, and by Saturation both maps are $h$-equivalences.  Taking the pushout along the first map, we obtain the following diagram:
  \begin{equation*}
    \begin{tikzpicture}
      \matrix (m) [matrix of math nodes, column sep=2em, row sep=2em, text depth=.5em, text height=1em]
      {B_1 & \cyl(b) & B_2\\
      Y_1  & M_{b,i_1} & Y_2 \\};
      \path[>->]
      (m-1-1) edge node[left] {$i_1$} (m-2-1)
      (m-1-2) edge node[left] {$i'$}  (m-2-2)
      (m-1-3) edge node[left] {$i_2$} (m-2-3)
      (m-1-1) edge node[above] {$b'$} (m-1-2)
      (m-2-1) edge node[above] {$j$} (m-2-2)
      ;
      \path[->]
      (m-1-2) edge node[above] {$$} (m-1-3)
      (m-2-1) edge[bend right] node[below]{$f$} (m-2-3)
      (m-2-2) edge node[above]{$p$} (m-2-3);
    \end{tikzpicture}
  \end{equation*}
  Here $M_{b,i_1}$ is the double mapping cylinder.  One now shows that $j$ is an $h$-equivalence using that  $b'$ is a cofibration and an $h$-equivalence, and that $p$ is an $h$-equivalence because $i_1$ is a cofibration.  The usual proofs of these facts apply almost verbatim.  We refer to~\cite[Proposition I.7.4]{Kamps-Porter;abstract-homotopy} for the details.
\end{proof}

\begin{proof}[Proof of Proposition~\ref{prop:gluing-lemma}]
  Cf.~also~\cite[7.1]{Kamps-Porter;abstract-homotopy}.
  Using the mapping cylinder we can factor the diagram~\eqref{diag:gluing-lemma-pushout} as follows:
  \begin{equation*}
  \begin{tikzpicture}
   \matrix (m) [matrix of math nodes, column sep=3em, row sep=3em, text depth=.5em, text height=1em]
    {X_2 & X' & X_0 & X_1 \\ 
     Y_2 & Y' & Y_0 & Y_1 \\};
   \path[>->]
     (m-1-3) edge node[above] {$x_1$} (m-1-4)
     (m-2-3) edge node[above] {$y_1$} (m-2-4)
     (m-1-3) edge node[below] {$x_2'$} (m-1-2)
     (m-2-3) edge node[above] {$y_2'$} (m-2-2);
   \path[->] 
    (m-1-3) edge[bend right] node[above] {$x_2$} (m-1-1)
    (m-2-3) edge[bend left]  node[below] {$y_2$} (m-2-1)
    (m-1-2) edge node[above] {$\sim$} node[below] {$x_3$} (m-1-1)
    (m-2-2) edge node[below] {$\sim$} node[above] {$y_3$} (m-2-1)
    (m-1-1) edge node[left]{$\sim^A$} node[right]{$f_2$} (m-2-1)
    (m-1-2) edge node[left]{$\sim^A$} node[right]{$f'$}  (m-2-2)
    (m-1-3) edge node[left]{$\sim^A$} node[right]{$f_0$} (m-2-3)
    (m-1-4) edge node[left]{$\sim^A$} node[right]{$f_1$} (m-2-4);
   \end{tikzpicture}
 \end{equation*}
 The maps $x_3$ and $y_3$ are $h$-equivalences by Lemma \ref{lemma:cylinder-functor}, so $f'$ is an $h^A$-equivalence by Saturation.  The right part of the diagram consisting of $x'_2, y'_2, x_1, y_1$ satisfies the assumptions of Lemma~\ref{lemma:gluing-lemma-special-case}.  Therefore, the induced map $f'' \colon X' \cup_{X_0} X_1 \rightarrow Y' \cup_{Y_0} Y_1$ is an $h^A$-equivalence. Abbreviate $X'' := X' \cup_{X_0} X_1$, $Y'' := Y' \cup_{Y_0} Y_1$. We get induced cofibrations $x_4 \colon X' \rightarrowtail X''$, $y_4 \colon Y' \rightarrowtail Y''$.

 We obtain the cube
  \begin{equation*}
  \begin{tikzpicture}
   \matrix (m) [matrix of math nodes, column sep=2em, row sep=2em, text depth=.5em, text height=1em]
    {    &  X' && X'' \\
     X_2 && X_2 \cup_{X_0} X_1 & \\
         &  Y' && Y'' \\
     Y_2 && Y_2 \cup_{Y_0} Y_1 & \\};
      \path[>->]
     (m-1-2) edge node[above] {$x_4$} (m-1-4)
     (m-3-2) edge node[above] {$$} (m-3-4);
   \path[->] 
    (m-1-2) edge node[pos=.6,below right]{$f'$}  (m-3-2)
    (m-1-4) edge node[left]{$\sim^A$} node[right]{$f''$} (m-3-4);
      \path[->] 
    (m-1-2) edge node[below right] {$x_3$} node[above left] {$\sim$} (m-2-1)
    (m-3-2) edge node[below right] {$y_3$} node[above left] {$\sim$} (m-4-1)
    (m-1-4) edge node[below right] {$x_5$} node[above left] {} (m-2-3)
    (m-3-4) edge node[below right] {$y_5$} node[above left] {} (m-4-3);
      \path[>->]
     (m-2-1) edge[-,line width=6pt,draw=white] (m-2-3)
     (m-2-1) edge node[pos=0.6,above right] {$$} (m-2-3)
     (m-4-1) edge node[pos=0.6,above] {$y_4$} (m-4-3);
   \path[->]
    (m-2-1) edge node[left]{$$} node[right]{$f_2$} (m-4-1)
    (m-2-3) edge[-,line width=6pt,draw=white] (m-4-3)
    (m-2-3) edge node[pos=.3,left]{$$} node[pos=.3,right]{$f$} (m-4-3);
   \end{tikzpicture}
 \end{equation*}
 where the top and bottom are pushout squares.  By Lemma~\ref{lemma:pushout-along-cofibration}, the maps $x_5$, $y_5$ are $h$-equivalences.  By Saturation, $f$ is an $h^A$-equivalence, which proves the proposition.
\end{proof}

\begin{lemma}\label{lem_finitelydominatedpushouts}
 Let $Y_2 \leftarrow Y_0 \rightarrowtail Y_1$ be a diagram of homotopy finite or finitely dominated objects.
 Then the pushout $Y_1 \cup_{Y_0} Y_2$ in $\cR^G(W,\fZ)$ is also homotopy finite/finitely dominated.
\end{lemma}
\begin{proof}
  For homotopy finite objects, this is a formal consequence of the Gluing Lemma~\ref{prop:gluing-lemma} for $h$-equivalences.
 For the second claim, we show that the following two statements are equivalent:
 \begin{enumerate}
  \item $Y \in \cR^G(W,\fZ)$ is finitely dominated. 
  \item $Y \in \cR^G(W,\fZ)$ is a retract of a homotopy finite object.
 \end{enumerate}
 Suppose $(Y,s_Y, r_Y)$ is finitely dominated, i.e., there are a finite object  $(D, s_D, r_D)$, a morphism $p \colon D \to Y$, a controlled map $i  \colon Y \to D$, and a homotopy $h \colon pi \simeq \id_Y$.
 These data give a map $f \colon \cyl(i) \to Y$ whose composition with the front inclusion is $\id_Y$ and whose composition with the back inclusion is $p$.
 Then $r := r_Y \circ f$ is a retraction which makes $\cyl(i)$ into a retractive space over $W$ and both $f$ and the front inclusion $Y \rightarrowtail \cyl(i)$ are morphisms. By construction, $Y$ is a retract of $\cyl(i)$. The back inclusion $D \rightarrowtail \cyl(i)$ is an $h$-equivalence, hence $\cyl(i)$ is a homotopy finite object.

 Conversely, assume that there is a homotopy finite object $F$ as well as morphisms $s \colon Y \to F$
 and $q \colon F \to Y$ such that $qs = \id_Y$. Since $F$ is homotopy finite, there is a finite object
 $D$ and a morphism $e \colon D \to F$ which is an $h$-equivalence. Let $\overline{e} \colon F \to D$
 be an inverse controlled map. Then $i := \overline{e}s \colon Y \to D$ is a controlled map
 to a finite object, and $p := qe \colon D \to Y$ is a morphism. Moreover, we have $pi = qe\overline{e}s \simeq qs = \id_Y$
 by assumption, so $Y$ is finitely dominated.
 
 With the characterization of finitely dominated objects as retracts of homotopy finite objects at our disposal,
 it is a formal consequence of the first part of the lemma and the universal property of the pushout
 that pushouts of finitely dominated objects are finitely dominated.
 
 Sections 7.3 and 7.4 of~\cite{arxiv:ullmann;controlled} spell out the formal
 arguments we left out here.
\end{proof}

\begin{corollary}
 For any $G$-invariant subset $A \subset Z$, the categories $\cR^G(W,\fZ)$, $\cR^G_f(W,\fZ)$, $\cR^G_{hf}(W,\fZ)$ and $\cR^G_{fd}(W,\fZ)$ are
 Waldhausen categories with respect to $h^A\cR^G(W,\fZ)$.  The Saturation Axiom holds for these categories.

 There is a cylinder functor on $\cR^G(W,\fZ)$ which restricts to a cylinder functor
 on the subcategories of finite, homotopy finite and finitely dominated objects;
 the $h^A$-equivalences satisfy the Cylinder Axiom.
\end{corollary}
\begin{proof}
 We only need to summarize what we already know. Lemmas \ref{lem_controlledpushouts} and \ref{lem_finitelydominatedpushouts} state
 that the cofibrations indeed form a subcategory of cofibrations. The collection of $h^A$-equivalences defines a subcategory of
 weak equivalences by the Gluing Lemma \ref{prop:gluing-lemma}. Saturation and the cylinder functor have been discussed right before
 the statement of the Gluing Lemma. Since every $h$-equivalence is an $h^A$-equivalence, the Cylinder Axiom is obvious.
\end{proof}

\subsection{Functoriality}

Let us turn to the question in which sense the categories $\cR^G(W,\fZ)$ are functorial with respect to the space $W$ and the coarse structure $\fZ$.  Changing $G$ will be discussed in Section~\ref{sec_assembly}.

If $f \colon W_1 \to W_2$ is a $G$-equivariant (continuous) map, pushout along $f$ and the structural inclusion of a given object defines an exact functor
\begin{equation*}
 \cR^G(f,\fZ) \colon \cR^G(W_1,\fZ) \to \cR^G(W_2,\fZ).
\end{equation*}

For changing the coarse structure, we need to define a notion of morphism (cf.~\cite[Section 3.3]{BFJR2004}).

\begin{definition}\label{def:morphism-coarse-structure}
 Let $\fZ_1 = (Z_1, \fC_1, \fS_1)$, $\fZ_2 = (Z_2, \fC_2, \fS_2)$ be two coarse structures.
 A \emph{morphism of coarse structures} $z \colon \fZ_1 \to \fZ_2$ is a $G$-equivariant map of sets
 $z \colon Z_1 \to Z_2$ satisfying the following properties:
 \begin{enumerate}
  \item\label{def:morphism-coarse-structure:1} For every $S_1 \in \fS_1$, there is some $S_2 \in \fS_2$ such that $z(S_1) \subset S_2$.
  \item\label{def:morphism-coarse-structure:2} For every $S \in \fS_1$ and $C_1 \in \fC_1$, there is some $C_2 \in \fC_2$ such that $(z \times z)((S \times S) \cap C_1) \subset C_2$.
  \item\label{def:morphism-coarse-structure:3} For every $S \in \fS_1$ and all subsets $A \subset S$ which are locally finite in $Z_1$, the set $z(A)$ is locally finite in $Z_2$ and for all $x \in z(A)$, the set $z^{-1}(x) \cap A$ is finite.
 \end{enumerate}
\end{definition}

Note that $z$ does not need to be continuous, but the topology of $Z_1$ and $Z_2$ is used in the third condition. Morphisms of coarse spaces induce morphisms of controlled categories:

\begin{proposition}\label{prop:functoriality:Z}
 The categories $\cR^G(W,\fZ)$, $\cR^G_f(W,\fZ)$, $\cR^G_{hf}(W,\fZ)$ and $\cR^G_{fd}(W,\fZ)$ are functorial in $\fZ$, i.e., they define functors
 from the category of coarse structures and their morphisms to the category of Waldhausen categories.
 
 The canonical inclusion functors yield natural transformations
 \begin{equation*}
  \cR^G_f(W,-) \to \cR^G_{hf}(W,-) \to \cR^G_{fd}(W,-) \to \cR^G(W,-).
 \end{equation*}
\end{proposition}

See also Remark~\ref{rem_settheoreticcaveat} for some set-theoretical issues.

\begin{proof}
 Let $z \colon \fZ_1 \to \fZ_2$ be a morphism of coarse structures. Define the induced functor
 \begin{equation*}
  \cR^G(W,z) \colon \cR^G(W,\fZ_1) \to \cR^G(W,\fZ_2)
 \end{equation*}
 by mapping an object $(Y, \kappa)$ to $(Y,z \circ \kappa)$ and by the identity on morphisms.  We only have to show that this is well-defined. Let $(Y,\kappa) \in \cR^G(W,\fZ_1)$.  For every $k \in \NN$, there is some $S_1 \in \fS_1$ such that $\kappa(\cells_k Y) \subset S_1$.  Since $z$ is a morphism of coarse structures, we can find some $S_2 \in \fS_2$ such that $z(\kappa(\cells_k Y)) \subset z(S_1) \subset S_2$. The verification that controlled maps are sent to controlled maps is similar. Condition \ref{def:morphism-coarse-structure:3} of Definition \ref{def:morphism-coarse-structure} ensures that this construction preserves finiteness. Hence, homotopy finite and finitely dominated objects are also preserved.
\end{proof}

\begin{example}\label{def_restrictedcoarsestructure}
 One of the most frequent examples of a morphism of coarse structures is the following.
 Let $\fZ = (Z,\fC,\fS)$ be a coarse structure, and suppose that $A \subset Z$ is a $G$-invariant subspace.
 Denote by $\fZ \cap A$ the coarse structure $(A,\fC \Cap \{A \times A\}, \fS \Cap \{ A \})$, where $\Cap$ denotes pointwise intersection.
 
 If $A$ is closed in $Z$, the inclusion map of $A$ into $Z$ defines a morphism $\fZ \cap A \to \fZ$ of coarse structures. Here closedness is required to preserve local finiteness of subsets.
\end{example}

\begin{remark}[Set-theoretical smallness requirements] \label{rem_settheoreticcaveat}
 In the following, we will discuss the algebraic $K$-theory of the categories $\cR^G_f(W,\fZ)$ and $\cR^G_{fd}(W,\fZ)$.  As always, one faces certain set-theoretic difficulties in making sense of the $K$-theory of these categories, cf.\ \cite[Remark on page 379]{Waldhausen1985}. Possible solutions include the use of a change-of-universe functor to make the categories at hand small, or to choose small models to replace these categories.  For example, we may redefine $\cR^G(W,\fZ)$ so that the underlying set of every retractive space is a subset of $W \times \lambda$, where $\lambda$ is a sufficiently large cardinal.
  
 The algebraic $K$-theory of $\cR^G_f(W,\fZ)$ does not depend, up to homotopy, on the set-theoretic model we choose, as long as $\lambda$ is large enough compared to $\fZ$. Proposition \ref{prop:functoriality:Z} then only asserts functoriality on some small, but arbitrarily large subcategory of the category of all coarse structures. To avoid further complications, we ignore these matters from now on.
\end{remark}

\section{Comparison theorems and vanishing theorems}\label{sec_comparisonthms}
In addition to the notions used in the previous section, we now have the opportunity to use all three fundamental results of Waldhausen $K$-theory: the Additivity Theorem \cite[Theorem 1.4.2]{Waldhausen1985}, the Fibration Theorem \cite[Theorem 1.6.4]{Waldhausen1985} and the Approximation Theorem \cite[Theorem 1.6.7]{Waldhausen1985}.

\subsection{Comparing finiteness conditions}
We discuss to which extent the $K$-theory spaces arising from the various finiteness conditions differ.
The answer is given in Proposition \ref{prop_cofinality}, but the proof requires two preparatory lemmas.

\begin{lemma}[Mapping cylinder argument]\label{lem_mappingcylinderargument}
  Let $f\colon Y \rightarrow Y'$ and $g \colon Y'' \rightarrow Y'$ be morphisms in $\cR^G(W, \fZ)$. Suppose that $g$ is a retraction up to homotopy, i.e., that there exists a map $\overline{g} \colon Y' \to Y''$ such that $g \overline{g}$ is controlled homotopic to the identity map. Then there is an object $Q$ in $\cR^G(W, \fZ)$ which fits into the following commutative diagram in $\cR^G(W, \fZ)$
 \begin{equation}\label{diag:cylinder-triangle}
  \begin{tikzpicture}
   \matrix (m) [matrix of math nodes, column sep=3em, row sep=3em, text depth=.5em, text height=1em]
   {Y & Y' \\
   Q  & Y'' \\};
   \path[>->]
    (m-1-1) edge node[left] {$i_{Y}$} (m-2-1)
    (m-2-2) edge node[pos=0.4, above] {$i_{Y''}$} node[below] {$\sim$} (m-2-1);
   \path[->]
    (m-1-1) edge node[above]{$f$} (m-1-2)
    (m-2-2) edge node[right]{$g$} (m-1-2)
    (m-2-1) edge node[above left]{$q$} (m-1-2);
  \end{tikzpicture}
 \end{equation}
 in which $i_Y$ and $i_{Y''}$ are cofibrations. The underlying controlled $G$-CW-complex of $Q$ can be chosen to be $\cyl(\overline{g}f)$.
  
 In particular, $q$ is an $h$-equivalence if and only if $g$ is one.
\end{lemma}
\begin{proof}
  Denote the retractions of $Y, Y', Y''$ by $r, r', r''$.
  Note that $\overline{g}$ does not need to respect the retraction. Define $Q := \cyl(\overline{g}f)$, and let $i_Y$ and $i_{Y''}$ be the front and back inclusion. Since $g \overline{g} f \simeq f$, any choice of homotopy $g \overline{g} \simeq \id_{Y''}$ induces a map $q \colon Q \to Y'$ which restricts to $f$ and $g$ on the front and back of the cylinder. We can turn $q$ into a morphism of retractive spaces by defining a retraction on $Q$ via $r_Q := r' \circ q$. Since $q$ restricts to $f$ and $g$ on the two ends of the cylinder and both of these maps are morphisms of retractive spaces, $i_Y$ and $i_{Y''}$ also respect the retractions. This proves the existence of the commutative diagram (\ref{diag:cylinder-triangle}).  
\end{proof}

The following lemma reflects the fact that something close to a Puppe sequence exists in any Waldhausen category $\cC$ with a cylinder functor. Even though the Extension Axiom does not hold in $\cR^G(W,\fZ)$, cf.~\cite[1.2]{Waldhausen1985}, it follows that the axiom does hold up to suspension. Recall that the suspension of an object $A \in \cC$ is defined to be
\begin{equation*}
 \Sigma A := \cyl(A \to *)/A,
\end{equation*}
and that this extends to an exact endofunctor on $\cC$ \cite[page 349]{Waldhausen1985}.

\begin{lemma}\label{lem_puppesequence}
 Let $\cC$ be a Waldhausen category which possesses a cylinder functor such that the Cylinder Axiom and the Saturation Axiom hold.
 Consider a morphism between exact sequences
 \begin{equation}
   \begin{tikzpicture}
   \matrix (m) [matrix of math nodes, column sep=2em, row sep=2em, text depth=.5em, text height=1em]
     {A & B & C \\ A' & B' & C' \\};
   \path[>->]
     (m-1-1) edge node[above]{$\alpha$} (m-1-2)
     (m-2-1) edge node[above]{$\alpha'$} (m-2-2);
   \path[->>] 
     (m-1-2) edge node[above]{$\beta$} (m-1-3)
     (m-2-2) edge node[above]{$\beta'$} (m-2-3);
   \path[->]
     (m-1-1) edge node[left]{$a$} node[right]{$\sim$} (m-2-1)
     (m-1-2) edge node[left]{$b$} (m-2-2)
     (m-1-3) edge node[left]{$c$} node[right]{$\sim$} (m-2-3);
   \end{tikzpicture}
 \end{equation}
 in which $a$ and $c$ are weak equivalences. Then $\Sigma b$ is a weak equivalence.
\end{lemma}
\begin{proof}
 Repeated use of the cylinder functor gives rise to the commutative diagram on page \pageref{diag:puppe}.
 \begin{figure}[t]
   \begin{tikzpicture}
   \matrix (m) [matrix of math nodes, column sep=1.5em, row sep=1em, text depth=.5em, text height=1em]
   {  & A' &              & B'            &           & C'         &      &       &   &    \\
    A &    & B            &               & C         &            &      &       &   &    \\
      & A' &              & \cyl(\alpha') &           & C(\alpha') &      &       &   &    \\
    A &    & \cyl(\alpha) &               & C(\alpha) &            &      &       &   &    \\
      &    &              & B'            &           & C(\alpha') &      & S'    &   &    \\
      &    & B            &               & C(\alpha) &            & S    &       &   &    \\
      &    &              & B'            &           & \cyl(j')   &      & C(j') &   &    \\
      &    & B            &               & \cyl(j)   &            & C(j) &       &   &    \\
      &    &              &               &           & C(\alpha') &      & C(j') &   & T' \\
      &    &              &               & C(\alpha) &            & C(j) &       & T &    \\};
      \path[->]
    (m-3-4) edge (m-1-4)
    (m-3-6) edge (m-1-6)
    (m-7-6) edge (m-5-6)
    (m-7-8) edge (m-5-8);
   \path
    (m-3-2) edge[double, double distance=1.5pt] (m-1-2)
    (m-5-6) edge[double, double distance=1.5pt] (m-3-6)
     (m-7-4) edge[double, double distance=1.5pt] (m-5-4)
     (m-9-8) edge[double, double distance=1.5pt] (m-7-8);
   \path[>->]
    (m-1-2) edge node[above]{$\alpha'$} (m-1-4)
    (m-3-2) edge (m-3-4)
    (m-5-4) edge (m-3-4) edge node[above,shift={(-.5,0)}]{$j'$} (m-5-6)
    (m-7-4) edge (m-7-6)
    (m-9-6) edge (m-7-6) edge (m-9-8);
   \path[->>]
    (m-1-4) edge node[above]{$\beta'$} (m-1-6)
    (m-3-4) edge (m-3-6)
    (m-5-6) edge (m-5-8)
    (m-7-6) edge (m-7-8)
    (m-9-8) edge (m-9-10);
      \path[->]
    (m-2-1) edge node[above left]{$a$} (m-1-2)
    (m-2-3) edge node[above left]{$b$} (m-1-4)
    (m-2-5) edge node[above left]{$c$} (m-1-6)
    (m-4-1) edge node[above left]{$a$} (m-3-2)
    (m-4-3) edge[-,line width=6pt,draw=white] (m-2-3) edge (m-2-3) edge (m-3-4)
    (m-4-5) edge[-,line width=6pt,draw=white] (m-2-5) edge (m-2-5) edge (m-3-6)
    (m-6-3) edge node[above left]{$b$} (m-5-4)
    (m-6-5) edge (m-5-6)
    (m-6-7) edge node[above left]{$s$} (m-5-8)
    (m-8-3) edge node[above left]{$b$} (m-7-4)
    (m-8-5) edge (m-6-5) edge (m-7-6)
    (m-8-7) edge[-,line width=6pt,draw=white] (m-6-7) edge (m-6-7) edge (m-7-8)
    (m-10-5) edge (m-8-5) edge (m-9-6)
    (m-10-7) edge (m-9-8)
    (m-10-9) edge node[above left]{$t$} (m-9-10);
   \path
    (m-4-1) edge[-,line width=6pt,draw=white] (m-2-1) edge[double, double distance=1.5pt] (m-2-1)
    (m-6-5) edge[-,line width=6pt,draw=white] (m-4-5) edge[double, double distance=1.5pt] (m-4-5)
    (m-8-3) edge[-,line width=6pt,draw=white] (m-6-3) edge[double, double distance=1.5pt] (m-6-3)
    (m-10-7) edge[-,line width=6pt,draw=white] (m-8-7) edge[double, double distance=1.5pt] (m-8-7);
   \path[>->]
    (m-2-1) edge[-,line width=6pt,draw=white] (m-2-3) edge node[above, shift={(.25,0)}]{$\alpha$} (m-2-3)
    (m-4-1) edge (m-4-3)
    (m-6-3) edge (m-4-3) edge[-,line width=6pt,draw=white] (m-6-5) edge node[above,shift={(.5,0)}]{$j$} (m-6-5)
    (m-8-3) edge (m-8-5)
    (m-10-5) edge (m-8-5) edge (m-10-7);
   \path[->>]
    (m-2-3) edge [-,line width=6pt,draw=white] (m-2-5) edge node[above,shift={(.5,0)}]{$\beta$} (m-2-5)
    (m-4-3) edge[-,line width=6pt,draw=white] (m-4-5) edge (m-4-5)
    (m-6-5) edge[-,line width=6pt,draw=white] (m-6-7) edge (m-6-7)
    (m-8-5) edge[-,line width=6pt,draw=white] (m-8-7) edge (m-8-7)
    (m-10-7) edge (m-10-9);
  \end{tikzpicture}
  {\caption*{The ``Puppe sequence''}}\label{diag:puppe}
 \end{figure}
 The Cylinder and Saturation Axioms imply that all vertical arrows in this diagram are weak equivalences.
 Moreover, we have the following commutative square in the category of arrows of $\cC$:
 \begin{equation*}
   \commsquare{(A \xrightarrow{\alpha} B)}{(A' \xrightarrow{\alpha'} B')}{(A \to *)}{(A' \to *)}{(a,b)}{(\id,*)}{(\id,*)}{(a,*)}
 \end{equation*}
 Applying the cylinder functor to this square, and taking quotients with respect to the front and back inclusions
 of the cylinders, we obtain a commutative square
 \begin{equation}\label{eq:extension-axiom:idenfication-sigma-a}
   \commsquare{S = \cyl(\alpha)/(A \vee B)}{\cyl(\alpha')/(A' \vee B') = S'}{\Sigma A}{\Sigma A'}{s}{\sim}{\sim}{\Sigma a}
 \end{equation}
 in which the vertical arrows are weak equivalences. Since we assumed $a$ to be a weak equivalence, $\Sigma a$ is one by the Gluing Lemma.
 It follows that $s$, and therefore also the induced (nameless) morphism $C(j) \to C(j')$ is a weak equivalence.
 Note that the (also nameless) morphism $C(\alpha) \to C(\alpha')$ is also a weak equivalence because $c$ is
 a weak equivalence. Hence, $t$ is a weak equivalence by the Gluing Lemma. Just like $s$, the morphism
 $t$ sits in a square like \eqref{eq:extension-axiom:idenfication-sigma-a} together with the induced morphism $\Sigma b \colon \Sigma B \to \Sigma B'$.
 Hence, $\Sigma b$ is a weak equivalence.
\end{proof}

%We also need the following cofinality theorem which is a reformulation of Theorem 1.6 of {\cite{Vogell1990}}, which Vogell attributes to Thomason:
We also need the following cofinality theorem {\cite[Theorem~1.6]{Vogell1990}}, which Vogell attributes to Thomason:
 \begin{theorem}[Vogell cofinality]\label{thm_vogellcofinality}
  Let $\cC$ be a Waldhausen category which has a cylinder functor such that the Cylinder Axiom holds.
  Let $\cD \subset \cC$ be a full Waldhausen subcategory of $\cC$. Assume that
%  Let $\cD \subset \cC$ be a Waldhausen subcategory of $\cC$. Assume that
  \begin{enumerate}
    \item \label{item:vogell-cofinality-1:weakly-cofinal}
%     $\cD \subset \cC$ is \emph{weakly cofinal} in the sense that there is some $k \in \NN$ such that for all $C \in \cC$
%    there is some $C' \in \cC$ such that $\Sigma^kC \vee C'$ is isomorphic to an object in $\cD$.
    $\cD \subset \cC$ is \emph{weakly cofinal} in the sense that for all $C \in \cC$
	there exist $C' \in \cC$ and $k \in \NN$ such that $\Sigma^kC \vee C'$ is isomorphic to an object in $\cD$.
   \item $\cD$ is \emph{saturated in $\cC$}, i.e., any object weakly equivalent (via some zig-zag) to an object in $\cD$ lies in $\cD$.
  \end{enumerate}
  Then there is a homotopy fiber sequence
  \begin{equation*}
   wS_\bullet\cD \to wS_\bullet\cC \to N_\bullet \coker(K_0\cD \to K_0 \cC).
  \end{equation*}
 \end{theorem}
 \begin{remark}
%   There are two differences in the statements of Theorem~\ref{thm_vogellcofinality} and Theorem~1.6 of~\cite{Vogell1990}.  First, condition (iii) in Vogell's formulation of the theorem is implied by the requirement that $\cD$ has to be a Waldhausen subcategory.  Second, our condition  ``weakly cofinal'' is stronger than Vogell's condition.  Namely, we require the existence of a $k$ which holds for all objects, while Vogell allows $k$ to be different for each object.
   There is a slight difference between Theorem~\ref{thm_vogellcofinality} and Theorem~1.6 of~\cite{Vogell1990}. We have omitted condition (iii) in Vogell's formulation of the theorem since it is implied by the requirement that $\cD$ is a Waldhausen subcategory.

   Let us also remark why Vogell's cofinality theorem holds true:
   As written, Vogell seems to prove only the cofinality theorem suggested by
   Thomason in \cite[Exercise~1.10.2]{TT1990} since, in the last three lines
   of his proof, he chooses ``$C_0'$ such that $C_0 \vee C'_0$ is in $\cD$''.
   This proves Theorem~\ref{thm_vogellcofinality} under the additional
   assumption that $k=0$ in
   condition~\eqref{item:vogell-cofinality-1:weakly-cofinal}.
%   uses strict cofinality in the sense that there exists for every $C
%   \in \cC$ some $C' \in \cC$ with $C \vee C' \in \cD$.  That is
   In fact, the more general statement follows:
   
   For $k \in \NN$, set $\cC_k := \{ C \in \cC \mid \exists~C' \in \cC \colon
   \Sigma^k C \vee C' \in \cD \}$. Then $\cC = \bigcup_{k \in \NN} \cC_k$,
   each $\cC_k$ is a Waldhausen subcategory of $\cC$ (one needs to check that
   each $\cC_k$ is closed under pushouts), and $\cD$ is cofinal in $\cC_0$.
   Hence, we can apply the case $k=0$ of the cofinality
   theorem~\ref{thm_vogellcofinality} to conclude that there is a homotopy
   fiber sequence $wS_\bullet\cD \to wS_\bullet\cC_0 \to
   N_\bullet\coker(K_0\cD \to K_0\cC_0)$. Observe that the suspension functor
   induces a functor $\Sigma \colon \cC_{k+1} \to \cC_k$; since the suspension functor induces an equivalence on algebraic $K$-theory \cite[Proposition~1.6.2]{Waldhausen1985}, we conclude that the inclusion functor $\cC_k \subset \cC_{k+1}$ does so, too. Therefore, $wS_\bullet\cC \simeq \hocolim_k wS_\bullet\cC_k \simeq wS_\bullet\cC_0$, and Theorem~\ref{thm_vogellcofinality}  follows.
%   In fact, in the last three lines of his proof, Vogell chooses ``$C_0'$ such that $C_0 \vee C'_0$ is in $\cD$'', which seems to use strict cofinality (and would thus prove the theorem only under stronger assumptions than assumption~\eqref{item:vogell-cofinality-1:weakly-cofinal}). However, using the notation of \textit{loc.\ cit.}, we can continue Vogell's proof as follows. By assumption \eqref{item:vogell-cofinality-1:weakly-cofinal}, we can choose $C_0'$ such that $\Sigma^kC_0 \vee C_0'$ is in $\cD$. As in Vogell's argument, it follows that $\Sigma^kC$ lies in $\cD$. Since $k$ does not depend on $C$, we have a functor $\Sigma^k \colon \cC^v \rightarrow \cD$.  The functor induces an equivalence on algebraic $K$-theory, with inverse induced by the inclusion $\cD \subseteq \cC^v$.  This follows as suspension gives an equivalence on algebraic $K$-theory by Proposition 1.6.2 of~\cite{Waldhausen1985}.  (Cf.~also the proof of Lemma~\ref{lem_ktheoryofstableequivalences} for a similar trick.  Note further, that the last display on page 171 of~\cite{Vogell1990} should be interpreted as being part of a ``Puppe sequence'', see the proof of Lemma~\ref{lem_puppesequence}.)
 \end{remark}
 
\begin{proposition}\label{prop_cofinality}
  \ 
 \begin{enumerate}
  \item The natural inclusion of the finite into the homotopy finite objects induces a weak equivalence
   \begin{equation*}
    hS_\bullet\cR^G_f(W,\fZ) \xrightarrow{\sim} hS_\bullet\cR^G_{hf}(W,\fZ),
   \end{equation*}
   hence a weak equivalence on algebraic $K$-theory spaces.
  \item The inclusion $\cR_f^G(W,\fZ) \subset \cR^G_{fd}(W,\fZ)$ induces an isomorphism on $K_i$ for $i \geq 1$ and an injection on $K_0$, where we take $K$-theory with respect to the $h$-equivalences.
 \end{enumerate}
\end{proposition}
\begin{proof}
 To prove the first part, we appeal to Waldhausen's Approximation Theorem. Both categories satisfy the Saturation Axiom and have a Cylinder Functor which satisfies the Cylinder Axiom.  The first part of the Approximation Property is clear, we check the second part.
  
 Let $F$ be a finite object, let $Y$ be homotopy finite, and let $f \colon F \to Y$ be a morphism.   We have to construct a finite object $F'$, and further a morphism $F \rightarrow F'$ as well as an $h$-equivalence $F' \xrightarrow{\sim} Y$, such that their composition is $f$.
  
 As $Y$ is homotopy finite, there is a finite object $F_0$ and an $h$-equivalence $e \colon F_0 \xrightarrow{\sim} Y$.
 The Approximation Property now follows from the mapping cylinder argument~\ref{lem_mappingcylinderargument}.

 We turn to the proof of the second part of the proposition, which uses Vogell's cofinality theorem \ref{thm_vogellcofinality}.
 
 Since the inclusion of $\cR^G_f(W,\fZ)$ into $\cR^G_{fd}(W,\fZ)$ factors via $\cR^G_{hf}(W,\fZ)$,
 we need only consider the inclusion of the latter category. We show that Vogell's cofinality theorem applies in this situation.
 
 We show first that $\cR^G_{hf}(W,\fZ)$ is saturated in $\cR^G_{fd}(W,\fZ)$.  
 Let $Y_1$ be homotopy finite.  There is a finite object $(F, s_F, r_F)$ and an $h$-equivalence $a\colon F \rightarrow Y_1$.  Let $b\colon Y_2 \rightarrow Y_1$ be an $h$-equivalence.  There is an inverse map $\overline{b}\colon Y_1 \rightarrow Y_2$, but the composition $\overline{b}a$ does not have to respect the retractions.  Define a retraction $r' := r_{Y_2} \overline{b} a$, then $(F, s_F, r') \rightarrow (Y_2, s_{Y_2}, r_{Y_2})$ is an $h$-equivalence and $(F, s_F, r')$ is a different object, but still finite.  Hence $Y_2$ is homotopy finite.  The case $b \colon Y_1 \rightarrow Y_2$ is obvious and the general case follows by induction.

 So we are left with showing weak cofinality. Let $Y \in \cR^G_{fd}(W,\fZ)$ be arbitrary. Then we can find
 a finite object $D$ as well as a morphism $d \colon D \to Y$ and a map $i \colon Y \to D$ such that $d \circ i$
 is controlled homotopic to the identity map. Define an object $\tilde{Y}$ which is the same controlled
 $G$-CW-complex as $Y$, but equipped with a new retraction which turns $i \colon \tilde{Y} \to D$ into a morphism.
 Note that the composition $d \circ i \colon \tilde{Y} \to Y$ is an $h$-equivalence.
 
 Let $C\tilde{Y}$ denote the cone $\cyl(\id_{\tilde{Y}}) / \tilde{Y}$,
 let $S\tilde{Y}$ denote the object $C\tilde{Y} \cup_{\tilde{Y}} C\tilde{Y}$, and define $S\cyl(i)$ and $SC(i)$ analogously.
 There is a canonical $h$-equivalence $S\tilde{Y} \to C\tilde{Y} \cup_{\tilde{Y}} * \cong \Sigma\tilde{Y}$.
 Moreover, we have a morphism $S\cyl(i) \to \Sigma\cyl(i) \vee \Sigma\cyl(i)$ given by the quotient map with respect to the canonical
 cofibration $\cyl(i) \hookrightarrow S\cyl(i)$. These objects fit into a commutative diagram
 \begin{equation*}
  \begin{tikzpicture}
   \matrix (m) [matrix of math nodes, column sep=2em, row sep=2em, text depth=.5em, text height=1em]
     {S\tilde{Y} & S\cyl(i) & SC(i) \\ \Sigma Y & \Sigma Y \vee \Sigma C(i) & \Sigma C(i) \\};
   \path[>->]
     (m-1-1) edge (m-1-2)
     (m-2-1) edge (m-2-2);
   \path[->>] 
     (m-1-2) edge (m-1-3)
     (m-2-2) edge (m-2-3);
   \path[->]
     (m-1-1) edge (m-2-1)
     (m-1-2) edge (m-2-2)
     (m-1-3) edge (m-2-3);
   \end{tikzpicture}
 \end{equation*}
 in which the upper row comes from the cofiber sequence and the lower row is the split cofiber sequence. The vertical arrows are given as follows: The left vertical morphism is the composition
 $S\tilde{Y} \xrightarrow{\sim} \Sigma\tilde{Y} \xrightarrow{\Sigma di} \Sigma Y$, and the right vertical arrow
 is the canonical morphism $SC(i) \to * \cup_{C(i)} C(C(i)) \cong \Sigma C(i)$, i.e., the  collapse of the other half of $SC(i)$. For the middle vertical morphism, we take the composition
 $S\cyl(i) \to \Sigma\cyl(i) \vee \Sigma\cyl(i) \xrightarrow{\Sigma d' \vee \Sigma q} \Sigma Y \vee \Sigma C(i)$,
 where $d'$ is the composition of the back projection $\cyl(i) \to D$ with $d$, and $q$ is the projection $\cyl(i) \to C(i)$.
 Since both the left and the right vertical arrows are $h$-equivalences, it follows from Lemma \ref{lem_puppesequence}
 that the induced morphism $\Sigma S \cyl(i) \to \Sigma^2 Y \vee \Sigma^2 C(i)$ is an $h$-equivalence.
 Recall that $\cyl(i)$ is homotopy finite (since it is $h$-equivalent to $D$).  Since $\Sigma S\cyl(i)$ is $h$-equivalent to $\Sigma^2 \cyl(i)$, and
 suspension preserves finiteness, it follows that $\Sigma^2 Y \vee \Sigma^2 C(i)$ is homotopy finite. This proves weak cofinality, and we are done.
\end{proof}

\begin{definition}
 Let $D \in \cR^G(W,\fZ)$, and let $\alpha \colon D \to D$ be a controlled map. The \emph{mapping telescope} $\tel(\alpha)$ of $\alpha$ is the controlled $G$-CW-complex relative $W$
 \begin{equation*}
  \cyl(\alpha) \cup_D \cyl(\alpha) \cup_D \dots
 \end{equation*}
 obtained by taking countably many copies of the mapping cylinder of $\alpha$ and gluing the back and front end of each consecutive pair of cylinders.
\end{definition}

Note that $\tel(\alpha)$ does not need to be a retractive space. However, in certain cases it can be equipped with a retraction, and can then be used to replace dominated spaces by ``nicer'' ones.

\begin{proposition}\label{prop:mapping-telescope}
 Let $Y,D \in \cR^G(W,\fZ)$. Suppose we have maps $i \colon Y \to D$ and $d \colon D \to Y$ such that $d \circ i$ is (controlled) homotopic to $\id_Y$.
 
 Then the canonical map $j \colon Y \xrightarrow{i} D \rightarrowtail \tel(i \circ d)$ is a controlled homotopy equivalence.
 If $d$ is a morphism, $\tel(i \circ d)$ admits a retraction such that there exists an $h$-equivalence $\tel(i \circ d) \xrightarrow{\sim} Y$ which is a homotopy inverse to $j$.
\end{proposition}
\begin{proof}
 The proof of Proposition 1.4 in \cite{Ferry-Ranicki;wall-finiteness} works also in our setting. Note that we have an $h$-equivalence $Y \simeq Y \halftimes [0,\infty[$ because the control map disregards the cylinder coordinate.
 
 The homotopy $d \circ i \simeq \id_Y$ induces a map $\overline{d} \colon \cyl(i \circ d) \to Y$ which restricts to $d$ on the front and back. Hence, countably many copies of $\overline{d}$ glue to a controlled map $T(d) \colon \tel(i \circ d) \to Y$. Since $T(d) \circ j = d \circ i \simeq \id_Y$, the map $T(d)$ is homotopy inverse to $j$. If $d$ is a morphism, we can define a retraction $r \colon \tel(i \circ d) \to W$ by composing $T(d)$ with the retraction of $Y$. This makes $D \rightarrowtail \tel(i \circ d)$ into a morphism, and $T(d)$ becomes an $h$-equivalence.
\end{proof}

\begin{corollary}\label{cor:comparison-finite-dimensional}
 Let $\cR^G_{fd,\dim < \infty}(W,\fZ) \subset \cR^G_{fd}(W,\fZ)$ denote the full Waldhausen subcategory of finite-dimensional objects.
 Then the inclusion functor induces a weak equivalence
 \begin{equation*}
  hS_\bullet\cR^G_{fd,\dim < \infty}(W,\fZ) \xrightarrow{\sim} hS_\bullet\cR^G_{fd}(W,\fZ).
 \end{equation*}
\end{corollary}
\begin{proof}
 This is another application of the Approximation Theorem, using Proposition \ref{prop:mapping-telescope} and the mapping cylinder argument \ref{lem_mappingcylinderargument}.
\end{proof}

\subsection{Comparing different notions of weak equivalences}
Let $\fZ$ be a coarse structure, and let $A \subset Z$ be a $G$-invariant subspace.
We would like to compare the $K$-theory spaces of $\cR^G_f(W,\fZ)$ with respect to the $h$- and $h^A$-equivalences.
Unfortunately, the standard procedure to obtain homotopy fiber sequences relating these does not apply
in our situation since the Fibration Theorem requires one subcategory of weak equivalences
to satisfy the Extension Axiom.
We present a solution to this problem which has also been employed by Weiss \cite[towards the end of the proof of Proposition 8.3]{Weiss2002}. For the sake of completeness, we record its validity for any suitable Waldhausen category.

Let $(\cC,co\cC,w\cC)$ be a small Waldhausen category which satisfies the Saturation Axiom and
possesses a cylinder functor which satisfies the Cylinder Axiom with respect to $w\cC$.

\begin{definition}\label{def_stableequivalences}
 We call a morphism $f$ in $\cC$ an \emph{equivalence after $n$-fold suspension} if $\Sigma^n f$ lies in $w\cC$. We say that $f$ is a \emph{stable equivalence} if there is some $n \in \NN$ such that $f$ is an equivalence after $n$-fold suspension. Denote the class of equivalences after $n$-fold suspension by $w_{\Sigma,n}\cC$, and the class of stable equivalences by $w_\Sigma \cC$.
\end{definition}

\begin{lemma}\label{lem_ktheoryofstableequivalences}
 Let $n \geq 0$.
 \begin{enumerate}
   \item The collections $w_{\Sigma,n}\cC$ and $w_\Sigma\cC$ are classes of weak equivalences which satisfy the Saturation Axiom. The cylinder functor satisfies the Cylinder Axiom with respect to both classes.
   Moreover, $w_\Sigma\cC$ satisfies the Extension Axiom.
   \item The natural map $wS_\bullet\cC \to w_\Sigma S_\bullet\cC$ is a weak equivalence.
 \end{enumerate}
 \end{lemma}
 \begin{proof}
 Almost everything of the first part of the lemma is straightforward; the only exception is the validity of the Extension Axiom for $w_\Sigma \cC$.
   
 Assume that we have a commutative diagram of exact sequences
 \begin{equation*}
   \begin{tikzpicture}
   \matrix (m) [matrix of math nodes, column sep=2em, row sep=2em, text depth=.5em, text height=1em]
     {A & B & C \\ A' & B' & C' \\};
   \path[>->]
     (m-1-1) edge (m-1-2)
     (m-2-1) edge (m-2-2);
   \path[->>] 
     (m-1-2) edge (m-1-3)
     (m-2-2) edge (m-2-3);
   \path[->]
     (m-1-1) edge node[left]{$a$} (m-2-1)
     (m-1-2) edge node[left]{$b$} (m-2-2)
     (m-1-3) edge node[left]{$c$} (m-2-3);
   \end{tikzpicture}
 \end{equation*}
 in which $a$ and $c$ are weak equivalences after $n$-fold suspension. Suspend the diagram $n$ times to
 obtain a diagram of the same shape in which the left and right arrows are weak equivalences.
 Then it follows from Lemma \ref{lem_puppesequence} that $b$ is a weak equivalence after $(n+1)$-fold suspension.
 
 For the second part, we can apply the Fibration Theorem to the inclusion $w\cC \subset w_\Sigma\cC$
 since we have just shown that $w_\Sigma\cC$ satisfies the Saturation and Extension Axioms, and that the Cylinder Axiom holds as well.
 Therefore, it suffices to show that $wS_\bullet\cC^{w_\Sigma}$ is contractible.
 Observe that $\cC^{w_\Sigma}$ is the union of the ascending sequence
 $\cC^w \subset \cC^{w_{\Sigma,1}} \subset \cC^{w_{\Sigma,2}} \subset \dots$.
 Since $K$-theory commutes with directed colimits, it is enough to show that each
 $\cC^{w_{\Sigma,n}}$ has contractible $K$-theory.
 
 By the Additivity Theorem, the exact endofunctor $\Sigma^n \colon \cC^{w_{\Sigma,n}} \to \cC^{w_{\Sigma,n}}$
 induces a self-homotopy equivalence in $K$-theory. Furthermore, it factors over $\cC^w$.
 Since $wS_\bullet\cC^w$ is contractible, the claim follows.
\end{proof}

\begin{proposition}[Modified Fibration Theorem]\label{prop_modifiedfibrationtheorem}
 Let $\cC$ be a category with cofibrations. Let $v\cC \subset w\cC$ be two subcategories of
 weak equivalences. Suppose that $\cC$ has a cylinder functor which satisfies the Cylinder Axiom with respect to $v\cC$
 (hence also with respect to $w\cC$). Assume that $v\cC$ and $w\cC$ satisfy the Saturation Axiom.
  
 Then the canonical inclusion functors induce a homotopy pullback square
 \begin{equation*}
   \commsquare{vS_\bullet\cC^w}{wS_\bullet\cC^w \simeq *}{vS_\bullet\cC}{wS_\bullet\cC}{}{}{}{}
 \end{equation*}
 in which the corner on the top right is canonically contractible.
 \end{proposition}
\begin{proof}
 The square that we are claiming to be a homotopy pullback comes with a transformation
 to the square
 \begin{equation}\label{diagram:fibration:theorem}
   \commsquare{vS_\bullet\mathcal{C}^{w_\Sigma}}{w_\Sigma S_\bullet\mathcal{C}^{w_\Sigma} \simeq *}{vS_\bullet\mathcal{C}}{w_\Sigma S_\bullet\mathcal{C}}{}{}{}{}
 \end{equation}
 This transformation is the identity on the lower left corner, and is induced by the canonical inclusion functors on the other three corners.
 As $w_\Sigma\cC$ satisfies the Extension Axiom by Lemma \ref{lem_ktheoryofstableequivalences}, the square \ref{diagram:fibration:theorem} is a homotopy pullback by the Fibration Theorem.
 The entries on the top right corners of both squares are contractible.
 The map between the lower right corners is a weak equivalence by Lemma \ref{lem_ktheoryofstableequivalences}.
 So all we have to check is that the canonical map $vS_\bullet\mathcal{C}^w \to vS_\bullet\mathcal{C}^{w_\Sigma}$ is a weak equivalence.
  
 Just as in the proof of Lemma \ref{lem_ktheoryofstableequivalences}, we can write $\cC^{w_\Sigma}$
 as a directed union $\cC^{w_\Sigma} = \bigcup_n \cC^{w_{\Sigma,n}}$. In the direct limit system
 \begin{equation*}
   \cC^w \hookrightarrow \cC^{w_{\Sigma,1}} \hookrightarrow \cC^{w_{\Sigma,2}} \hookrightarrow \dots,
 \end{equation*}
 each arrow induces an equivalence in $K$-theory (the suspension functor provides a homotopy inverse),
 and the claim follows from this.
\end{proof}

We can now begin to study the $K$-theory of categories of controlled retractive spaces.
For technical reasons, we will need to impose certain intermediate finiteness conditions
on the objects as long as we are dealing with connective $K$-theory.
This phenomenon is well-known in the linear setting, see \cite{CP1997}.

Let $F \leq K_0(\cR^G_{fd}(W,\fZ),h)$ be a subgroup. Denote by $\cR^G_{fd,F}(W,\fZ)$
the full subcategory of all those objects whose $K_0$-class lies in $F$.
We think of these objects as being subject to an intermediate finiteness condition.
Note that these objects can be equivalently characterized as those complexes
whose $K_0$-class lies in the kernel of the projection homomorphism $K_0(\cR^G_{fd}(W,\fZ),h) \to K_0(\cR^G_{fd}(W,\fZ),h)/F$.
It is now a consequence of Thomason's cofinality theorem \cite[Cofinality Theorem 1.10.1]{TT1990}
that there is a homotopy fiber sequence
\begin{equation*}
 h S_\bullet \cR^G_{fd,F}(W,\fZ) \to h S_\bullet \cR^G_{fd}(W,\fZ) \to N_\bullet(K_0(\cR^G_{fd}(W,\fZ),h)/F).
\end{equation*}
In particular, the change of finiteness condition only affects $K_0$; there, the induced map is a monomorphism
which can be identified with the inclusion $F \hookrightarrow K_0(\cR^G_{fd}(W,\fZ),h)$.
A typical choice for $F$ is $F := K_0(\cR_f^G(W,\fZ),h)$, which we regard as a subgroup of $K_0(\cR^G_{fd}(W,\fZ),h)$ by virtue of Proposition \ref{prop_cofinality}.

Let $\fZ = (Z,\fC,\fS)$ be a coarse structure and $A$ be a $G$-invariant subset of $Z$.
Let $\fS\langle A \rangle$ be the collection of all sets of the form $A^C \cap S$, where $S$ is an element of $\fS$ and $C \in \fC$, see Definition \ref{def:thickenings:cofinal:subcomplexes}.
Define a new coarse structure $\fZ\langle A \rangle := (Z, \fC, \fS\langle A \rangle)$.

Recall from Definition \ref{def_controlledhomotopyequivalences} that $A$ gives rise to a class of weak equivalences $h^A\cR^G_f(W,\fZ)$.

\begin{theorem}\label{thm_connectivefibresequence}
 Let $\fZ$ be a coarse structure and let $A \subset Z$ be a $G$-invariant subset.
  
 Set $K := K_0(\cR^G_f(W,\fZ),h) \leq K_0(\cR^G_{fd}(W,\fZ),h)$. Let $F \leq K_0(\cR^G_{fd}(W,\fZ\langle A\rangle),h)$ be the preimage of $K$
 under the natural homomorphism $K_0(\cR^G_{fd}(W,\fZ\langle A\rangle),h) \to K_0(\cR^G_{fd}(W,\fZ),h)$.
 Then
 \begin{equation}\label{eq:connectivefibresequence:1}
   hS_\bullet\cR^G_{fd,F}(W,\fZ\langle A \rangle) \to hS_\bullet\cR^G_{fd,K}(W,\fZ) \to h^AS_\bullet\cR^G_{fd,K}(W,\fZ)
 \end{equation}
 is a homotopy fiber sequence. Upon realization, there is a homotopy fiber sequence
 \begin{equation}\label{eq:connectivefibresequence:2}
  \abs{hS_\bullet\cR^G_{fd,F}(W,\fZ\langle A \rangle)} \to \abs{hS_\bullet\cR^G_f(W,\fZ)} \to \abs{h^AS_\bullet\cR^G_f(W,\fZ)}
 \end{equation}
 which is weakly equivalent to the former one.
\end{theorem}
\begin{proof}
 The Modified Fibration Theorem \ref{prop_modifiedfibrationtheorem} applies to our situation, so we have a homotopy fiber sequence
 \begin{equation*}
  hS_\bullet\cR^G_f(W,\fZ)^{h^A} \to hS_\bullet\cR^G_f(W,\fZ) \to h^AS_\bullet\cR^G_f(W,\fZ).
 \end{equation*}
 Define $F' \leq K_0(\cR^G_{fd}(W,\fZ)^{h^A},h)$ to be the preimage of $K_0(\cR^G_f(W,\fZ),h)$
 under the canonical homomorphism $K_0(\cR^G_{fd}(W,\fZ)^{h^A},h) \to K_0(\cR^G_{fd}(W,\fZ),h)$.
 We first prove the following two assertions:
 \begin{enumerate}
  \item\label{item:connectivefibresequence:1} The natural inclusion functor $\cR^G_f(W,\fZ)^{h^A} \to \cR^G_{fd,F'}(W,\fZ)^{h^A}$ induces an equivalence in $K$-theory.
  \item\label{item:connectivefibresequence:2} The natural inclusion functor $\cR^G_{fd,F}(W,\fZ\langle A \rangle) \to \cR^G_{fd,F'}(W,\fZ)^{h^A}$ induces an equivalence in $K$-theory.
 \end{enumerate}
 
 The first claim is proved in a fashion similar to that of Proposition \ref{prop_cofinality}. First of all, we may replace $\cR^G_f(W,\fZ)^{h^A}$
 by the category of homotopy finite objects $\cR^G_{hf}(W,\fZ)^{h^A}$ since the inclusion of the former into the latter induces an equivalence
 in $K$-theory by the Approximation Theorem.
 
 To conclude that the inclusion $\cR^G_{hf}(W,\fZ)^{h^A} \to \cR^G_{fd,F'}(W,\fZ)^{h^A}$
 induces an equivalence as well, we rely on Vogell's cofinality theorem \ref{thm_vogellcofinality} once again.
 As we alreay saw in the proof of Proposition \ref{prop_cofinality}, there are for every object $Y_1 \in \cR^G_{fd,F'}(W,\fZ)^{h^A}$ some
 finitely dominated object $Y_2 \in \cR^G_{fd}(W,\fZ)$ with $[Y_2] \in K_0(\cR^G_f(W,\fZ),h)$, a finite object $Y \in \cR^G_f(W,\fZ)$ and an $h$-equivalence $f \colon Y \xrightarrow{\sim} \Sigma^2 Y_1 \vee Y_2$.
 However, there is no reason for $Y_2$ to be $h^A$-contractible, so we have to improve it.
 
 Since $Y_1$ is $h^A$-contractible, there is some cofinal subcomplex $Y_1' \subset Y_1$ away from $A$ such that
 the inclusion map $j_1 \colon Y_1' \hookrightarrow Y_1$ is nullhomotopic. By Lemma \ref{lemma:cofinal-subcomplexes}, there is a cofinal subcomplex $Y' \subset Y$
 whose image under $f$ is contained in $\Sigma^2Y_1' \vee Y_2$; note that $Y' \subset Y$ is finite. Let $j \colon Y' \to Y$ be the inclusion map.
 Define $f' \colon Y' \to Y_2$ as the composition $Y' \xrightarrow{j} Y \xrightarrow{f} \Sigma^2Y_1 \vee Y_2 \twoheadrightarrow Y_2$.
 Set $Y_3 := C(f')$ and observe that this is a finitely dominated object with $[Y_3] \in K_0(\cR_f^G(W,\fZ),h)$.
 Let $i_2 \colon Y_2 \rightarrowtail \Sigma^2Y_1 \vee Y_2$ be the canonical cofibration.
 Since $j_1$ is nullhomotopic, the map $fj$ is homotopic to $i_2f'$. Hence, there is an $h$-equivalence $C(fj) \xrightarrow{\sim_h} C(i_2f') \simeq_h \Sigma^2Y_1 \vee Y_3$.
 Observe that $C(fj)$ is homotopy finite, so $\Sigma^2Y_1 \vee Y_3$ is homotopy finite, too.
 
 Moreover, the natural map $C(f j) \to C(f)$ is an $h^A$-equivalence
 because $j$ is one. As $C(f)$ is contractible, both $C(f j)$ and $\Sigma^2Y_1 \vee Y_3$ are $h^A$-contractible.
 It follows that $Y_3$ is also $h^A$-contractible. Note $[Y_3] \in F'$. Therefore, Vogell's cofinality theorem applies.
 That is, all higher $K$-theory groups of $\cR^G_{hf}(W,\fZ)^{h^A}$ and $\cR^G_{fd,F'}(W,\fZ)^{h^A}$ coincide,
 and the induced homomorphism $K_0(\cR^G_{hf}(W,\fZ)^{h^A},h) \to K_0(\cR^G_{fd,F'}(W,\fZ)^{h^A},h)$ is injective.
 
 On the level of $K_0$, we have a commutative diagram
 \begin{equation*}
  \commtriangle{K_0(\cR^G_f(W,\fZ)^{h^A},h)}{K_0(\cR^G_{fd}(W,\fZ)^{h^A},h)}{K_0(\cR^G_{fd,F'}(W,\fZ)^{h^A},h)}{}{}{}
 \end{equation*}
 in which the left diagonal arrow is an injection. The right diagonal map
 is an injection by Thomason's cofinality theorem. Since the top horizontal and right diagonal homomorphism
 have the same image, it follows that the left diagonal map is an isomorphism. This shows Claim \ref{item:connectivefibresequence:1}.
 
 Let us now turn to the second claim. We apply the Approximation Theorem.
 By Corollary \ref{cor:comparison-finite-dimensional}, we may assume without loss of generality that all complexes are finite-dimensional.
 Only the second part of the Approximation Property needs checking. Let $Y_0 \in \cR^G_{fd,F,\dim < \infty}(W,\fZ\langle A \rangle)$; then $Y_0$ is $h^A$-contractible.
 Let $f \colon Y_0 \to Y$ be a morphism in $\cR^G_{fd,F',\dim < \infty}(W,\fZ)^{h^A}$.
 
 Let $r \colon Y \to W$ be the structural retraction, $s \colon W \to Y$ be the structural inclusion of $Y$.
 Since $Y$ is $h^A$-contractible, there are a subcomplex $Y' \subset Y$ which is cofinal away from $A$ and a homotopy $h$
 from the inclusion map $Y' \hookrightarrow Y$ to the composition $Y' \xrightarrow{r|_{Y'}} W \xrightarrow{s} Y$.
 By the CHEP, we find an extension of $h$ to a controlled homotopy $H \colon Y \halftimes [0,1] \to Y$ from $\id_Y$ to a controlled map $p' \colon Y \to Y$ which extends $s \circ r|_{Y'}$.
  
 Let $Y''$ be the $G$-subcomplex of $Y$ generated by the image of $p'$. Since $p'$ is controlled and $Y$ is finite-dimensional, $Y''$ is supported on some $\fZ$-thickening of $A$.
 Note that we do not claim that $Y''$ is finitely dominated. Let $j \colon Y'' \hookrightarrow Y$ be the inclusion. Denote by $p$ the map $p'$, regarded as a map $Y \to Y''$. Then $Y \xrightarrow{p} Y'' \xrightarrow{j} Y$ is homotopic to the identity.
 Proposition \ref{prop:mapping-telescope} provides us with an $h$-equivalence $T(j) \colon \tel(jp) \xrightarrow{\sim} Y$.
 In particular, $[\tel(jp)] = [Y] \in K_0(\cR^G_{fd}(W,\fZ),h)$, hence both lie in $K_0(\cR^G_f(W,\fZ),h)$. Observe that $\tel(jp)$ is also supported on a $\fZ$-thickening of $A$.
 
 Pick now a finite domination $Y \xrightarrow{i} D \xrightarrow{d} Y$ of $Y$ in $\cR^G(W,\fZ)$.
 Let $D'$ be the smallest subcomplex of $D$ which contains the image of $i \circ T(j)$.
 Since $i \circ T(j)$ is a controlled map and $\tel(jp)$ is supported on a $\fZ$-thickening of $A$,
 the complex $D'$ is also supported on some $\fZ$-thickening of $A$. The composition of the maps $\tel(jp) \xrightarrow{i \circ T(j)} D'$
 and $d' \colon D' \hookrightarrow D \xrightarrow{d} Y \to \tel(jp)$ is homotopic to the identity. Redefining the retraction of $D'$ to be $d'$ composed with the retraction of $\tel(jp)$, the latter map becomes a morphism.
 Hence, $\tel(jp) \in \cR^G_{fd}(W,\fZ\langle A \rangle)$.
 
 Another application of the mapping cylinder argument \ref{lem_mappingcylinderargument} to $Y_0 \xrightarrow{f} Y \xleftarrow{T(j)} \tel(jp)$ yields the Approximation Property, and hence the second assertion.
 
 There is a map of homotopy fiber sequences
 \begin{equation*}
  \twocommsquare{hS_\bullet\cR^G_f(W,\fZ)^{h^A}}{hS_\bullet\cR^G_f(W,\fZ)}{h^AS_\bullet\cR^G_f(W,\fZ)}
   {hS_\bullet\cR^G_{fd,K}(W,\fZ)^{h^A}}{hS_\bullet\cR^G_{fd,K}(W,\fZ)}{h^AS_\bullet\cR^G_{fd,K}(W,\fZ)}
 \end{equation*}
 The middle vertical map is a weak equivalence by Thomason cofinality. Observe that $\cR^G_{fd,F'}(W,\fZ)^{h^A} = \cR^G_{fd,K}(W,\fZ)^{h^A}$. Hence, assertion \eqref{item:connectivefibresequence:1} implies that the left vertical map is a weak equivalence. Therefore, the right vertical map is a weak equivalence, too. Composing the weak equivalence $hS_\bullet\cR^G_{fd,F}(W,\fZ\langle A \rangle) \xrightarrow{\sim} hS_\bullet\cR^G_{fd,F'}(W,\fZ)^{h^A}$ from assertion \eqref{item:connectivefibresequence:2} with the inclusion of the homotopy fiber yields sequence \eqref{eq:connectivefibresequence:1}.
 
 After taking realizations, we can invert the weak equivalence of assertion \eqref{item:connectivefibresequence:1} to obtain sequence \eqref{eq:connectivefibresequence:2}.
\end{proof}

\begin{definition}[{cf.~\cite[Definition 4.1]{BFJR2004}}]\label{def_propercoarsestructure}
 A coarse structure $\fZ$ is called \emph{$G$-proper with respect to $A$}
 if for every $C \in \fC$, every $S \in \fS$ there are $S' \in \fS$, $C' \in \fC$ and a $G$-equivariant function
 $c \colon A^C \cap S \to A \cap S'$ such that
 \begin{enumerate}
  \item $\{ (c(z),z) \mid z \in A^C \cap S \} \subset C'$.
  \item for every set $B \subset A^C \cap S$ which is locally finite in $Z$, the image $c(B)$ is locally finite in $Z$ and $c^{-1}(x) \cap B$ is finite for all $x \in c(B)$.
 \end{enumerate}
\end{definition}

\begin{proposition}\label{prop:thickenings:restrictions}
 Let $A \subset Z$ be a closed, $G$-invariant subset. Let $\fZ$ be a coarse structure which is $G$-proper with respect to $A$.
 Recall the definition of the coarse structure $\fZ \cap A$ from Example \ref{def_restrictedcoarsestructure}.
 Let $F \leq K_0(\cR^G_{fd}(W,\fZ \cap A),h)$ be the preimage of $K_0(\cR^G_f(W,\fZ),h)$
 under the natural homomorphism $K_0(\cR^G_{fd}(W,\fZ \cap A),h) \to K_0(\cR^G_{fd}(W,\fZ),h)$, and define $F' \leq K_0(\cR^G_{fd}(W,\fZ\langle A \rangle),h)$ analogously.
 
 Then the exact inclusion functor
 \begin{equation*}
  \cR^G_{fd,F}(W,\fZ \cap A) \hookrightarrow \cR^G_{fd,F'}(W, \fZ\langle A \rangle)
 \end{equation*}
 is an equivalence of Waldhausen categories.
\end{proposition}
\begin{proof}
 It suffices to show that the inclusion functor is essentially surjective. So let $Y \in \cR^G_{fd,F'}(W,\fZ\langle A \rangle)$.
 For each $k \in \NN$, pick $S_k \in \fS$ such that $\kappa(\cells_k Y) \subset A^{C_k} \cap S_k$.
 Since $\fZ$ is $G$-proper with respect to $A$, there is a $G$-equivariant function $c_k \colon A^{C_k} \cap S_k \to A \cap S_k'$ as in Definition \ref{def_propercoarsestructure}.
 The collection $\{ c_k \circ \kappa|_{\cells_k Y} \}_k$ defines a $G$-equivariant function $\kappa_A \colon \cells Y \to A$
 such that $\kappa(\cells_k Y) \subset A \cap S_k'$. By construction, the identity map is a controlled
 isomorphism between $(Y,\kappa)$ and $(Y,\kappa_A)$, where the latter complex is now supported on $A$.
 The modification we make to the control map $\kappa$ also preserves finiteness, and hence finite dominations.
 Note that $[(Y,\kappa)] = [(Y,\kappa_A)] \in K_0(\cR^G_{fd}(W,\fZ),h)$, i.e., $[(Y,\kappa_A)] \in F$. This finishes the proof.
\end{proof}

\subsection{The coarse Mayer--Vietoris theorem}
The main application of the homotopy fiber sequence established in the previous subsection is the excision result we prove next.
Let $\fZ = (Z,\fC,\fS)$ be a coarse structure.

\begin{definition}[{\cite[Proposition 4.3]{BFJR2004}}]\label{def_coarselyexcisive}
 A pair $(A,B)$ of $G$-invariant subspaces in $Z$ is called \emph{coarsely excisive} if for all $C \in \fC$
 there is $C' \in \fC$ such that $A^C \cap B^C \subset (A \cap B)^{C'}$.
   
 A \emph{coarsely excisive triple} is a coarse structure $\fZ$ together with two closed, $G$-invariant subspaces
 $A_1$, $A_2 \subset Z$ such that $A_1 \cup A_2 = Z$ and the pair $(A_1, A_2)$ is coarsely excisive.
\end{definition}

We require a little more notation. For a closed, $G$-invariant subspace $A \subset Z$ we define $\fZ|_A := (Z,\fC,\fS \Cap A)$. Observe that $\cR^G(W,\fZ|_A) \cong \cR^G(W,\fZ \cap A)$.

\begin{lemma}\label{lem_equivalencesandexcision}
 Suppose that $(A,B)$ is a coarsely excisive pair. 
 Then
 \begin{equation*}
   h^{A \cap B}S_\bullet\cR^G_f(W,\fZ|_A) = h^BS_\bullet\cR^G_f(W,\fZ|_A).
 \end{equation*}
\end{lemma}
\begin{proof}
 Let $(Y,\kappa)$ be an object in $\cR^G_f(W,\fZ|_A)$. Note that the image of $\kappa$ is contained in $A$.
 It suffices to show that a subcomplex $Y' \subset Y$ is cofinal away from $A \cap B$ if and only if it is cofinal away from $B$.
 Since $A \cap B \subset B$, it is obvious that every subcomplex which is cofinal away from $A \cap B$
 is also cofinal away from $B$.
  
 Now suppose that $Y' \subset Y$ is cofinal away from $B$. Let $k \in \NN$.
 Then there is $C \in \fC$ such that $\kappa^{-1}(Z \setminus B^C) \cap \cells_k Y \subset \cells_k Y'$.
 By assumption, we can find $C' \in \fC$ such that $A^C \cap B^C \subset (A \cap B)^{C'}$.
 Then we have
 \begin{equation*}
   \begin{split}
    \kappa^{-1}(Z \setminus (A \cap B)^{C'}) \cap \cells_k Y
    &\subset \kappa^{-1}(Z \setminus (A^C \cap B^C)) \cap \cells_k Y \\
    &= \kappa^{-1}(Z \setminus A^C \cup Z \setminus B^C) \cap \cells_k Y \\
    &= (\kappa^{-1}(Z \setminus A^C) \cap \cells_k Y) \cup (\kappa^{-1}(Z \setminus B^C) \cap \cells_k Y) \\
    &= \kappa^{-1}(Z \setminus B^C) \cap \cells_k Y \\
    &\subset \cells_k Y'.
  \end{split}
 \end{equation*}
This shows that $Y' \subset Y$ is cofinal away from $A \cap B$, and we are done.
\end{proof}

\begin{theorem}[Coarse Mayer--Vietoris, connective version]\label{thm_connectivemayervietoris}
 Let $(\fZ,A_1,A_2)$ be a coarsely excisive triple, and assume that $\fZ$ is $G$-proper with respect to $A_1$, $A_2$ and $A_1 \cap A_2$.
 Let $F$ be the preimage of $K := K_0(\cR^G_f(W,\fZ),h)$ under the canonical homomorphism $K_0(\cR^G_{fd}(W,\fZ \cap A_2),h) \to K_0(\cR^G_{fd}(W,\fZ),h)$,
 and let $F'$ be the preimage of $K' := K_0(\cR^G_f(W,\fZ \cap A_1),h)$ under the canonical homomorphism $K_0(\cR^G_{fd}(W,\fZ \cap (A_1 \cap A_2)),h) \to K_0(\cR^G_{fd}(W,\fZ \cap A_1),h)$.
   
 Then the natural inclusion maps induce a homotopy pullback square
 \begin{equation*}
  \commsquare{hS_\bullet\cR^G_{fd,F'}(W,\fZ \cap (A_1 \cap A_2))}{hS_\bullet\cR^G_{fd,K'}(W,\fZ \cap A_1)}{hS_\bullet\cR^G_{fd,F}(W,\fZ \cap A_2)}{hS_\bullet\cR^G_{fd,K}(W,\fZ)}{}{}{}{}
 \end{equation*}
\end{theorem}
\begin{proof}
 Using Theorem \ref{thm_connectivefibresequence} and Proposition \ref{prop:thickenings:restrictions}, we have a map of homotopy fiber sequences
 \begin{equation*}
  \twocommsquare{hS_\bullet\cR^G_{fd,F'}(W,\fZ \cap (A_1 \cap A_2))}{hS_\bullet\cR^G_{fd,K'}(W,\fZ \cap A_1)}{h^{A_1 \cap A_2}S_\bullet\cR^G_{fd,K'}(W,\fZ \cap A_1)}{hS_\bullet\cR^G_{fd,F}(W,\fZ \cap A_2)}{hS_\bullet\cR^G_{fd,K}(W,\fZ)}{h^{A_2}S_\bullet\cR^G_{fd,K}(W,\fZ)}
 \end{equation*}
 Hence, it suffices to show that the right vertical map is a weak equivalence.
 We may replace $\cR^G_{fd,K'}(W,\fZ \cap A_1)$ and $\cR^G_{fd,K}(W,\fZ)$ by $\cR^G_f(W,\fZ \cap A_1)$ and $\cR^G_f(W,\fZ)$, respectively.
 Using Lemma \ref{lem_equivalencesandexcision}, we can therefore identify the right vertical map as the natural inclusion map
 \begin{equation*}
   h^{A_1 \cap A_2}S_\bullet\cR^G_f(W,\fZ \cap A_1) \cong h^{A_1 \cap A_2}S_\bullet\cR^G_f(W,\fZ|_{A_1}) = h^{A_2}S_\bullet\cR^G_f(W,\fZ|_{A_1}) \to h^{A_2}S_\bullet\cR^G_f(W,\fZ).
 \end{equation*}
 Our claim is that the Approximation Theorem applies to show that this is an equivalence. We need only check the second part of the Approximation Property.
 So let $f \colon Y_1 \to Y_2$ be a morphism in $\cR^G_f(W,\fZ)$ such that $Y_1$ is an object in $\cR^G_f(W,\fZ|_{A_1})$.
 Define $Y_2'$ as the smallest subcomplex of $Y_2$ which contains both the image of $f$ and all cells which are labelled with points in $A_1$.
 Then $Y_2'$ is supported on a $\fZ$-thickening of $A_1$. Since $\kappa^{-1}(Z \setminus A_2) \subset \kappa^{-1}(A_1) \subset \cells Y_2'$, the subcomplex $Y_2' \subset Y_2$ is cofinal away from $A_2$.
 This implies that the inclusion map $Y_2' \hookrightarrow Y_2$ is an $h^{A_2}$-equivalence.
 As in the proof of Proposition \ref{prop:thickenings:restrictions}, $\fZ$ being $G$-proper with respect to $A_1$ implies that $Y_2'$ is isomorphic to an object $Y_3$ with support in $A_1$.
 Then $f$ factors over $Y_3$, which shows the Approximation Property.
\end{proof}

\subsection{A vanishing result}\label{subsec:vanishing}
To conclude this section, we also record a criterion which guarantees the vanishing of the $K$-theory of a category of controlled retractive spaces.

\begin{proposition}[Eilenberg swindles]\label{prop_eilenbergswindles}
 Let $\cC$ be a small Waldhausen category. Let $\vee$ be a functorial coproduct on $\cC$.
 Suppose that there is an exact endofunctor $\swindle$ on $\cC$ and a natural isomorphism
 $\id \vee \swindle \cong \swindle$.
 
 Then there is a contraction $H_{\swindle}$ of $K(\cC)$ which is natural in the following sense:
 Let $\cC_1$ and $\cC_2$ be small Waldhausen categories, equipped with functorial coproducts $\vee_i$, $i=1,2$.
 Let $F \colon \cC_1 \to \cC_2$ be an exact functor which strictly preserves the coproduct, i.e., $F \circ \vee_1 = \vee_2 \circ (F \times F)$.
 Let $\swindle_i$ be exact endofunctors on $\cC_i$ together with natural isomorphisms $\eta_i \colon \id_{\cC_i} \vee \swindle_i \cong \swindle_i$, $i=1,2$,
 such that $\swindle_2 \circ F = F \circ \swindle_1$ and $F \circ \eta_{1,A} = \eta_{2,FA}$. Then
 \begin{equation*}
  H_{\swindle_2} \circ (K(F) \times [0,1]) = K(F) \circ H_{\swindle_1}.
 \end{equation*}
\end{proposition}
\begin{proof}
 Recall that $K(\cC) = \Omega\abs{wS_\bullet\cC}$. Hence, concatenation of loops defines an $H$-space structure ``$+$'' on $K(\cC)$.
 Subject to the choice of an orientation-preserving homeomorphism $[0,1] \cong [0,2]$, the $H$-space product is naturally homotopy associative.
 Similarly, any choice of orientation-reversing homeomorphism $[0,1] \cong [0,1]$, say $t \mapsto 1-t$, induces a homotopy inverse $\inv$
 such that $\id + \inv$ is nullhomotopic. This nullhomotopy depends on a choice of contraction of $[0,1]$ to the point $0$.
 Fixing, once and for all, suitable homeomorphisms $[0,1] \cong [0,2]$ and $[0,1] \cong [0,1]$ and a contraction of $[0,1]$,
 all these homotopies become natural with respect to maps induced by exact functors.
 
 The functorial coproduct $\vee \colon \cC \times \cC \to \cC$ induces another $H$-space structure ``$\vee$'' on $K(\cC)$.
 Since $+$ and $\vee$ satisfy the interchange law, the Eckmann--Hilton-argument shows that there is a natural homotopy $\vee \simeq +$.
 By abuse of notation, we use in the sequel the same name for functors and the maps they induce on $K$-theory.
 Let $0$ denote the constant functor mapping everything to the zero object.
 Then we have:
 \begin{equation*}
  \begin{split}
   \id
   &\simeq \id + 0 \\
   &\simeq \id + (\swindle + (\inv \circ \swindle)) \\
   &\simeq (\id + \swindle) + (\inv \circ \swindle) \\
   &\simeq (\id \vee \swindle) + (\inv \circ \swindle) \\
   &\simeq \swindle + \inv \circ \swindle \\
   &\simeq 0
  \end{split}
 \end{equation*}
 The fifth homotopy is induced by the natural isomorphism $\eta$. Hence, the concatenation of these homotopies defines a contraction of $K(\cC)$,
 and it is straightforward to check that this contraction is natural in the desired sense.
\end{proof}

\begin{proposition}\label{prop_connectiveeilenbergswindle}
 Let $\fZ$ be a coarse structure and let $A \subset Z$ be a $G$-invariant subset. Suppose that there is
 a sequence of $G$-equivariant functions $(s_n \colon Z \to Z)_{n \in \NN}$ which satisfies the following properties:
 \begin{enumerate}
  \item\label{es.0} $s_0 = \id_Z$.
  \item\label{es.1} For every $C \in \fC$ and $S \in \fS$ there is some $C' \in \fC$ such that 
   \begin{equation*}
    \bigcup_{n \geq 0} (s_n \times s_n)(C \cap (S \times S)) \subset C'.
   \end{equation*}
  \item\label{es.2} For every $S \in \fS$ there is some $S' \in \fS$ such that $\bigcup_n s_n(S) \subset S'$.
  \item\label{es.3} For every $S \in \fS$ and every $B \subset S$ which is locally finite in $Z$, each image $s_n(B)$ is locally finite in $Z$ and $s_n^{-1}(x) \cap B$ is finite for all $x \in s_n(B)$.
   Furthermore, there are for every $z \in Z$ some $n_0$ and an open neighbourhood $U$ of $z$ such that $s_n^{-1}(U) = \varnothing$ for all $n \geq n_0$.
  \item\label{es.4} For every $C \in \fC$ there exists $C' \in \fC$ such that
   \begin{equation*}
    \bigcup_{n \geq 0} s_n(A^C) \subset A^{C'}.
   \end{equation*}
  \item\label{es.5} For every $S \in \fS$ there is some $C \in \fC$ such that
   \begin{equation*}
    \bigcup_{n \geq 0} \{ (s_{n+1}(x),s_n(x)) \mid x \in S \} \subset C.
   \end{equation*}
 \end{enumerate}
 Then there is an exact endofunctor on $(\cR^G_f(W,\fZ), h^A)$ as in Proposition \ref{prop_eilenbergswindles}.
 This swindle is natural in the following sense: Let $\fz \colon \fZ_1 \to \fZ_2$ be a morphism of coarse structures.
 If $(s^i_n \colon Z_i \to Z_i)_n$, $i=1,2$, are as above such that $\fz \circ s^1_n = s^2_n \circ \fz$ for all $n$,
 then the induced exact functor $\cR(\fz)$ satisfies the assumptions of Proposition \ref{prop_eilenbergswindles}.
 
 The same holds with $(\cR^G_{fd}(W,\fZ),h^A)$ instead of $(\cR^G_f(W,\fZ),h^A)$.
\end{proposition}
\begin{proof}
 Define a functor $S \colon \cR^G_f(W,\fZ) \to \cR^G_f(W,\fZ)$ as follows.
 Given a controlled retractive space $(Y, \kappa)$ over $W$, consider the infinite coproduct $Y^\infty := \bigvee_{n \geq 0} Y$.
 Then $\cells Y^\infty = \coprod_{n \geq 0} \cells Y$.
 Define a control map $\kappa^\infty \colon \cells Y^\infty \to Z$ by $\kappa^\infty(e) := (s_n \circ \kappa)(e)$ if $e$ is a cell in the $n$-th copy of $Y$.
 Then conditions (\ref{es.1}) and (\ref{es.2}) ensure that $(Y^\infty,\kappa^\infty)$ is a controlled retractive space over $W$.
 If $f \colon Y_1 \to Y_2$ is a controlled morphism, then $\bigvee_{n \geq 0} f \colon Y_1^\infty \to Y_2^\infty$ is again a controlled morphism.
 Moreover, condition (\ref{es.3}) guarantees that $Y^\infty$ is finite if $Y$ is finite.
  
 We define $S(Y) := (Y^\infty,\kappa^\infty)$. We claim that this functor preserves $h^A$-equivalences.
 It suffices to check that for any subcomplex $Y' \subset Y$ which is cofinal away from $A$,
 the subcomplex $S(Y') \subset S(Y)$ is also cofinal away from $A$.
 
 For the purpose of the next paragraph, denote the $n$-th copy of $Y$ by $Y_n$, and use $Y'_n$ in the same way.
 Let $k \in \NN$. Then there is some $C \in \fC$ such that $\kappa^{-1}(Z \setminus A^C) \cap \cells_k Y \subset \cells_k Y'$.
 Let $e \in (\kappa^\infty)^{-1}(Z \setminus s_n(A^C)) \cap \cells_k Y_n$, and let $e'$ be the corresponding 
 $k$-cell in the original copy of $Y$. Since $s_n(\kappa(e')) = \kappa^\infty(e) \notin s_n(A^C)$,
 it follows that $e' \in \cells_k Y'$. Consequently, $e \in \cells_k Y'_n$,
 and we have shown that $(\kappa^\infty)^{-1}(Z \setminus s_n(A^C)) \cap \cells_k Y_n \subset \cells_k Y'_n$.
 Choosing $C' \in \fC$ as in (\ref{es.4}), we have
 \begin{equation*}
   \begin{split}
    (\kappa^\infty)^{-1}(Z \setminus A^{C'}) \cap \cells_k Y^\infty
    &\subset (\kappa^\infty)^{-1}(Z \setminus \bigcup_n s_n(A^C)) \cap \cells_k Y^\infty \\
    &\subset \bigcup_{m \geq 0} \big( (\kappa^\infty)^{-1}(Z \setminus \bigcup_n s_n(A^C)) \cap \cells_k Y_m \big) \\
    &\subset \bigcup_{m \geq 0} \big( (\kappa^\infty)^{-1}(Z \setminus s_m(A^C)) \cap \cells_k Y_m \big) \\
    &\subset \bigcup_{m \geq 0} \cells_k Y'_m = \cells_k (Y')^\infty.
   \end{split}
 \end{equation*}
 So $S(Y') \subset S(Y)$ is also cofinal away from $A$, and it follows that $S$ is an exact functor with respect to the $h^A$-equivalences.
  
 The map $Y \vee Y^\infty \to Y^\infty$ which maps $Y_n$ identically to $Y_{n+1}$ and $Y$ to $Y_0$ is a controlled isomorphism by condition \ref{es.5}.
 It induces a natural isomorphism $\id \vee S \cong S$.
 
 Checking the naturality statement is straightforward.
\end{proof}

\begin{remark}
 Note that the the conditions of Proposition \ref{prop_connectiveeilenbergswindle} hold for $A = \varnothing$ whenever they are satisfied for some $A \subset Z$.
 Hence, the $K$-theory space $\Omega\abs{hS_\bullet\cR^G_f(W,\fZ)}$ is also contractible.
 
 Most of the time, the sequence of maps $(s_n)_n$ is induced by an \emph{infinite shift map}, i.e., 
 a $G$-equivariant function $s \colon Z \to Z$ with the following properties:
 \begin{enumerate}
   \item For every $C \in \fC$ and $S \in \fS$ there is some $C' \in \fC$ such that 
    \begin{equation*}
     \bigcup_n (s \times s)^n(C \cap (S \times S)) \subset C'.
    \end{equation*}
   \item For every $S \in \fS$ there is some $S' \in \fS$ such that $\bigcup_n s^n(S) \subset S'$.
   \item For every $S \in \fS$ and every $B \subset S$ which is locally finite in $Z$, the image $s(B)$ is locally finite in $Z$ and $s^{-1}(x) \cap B$ is finite for all $x \in s(B)$.
    Furthermore, there are for every $z \in Z$ some $n_0$ and an open neighbourhood $U$ of $z$ such that $(s^n)^{-1}(U) = \varnothing$ for all $n \geq n_0$.
   \item For every $C \in \fC$ there exists $C' \in \fC$ such that
    \begin{equation*}
     \bigcup_{n \geq 0} s^n(A^C) \subset A^{C'}.
    \end{equation*}
   \item For every $S \in \fS$ there is some $C \in \fC$ such that
    \begin{equation*}
     \{ (s(x),x) \mid x \in \bigcup_n s^n(S) \} \subset C.
    \end{equation*}
 \end{enumerate}
 In this case, the proposition applies with $s_n := s^n$, and the corresponding naturality statement applies
 whenever we have two infinite shift maps $s_1$ and $s_2$ as well as a morphism of coarse structures $\fz$ such that
 $\fz \circ s_1 = s_2 \circ \fz$.
\end{remark}

\section{Non-connective $A$-theory spectra}\label{sec_nonconnectiveA}
We are now ready to put the results of the previous sections to use. Namely, we define (potentially) non-connective deloopings of the $K$-theory spaces
of controlled retractive spaces. The resulting spectra are insensitive to specific choices of finiteness conditions, and the main results of Section \ref{sec_comparisonthms} simplify accordingly.

For linear $K$-theory, such deloopings have been defined previously by Pedersen--Weibel \cite{PW1985}. Vogell \cite{Vogell1990} adopted this approach to define a non-connective delooping of $A(X)$.

\begin{definition}\label{def:pullback-coarse-structure}
 Suppose that $Z = Z_1 \times Z_2$, and that $(Z_1, \fC, \fS)$ is a coarse structure. Let $p \colon Z \to Z_1$ be the projection map. Then we define a coarse structure
 $(Z, p^*\fC, p^*\fS)$ by setting $p^*\fC := \{ (p \times p)^{-1}(C) \mid C \in \fC \}$ and $p^*\fS := \{ p^{-1}(S) \mid S \in \fS \}$.
\end{definition}

Let $\fZ = (Z, \fC, \fS)$ be a coarse structure. Let $p_n \colon \RR^n \times Z \to \RR^n$ and $p_Z \colon \RR^n \times Z \to Z$ denote the respective projection maps. 
Consider the bounded coarse structure $\fB(\RR^n) = (\RR^n, \fC_{bdd}(\RR^n), \fS_{triv}(\RR^n))$ from Example \ref{example:control-structures}.

\begin{definition}\label{def:delooping-coarse-structure}
 For $n \in \NN$ define the coarse structure $\fZ(n) = (\RR^n \times Z, \fC(n), \fS(n))$ as follows:
 A set $C \subset (\RR^n \times Z)^2$ is in $\fC(n)$ if and only if:
   \begin{enumerate}
   \item $C$ is symmetric, $G$-invariant and contains the diagonal.
   \item There is a $C' \in p_n^*\fC_{bdd}(\RR^n)$ such that $C \subseteq C'$.
   \item For all $K \subset \RR^n$ compact, there is a $C'' \in p_Z^*\fC$ such that
     \begin{equation*}
      C \cap ((K \times Z) \times (K \times Z)) \subset C''.
     \end{equation*}
   \end{enumerate}
 Set $\fS(n) := p_Z^*\fS$.
\end{definition}

Consider for all $n$ also the restricted coarse structures
\begin{equation*}
 \begin{split}
   \fZ(n+1)^+ &:= \fZ(n+1) \cap (\RR^n \times \RR_{\geq 0} \times Z), \\
   \fZ(n+1)^- &:= \fZ(n+1) \cap (\RR^n \times \RR_{\leq 0} \times Z). \\
 \end{split}
\end{equation*}
Note that $\fZ(n+1) \cap (\RR^n \times \{ 0 \} \times Z) = \fZ(n)$.

Let $A \subset Z$ be a $G$-invariant subset. The obvious inclusion maps give rise to a commutative square
\begin{equation}\label{diag:deloopingsquare}
 \commsquare{h^A S_\bullet \cR^G_f(W,\fZ(n))}{h^A S_\bullet \cR^G_f(W,\fZ(n+1)^+)}{h^A S_\bullet \cR^G_f(W,\fZ(n+1)^-)}{h^A S_\bullet \cR^G_f(W,\fZ(n+1))}{}{}{}{}
\end{equation}
Using the results of Section \ref{subsec:vanishing}, the top right and bottom left corners of this square are contractible since they admit infinite shift maps
$(\vec{x},x_{n+1},z) \mapsto (\vec{x},x_{n+1} \pm 1,z)$. This provides us with structure maps for a spectrum
\begin{equation*}
 \KK^{-\infty}(\cR^G_f(W,\fZ),h^A) := \big\{ K(\cR^G_f(W,\fZ(n)), h^A) \big\}_n.
\end{equation*}
These are the algebraic $K$-theory spectra we use for our main results. It follows from Propositions \ref{prop:functoriality:Z}, \ref{prop_eilenbergswindles} and \ref{prop_connectiveeilenbergswindle} that the construction of this spectrum is natural in $\fZ$.

\begin{remark}\label{rem:different-deloopings}
 Definition \ref{def:delooping-coarse-structure} is more involved than one might expect. The coarse structure
 \begin{equation*}
  \fZ[n] := (\RR^n \times Z, p_n^*\fC_{bdd}(\RR^n) \Cap p_Z^*\fC, p_Z^*\fS)
 \end{equation*}
 might appear to be a more intuitive choice. There is a canonical inclusion functor $\cR^G_f(W,\fZ[n]) \to \cR^G_f(W,\fZ(n))$ which induces an isomorphism on homotopy groups in sufficiently high degrees, using $\fZ[0] = \fZ(0)$ and Proposition \ref{prop_connectivitystructuremaps} below. We conjecture that this map is in fact a weak equivalence.
 
 The difference between the coarse structures $\fZ(n)$ and $\fZ[n]$ is analogous to the linear situation (cf.~\cite{PW-homology}): Take for example categories $\cC_X(R)$ of bounded morphisms over a metric space. Then $\fZ[n]$ corresponds to the category $\cC_{\RR^n \times X}(R)$, while $\fZ(n)$ corresponds to $\cC_{\RR^n}(\cC_X(R))$. The inclusion functor $\cC_{\RR^n \times X}(R) \to \cC_{\RR^n}(\cC_X(R))$ always induces an equivalence of algebraic $K$-theory spectra: Apply non-connective algebraic $K$-theory to the inclusion map and prove that both sides are equivalent to the spectrum $\Omega^n\KK^{-\infty}(\cC_X(R))$.
\end{remark}

\begin{proposition}\label{prop_connectivitystructuremaps}
 \
 \begin{enumerate}
  \item The structure maps of the spectrum $\KK^{-\infty}(\cR^G_f(W,\fZ),h)$ induce isomorphisms on $\pi_i$ for $i \geq 1$.
  \item The structure maps of the spectrum $\KK^{-\infty}(\cR^G_f(W,\fZ),h^A)$ induce isomorphisms on $\pi_i$ for $i \geq 2$.
 \end{enumerate} 
\end{proposition}
\begin{proof}
 By Theorem \ref{thm_connectivemayervietoris}, there is a homotopy pullback square
 \begin{equation}\label{diag:deloopingsquare-fd}
  \commsquare{h S_\bullet \cR^G_{fd,F'}(W,\fZ(n))}{h S_\bullet \cR^G_{fd,K'}(W,\fZ(n+1)^+)}{h S_\bullet \cR^G_{fd,F}(W,\fZ(n+1)^-)}{h S_\bullet \cR^G_{fd,K}(W,\fZ(n+1))}{}{}{}{}
 \end{equation}
 There is a transformation from square \eqref{diag:deloopingsquare} to \eqref{diag:deloopingsquare-fd} induced by inclusion functors. By Thomason cofinality, this transformation is a weak equivalence on the top right and bottom right corners. Therefore, $hS_\bullet\cR^G_{fd,K'}(W,\fZ(n+1)^+)$ is weakly contractible. In particular, its $K_0$ is trivial, so $\cR^G_{fd,F'}(W,\fZ(n)) = \cR^G_{fd}(W,\fZ(n))$.
 
 We claim that $K_0(\cR^G_f(W,\fZ(n+1)),h) = 0$. Since we can filter any object by its skeleta, and suspension corresponds to taking inverses in $K_0$, the class of any object in $\cR^G_f(W,\fZ(n+1))$ equals its $K$-theoretic Euler characteristic, i.e., it equals an alternating sum of classes of $0$-dimensional objects. The same argument as in the linear case \cite[Corollary 1.3.1]{PW1985} now shows that the $K_0$-class of every $0$-dimensional object is trivial.
 
 From $K_0(\cR^G_f(W,\fZ(n+1)),h) = 0$ it follows that $\cR^G_{fd,F}(W,\fZ(n+1)^-) = \cR^G_{fd}(W,\fZ(n+1)^-)$. Since $\fZ(n+1)^-$ admits an infinite shift map, $hS_\bullet\cR^G_{fd}(W,\fZ(n+1)^-)$ is weakly contractible by Section \ref{subsec:vanishing}. We already know that $hS_\bullet\cR^G_f(W,\fZ(n+1)^-)$ is weakly contractible, so the transformation from \eqref{diag:deloopingsquare} to \eqref{diag:deloopingsquare-fd} is also a weak equivalence on the bottom left corner.
 
 As the square \eqref{diag:deloopingsquare-fd} is a homotopy pullback in which the bottom left and top right corners are weakly contractible, we get a weak equivalence 
 \begin{equation}\label{eq:delooping-equivalence}
   \abs{h S_\bullet \cR^G_{fd}(W,\fZ(n))} \xrightarrow{\sim} \Omega\abs{h S_\bullet \cR^G_{fd,K}(W,\fZ(n+1))}. 
 \end{equation}
 By Proposition \ref{prop_cofinality}, the map $\abs{hS_\bullet\cR^G_f(W,\fZ(n))} \to \abs{h S_\bullet \cR^G_{fd}(W,\fZ(n))}$ induces an isomorphism on $\pi_i$ for $i \geq 2$. Hence, the structure map $K(\cR^G_f(W,\fZ(n)),h) \to \Omega K(\cR^G_f(W,\fZ(n+1)),h)$ is an isomorphism on $\pi_i$ for $i \geq 1$.
 
 The structure map $K(\cR^G_f(W,\fZ(n)),h^A) \to \Omega K(\cR^G_f(W,\fZ(n+1)),h^A)$ sits in a map of homotopy fiber sequences arising from Theorem \ref{thm_connectivefibresequence}. The second assertion of the proposition follows from the first assertion and a five-lemma argument.
\end{proof}

\begin{remark}\label{remark:nonconnective-finite-fd}
 We can also define a non-connective spectrum $\KK^{-\infty}(\cR^G_{fd}(W,\fZ),h^A)$ using the finitely dominated objects. The natural maps $K(\cR^G_f(W,\fZ(n)),h^A) \to K(\cR^G_{fd}(W,\fZ(n)),h^A)$ are isomorphisms on $\pi_i$ for $i \geq 1$, hence the induced map $\KK^{-\infty}(\cR^G_f(W,\fZ),h^A) \to \KK^{-\infty}(\cR^G_{fd}(W,\fZ),h^A)$ is a stable equivalence of spectra by Proposition \ref{prop_connectivitystructuremaps}.
\end{remark}

For convenience, we record the non-connective versions of the main results of the previous section.

\begin{theorem}\label{thm_fibresequence}
 Let $\fZ$ be a coarse structure and let $A \subset Z$ be a closed, $G$-invariant subset such that $\fZ$ is $G$-proper with respect to $A$.
 Then the inclusion functors induce a homotopy fiber sequence
 \begin{equation*}
  \KK^{-\infty}(\cR^G_f(W,\fZ \cap A),h) \to \KK^{-\infty}(\cR^G_f(W,\fZ), h) \to \KK^{-\infty}(\cR^G_f(W,\fZ),h^A).
 \end{equation*}
\end{theorem}
\begin{proof}
 This is Theorem \ref{thm_connectivefibresequence} together with Proposition \ref{prop:thickenings:restrictions} and Remark \ref{remark:nonconnective-finite-fd}.
\end{proof}

\begin{theorem}[Coarse Mayer--Vietoris Theorem]\label{thm_mayervietoris}
 Let $(\fZ,A_1,A_2)$ be a coarsely excisive triple, and assume that $\fZ$ is $G$-proper with respect to $A_1$, $A_2$ and $A_1 \cap A_2$.
 Then the obvious inclusion maps give rise to a homotopy pullback square of spectra
 \begin{equation*}
  \commsquare{\KK^{-\infty}(\cR^G_f(W,\fZ \cap (A_1 \cap A_2)),h)}{\KK^{-\infty}(\cR^G_f(W,\fZ \cap A_1),h)}{\KK^{-\infty}(\cR^G_f(W,\fZ \cap A_2),h)}{\KK^{-\infty}(\cR^G_f(W,\fZ),h)}{}{}{}{}
 \end{equation*}
\end{theorem}
\begin{proof}
 This is Theorem \ref{thm_connectivemayervietoris} together with Remark \ref{remark:nonconnective-finite-fd}.
\end{proof}

\begin{theorem}[Eilenberg swindle]\label{thm_eilenbergswindle}
 Let $\fZ$ be a coarse structure and let $A \subset Z$ be a $G$-invariant subset.
 Suppose that there is a sequence of $G$-equivariant functions $(s_n \colon Z \to Z)_n$ as in Proposition \ref{prop_connectiveeilenbergswindle}.
 
 Then $\KK^{-\infty}(\cR^G_f(W,\fZ),h^A)$ is weakly contractible.
\end{theorem}
\begin{proof}
 This follows from Section \ref{subsec:vanishing}.
\end{proof}

\section{The Davis--L\"uck assembly map}\label{sec_assembly}
We can now translate the model of the assembly map given in \cite{BFJR2004} to $A$-theory.
Assume from now on that $G$ is a countable discrete group.

\begin{definition}\label{def_continuouscontrol}
 Let $X$ be a $G$-CW-complex and $M$ a metric space with free, isometric $G$-action.
  
 Define the coarse structure $\JJ(M,X) = (M \times X \times [1,\infty[, \fC(M,X), \fS(M,X))$ as follows.
 Let $p_M$, $p_{M \times X}$, $p_{X \times [1,\infty[}$ and $p_{[1,\infty[}$ denote the projection maps
 from $M \times X \times [1,\infty[$ to the factor indicated by the index of $p$.
 \begin{enumerate}
  \item Set $\fC(M,X) := p_M^*\fB(M) \Cap p_{X \times [1,\infty[}^*\fC_{G\text{cc}}(X)$.
  \item Set $\fS(M,X) := p_{M \times X}^*\fS_{G\text{-cpt}}(M \times X)$.
 \end{enumerate}
 The bounded coarse structure, $G$-compact support condition and $G$-continuous control condition have been defined in Example \ref{example:control-structures}.
\end{definition}

One particular instance of this definition is the case where $M = G$, equipped with a left invariant and proper metric, ``proper'' meaning that every ball of finite radius is finite.
Such metrics exist \cite[Proposition 1.3]{DS2006}; if $G$ is finitely generated, we can pick a word metric. Whenever $d$ and $d'$ are two left invariant, proper metrics on $G$, the identity map $\id \colon (G,d) \to (G,d')$ is a coarse equivalence by \cite[Proposition 1.1]{DS2006}.  
In particular, every $R$-ball with respect to $d$ is contained in some $R'$-ball with respect to $d'$, and vice versa. 
Hence, the bounded control condition on $G$ is independent of the choice of left invariant, proper metric, and we can suppress the metric in our notation.

\begin{definition}
  We abbreviate $\JJ(X) := \JJ(G,X)$.
\end{definition}

When considering $\cR^G_f(W,\JJ(M,X))$, we denote the class of weak equivalences $h^{M \times X \times \{1\}}$
by $h^\infty$. Observe also that for a $G$-invariant subcomplex $A \subset X$, we have $\JJ(M,X) \cap (M \times A \times [1,\infty[) = \JJ(M,A)$.
Finally, we note that $\JJ(M,X)$ is $G$-proper with respect to subspaces of the form $M \times A \times [1,\infty[$ for $A \subset X$ a $G$-invariant subcomplex.

\begin{definition}\label{def_todsequence}
 Let us introduce the following shorthands:
 \begin{enumerate}
  \item $\TT(G,W,X) := \KK^{-\infty}\big(\cR^G_f(W,\JJ(X)\cap (G \times X \times \{1\})), h\big)$,
  \item $\FF(G,W,X) := \KK^{-\infty}\big(\cR^G_f(W, \JJ(X)), h\big)$,
  \item $\DD(G,W,X) := \KK^{-\infty}\big(\cR^G_f(W, \JJ(X)), h^\infty\big)$.
 \end{enumerate}
\end{definition}
As a consequence of Theorem \ref{thm_fibresequence}, these spectra fit into a natural homotopy fibre sequence
\begin{equation}\label{tod-sequence}
 \TT(G,W,X) \to \FF(G,W,X) \to \DD(G,W,X).
\end{equation}

\begin{definition}\label{def_ghomologytheory}
 An \emph{(unreduced) $G$-homology theory} is a functor $\HH$ from the category of
 $G$-CW-complexes to the category of spectra such that the following holds:
 \begin{enumerate}
  \item\label{ghomology.invariance} Every $G$-equivariant homotopy equivalence $f \colon X_1 \xrightarrow{\sim} X_2$
   induces a weak equivalence $\HH(f) \colon \HH(X_1) \to \HH(X_2)$.
  \item\label{ghomology.mayervietoris} Every homotopy pushout square of $G$-CW-complexes induces a homotopy pullback square of spectra upon application of $\HH(-)$.
  \item\label{ghomology.colimits} If $X = \colim_i X_i$ is a directed colimit, the natural map $\hocolim_i \HH(X_i) \to \HH(X)$ is a weak equivalence.
 \end{enumerate}
\end{definition}

\begin{remark}
 Observe that any unreduced $G$-homology theory in the sense of Definition \ref{def_ghomologytheory} automatically respects
 finite coproducts because
 \begin{equation*}
  \commsquare{\varnothing}{X_1}{X_2}{X_1 \coprod X_2}{}{}{}{}
 \end{equation*}
 is a homotopy pushout square. In view of the direct limit axiom \ref{def_ghomologytheory}.\eqref{ghomology.colimits}, we conclude that any unreduced $G$-homology theory
 commutes with arbitrary coproducts.
\end{remark}

\begin{theorem}\label{thm_a-homology}\
 \begin{enumerate}
  \item\label{a-homology.coefficients} The projection $X \to G/G$ induces a weak equivalence
   \begin{equation*}
    \TT(G,W,X) \xrightarrow{\sim} \TT(G,W,G/G)
   \end{equation*}
   for every $G$-CW-complex $X$.
  \item\label{a-homology.homology} The assignment $X \mapsto \DD(G,W,X)$ is an unreduced $G$-equivariant homology theory.
  \item\label{lem_connectingmap} The connecting map $\Omega\DD(G,W,G/G) \to \TT(G,W,G/G)$ is a weak equivalence.
 \end{enumerate}
\end{theorem}
\begin{proof}
 For part \eqref{a-homology.coefficients}, consider the functor
 \begin{equation*}
  p \colon \cR^G_f(W,\JJ(X)(n)\cap (G \times X \times \{1\})(n)) \to \cR^G_f(W,\JJ(G/G)\cap (G \times G/G \times \{1\})(n))
 \end{equation*}
 induced by the projection map $X \to G/G$; it is well-defined because of the $G$-compact support condition on $G \times X$. Let $(Y \leftrightarrows W, \kappa)$ be any object from the category on the right hand side.
 Then any choice of a point $x \in X$ induces a control map
 \begin{equation*}
  \tilde{\kappa} \colon \cells Y  \to \RR^n \times G \times X, \quad e \mapsto (\kappa_{\RR^n}(e), \kappa_G(e),\kappa_G(e) \cdot x)
 \end{equation*}
 which turns $Y$ into an object $\tilde{Y}$ of the left hand side. This construction provides an inverse to $p$, showing that $p$ is an exact equivalence of Waldhausen categories.
 
 For the second part of the theorem, observe first that $X \mapsto \cR^G_f(W,\JJ(X))$
 is indeed a functor on $G$-CW-complexes; this follows from Proposition \ref{prop:functoriality:Z} using Lemma 3.3 from \cite{BFJR2004}
 and the $G$-compact support condition.
 Also, due to the $G$-compact support condition, in conjunction with the fact that algebraic $K$-theory commutes with directed colimits,
 we immediately obtain the direct limit axiom (\ref{ghomology.colimits}). Hence, it suffices to consider only cocompact $G$-CW-complexes.
 The remainder of the proof is formally the same as in \cite[\S 5]{BFJR2004}; note that in the proof of property \ref{def_ghomologytheory}.\eqref{ghomology.mayervietoris},
 the special case of a coproduct is missing in \cite{BFJR2004} and has to be treated separately.
 For more details, see also \cite[Section 7.2]{thesis-ullmann}.
 
 For the last part of the theorem, consider $\cR^G_{f}(W,\JJ(G/G))$.
 The map $s \colon G \times G/G \times [1,\infty[$ given by $(g,G,t) \mapsto (g,G,t+1)$ is an infinite shift map, so $\FF(G,W,G/G)$
 is weakly contractible by Theorem \ref{thm_eilenbergswindle}. The claim follows.
\end{proof}

Let $Or(G)$ denote the \emph{orbit category} of $G$, i.e., the category of left $G$-sets $G/H$ and $G$-equivariant maps between them.

Let $V$ be a topological space. Then $\cR_f(V)$, the category of finite retractive spaces over
$V$, is isomorphic to the category $\cR^{\{1\}}_f(V,\fT(*))$, where $\fT(*)$ is the trivial coarse structure over a point from Example \ref{example:control-structures}. The results of Section \ref{sec_nonconnectiveA}
provide us with a spectrum
\begin{equation*}
 \Aa^{-\infty}(V) := \KK^{-\infty}(\cR^{\{1\}}_f(V,\fT(*)),h)
\end{equation*}
which is a (potentially non-connective) delooping of $A(V)$. We call this the \emph{non-connective algebraic $K$-theory spectrum of $V$}.
Given any $G$-space $W$, we may therefore define an $Or(G)$-spectrum $\Aa^{-\infty}_W$ by setting
\begin{equation*}
 \Aa^{-\infty}_W(G/H) := \Aa^{-\infty}(W^{op} \times_G G/H) \cong \Aa^{-\infty}(H \backslash W),
\end{equation*}
where $W^{op}$ denotes the space $W$ equipped with the right action of $G$ induced by the original left action via $w \cdot g := g^{-1}w$.

\begin{theorem}\label{thm_coefficients}
 Let $W$ be a free $G$-CW-complex. Then there is a zig-zag of equivalences of $Or(G)$-spectra
 \begin{equation*}
  \Omega\DD(G,W,-) \simeq \Aa^{-\infty}_W(-).
 \end{equation*}
\end{theorem}
\smallskip

With the exception of Corollary \ref{cor_identificationassemblymap} below, the proof of this theorem occupies the rest of this section.

Consider $\DD(G,W,G/H)$ for some $G/H \in Or(G)$. Define a coarse structure $\JJ^{dis}(G/H) = (G \times G/H \times [1,\infty[, \fC^{dis}(G,G/H), \fS(G,G/H))$, where $\fC^{dis}(G,G/H)$ is the collection of all $C \in \fC(G,G/H)$, such that $\gamma H = \gamma' H$ for all $((g,\gamma H, t),(g',\gamma' H, t')) \in C$.

\begin{lemma}\label{lemma:discrete-control}
 For all $n$, the natural inclusion functor $\cR^G_f(W,\JJ^{dis}(G/H)(n)) \hookrightarrow \cR^G_f(W,\JJ(G/H)(n))$ induces an equivalence in $K$-theory with respect to the $h^\infty$-equivalences.
\end{lemma}
\begin{proof}
 Let $f \colon Y_1 \to Y_2$ be an arbitrary morphism in $\cR^G_f(W,\JJ(G/H)(n))$. Let $C$ be a control condition witnessing that $f$ is a controlled map. For each closed ball $B_R \subset \RR^n$, $C \cap (B_R \times G \times G/H \times [1,\infty[)^2$ satisfies the continuous control condition. Therefore, there is some $t_0 > 1$ such that $((x,g,\gamma H,t),(x',g',\gamma' H, t')) \in C$ implies $\gamma H = \gamma' H$ whenever $x,x' \in B_R$ and $t,t' \geq t_0$. Since we require bounded control over $\RR^n$, and since the $G$-continuous control condition includes bounded control over $[1,\infty[$, there exists some cofinal subcomplex $Y_1' \subset Y_1$ away from $\RR^n \times G \times G/H \times \{1\}$ such that $f|_{Y_1'}$ satisfies a control condition in $\fC^{dis}(G,G/H)$.
 
 We want to prove the Approximation Property. For the first part, let $f \colon Y_1 \to Y_2$ be a morphism in $\cR^G_f(W,\JJ^{dis}(G/H)(n))$ which is an $h^\infty$-equivalence in $\cR^G_f(W,\JJ(G/H)(n))$. If $g$ is an $h^\infty$-inverse to $f$, we can restrict it to a suitable cofinal subcomplex such that its restriction satisfies a control condition in $\fC^{dis}(G,G/H)$ by the previous paragraph. The same works for homotopies. This shows the first part of the Approximation Property.
 
 For the second part, let $f \colon Y_1 \to Y_2$ be a morphism in $\cR^G_f(W,\JJ(G/H))$, where $Y_1$ is an object in $\cR^G_f(W,\JJ^{dis}(G/H))$. Again by the first paragraph, there is some cofinal subcomplex $Y_2' \subset Y_2$ which satisfies a control condition in $\fC^{dis}(G,G/H)$. Then there exists a cofinal subcomplex $Y_1' \subset Y_1$ such that $f|_{Y_1'}$ maps into $Y_2'$. Let $Y$ be the pushout of $Y_1 \leftarrowtail Y_1' \xrightarrow{f|_{Y_1'}} Y_2'$. Then $Y$ is an object in $\cR^G_f(W,\JJ^{dis}(G/H)(n))$ and the canonical morphism $Y_2' \rightarrowtail Y$ is an $h^\infty$-equivalence. Hence, the morphism $Y \to Y_2$ induced by the universal property of the pushout is also an $h^\infty$-equivalence by Saturation. This proves the second part of the Approximation Property. 
\end{proof}

Defining $\DD'(G,W,G/H) := \KK^{-\infty}(\cR^G_f(W,\JJ^{dis}(G/H)),h^\infty)$, Lemma \ref{lemma:discrete-control} states that the natural map $\DD'(G,W,G/H) \to \DD(G,W,G/H)$ is a weak equivalence.
Observe that, by considering $\JJ^{dis}(G/H)$, we have effectively eliminated the $G$-continuous control condition. It has been replaced by bounded control over $[1,\infty[$ together with discrete control over $G/H$.

\begin{lemma}\label{lemma:coefficients-deloop}
 For all $n \geq 1$, there is a zig-zag of exact functors
 \begin{equation*}
  (\cR^G_f(W,\JJ^{dis}(G/H)(n-1)),h^\infty) \to \dots \leftarrow (\cR^G_f(W, (\JJ^{dis}(G/H) \cap (G \times G/H \times \{1\}))(n),h)
 \end{equation*}
 which induces equivalences in $K$-theory and which is natural in $G/H$.
\end{lemma}
\begin{proof}
 Recall the temporary notation $\fZ[n]$ we introduced in Remark \ref{rem:different-deloopings}. We only need to use
 \begin{equation*}
  \fZ[1] := (\RR \times Z, p_\RR^*\fC_{bdd}(\RR) \Cap p_Z^*\fC, p_Z^*\fS).
 \end{equation*}
 In analogy to the delooping construction we discussed in Section \ref{sec_nonconnectiveA}, we also use coarse structures $\fZ[1]^+$ and $\fZ[1]^-$.
 
 For the purpose of this proof, define
 \begin{equation*}
  \JJ^{dis}(G/H)_1 := \JJ^{dis}(G/H) \cap (G \times G/H \times \{1\}).
 \end{equation*}
 The underlying space of the coarse structure $\JJ^{dis}(G/H)(n-1)$ is $\RR^{n-1} \times G \times G/H \times [1,\infty[$. The obvious isometry $[1,\infty[ \cong [0,\infty[$ induces a homeomorphism $\RR^{n-1} \times G \times G/H \times [1,\infty[ \cong \RR^{n-1} \times [0,\infty[ \times G \times G/H$. This homeomorphism gives rise to an isomorphism of Waldhausen categories
 \begin{equation}\label{eq:coefficients-deloop-1}
 \cR^G_f(W,\JJ^{dis}(G/H)(n-1)) \cong \cR^G_f(W,\JJ^{dis}(G/H)_1[1]^+(n-1)).
 \end{equation}
 Under this isomorphism, the class of $h^\infty$-equivalences corresponds to the homotopy equivalences $h^0$ away from $\RR^{n-1} \times \{0\} \times G \times G/H \times \{1\}$. As in the proof of Theorem \ref{thm_connectivemayervietoris}, we obtain a weak equivalence
 \begin{equation}\label{eq:coefficients-deloop-2}
  h^0S_\bullet\cR^G_f(W,\JJ^{dis}(G/H)_1[1]^+(n-1)) \xrightarrow{\sim} h^-S_\bullet\cR^G_f(W,\JJ^{dis}(G/H)_1[1](n-1)),
 \end{equation}
 where in the second term $h^-$ refers to the class of homotopy equivalences away from $\RR^{n-1} \times \RR_{\leq 0} \times G \times G/H \times \{1\}$.
 
 There is a natural, exact inclusion functor $\cR^G_f(W,\JJ^{dis}(G/H)_1[1](n-1)) \hookrightarrow \cR^G_f(W,\JJ^{dis}(G/H)_1(1)(n-1))$. Analogous to~\eqref{eq:delooping-equivalence} in the proof of Proposition \ref{prop_connectivitystructuremaps}, there are weak equivalences
 \begin{equation*}
  \begin{split}
   \abs{hS_\bullet\cR^G_{fd}(W,\JJ^{dis}(G/H)_1(n-1))} &\xrightarrow{\sim} \Omega\abs{hS_\bullet\cR^G_f(W,\JJ^{dis}(G/H)_1[1](n-1))}, \\
   \abs{hS_\bullet\cR^G_{fd}(W,\JJ^{dis}(G/H)_1(n-1))} &\xrightarrow{\sim} \Omega\abs{hS_\bullet\cR^G_f(W,\JJ^{dis}(G/H)_1(1)(n-1))}.
  \end{split}
 \end{equation*}
 Since the inclusion maps
 \begin{equation*}
  \begin{split}
   hS_\bullet\cR^G_f(W,\JJ^{dis}(G/H)_1[1](n-1)) &\to h^-S_\bullet\cR^G_f(W,\JJ^{dis}(G/H)_1[1](n-1)), \\
   hS_\bullet\cR^G_f(W,\JJ^{dis}(G/H)_1(1)(n-1)) &\to h^-S_\bullet\cR^G_f(W,\JJ^{dis}(G/H)_1(1)(n-1)) \\
  \end{split}
 \end{equation*}
 are weak equivalences, too, we conclude that the map
 \begin{equation}\label{eq:coefficients-deloop-3}
  h^-S_\bullet\cR^G_f(W,\JJ^{dis}(G/H)_1[1](n-1)) \to h^-S_\bullet\cR^G_f(W,\JJ^{dis}(G/H)_1(1)(n-1))
 \end{equation}
 is also a weak equivalence. There is another exact inclusion functor
 \begin{equation}\label{eq:coefficients-deloop-4}
  \cR^G_f(W,\JJ^{dis}(G/H)_1(n)) \hookrightarrow \cR^G_f(W,\JJ^{dis}(G/H)_1(1)(n-1))
 \end{equation}
 which induces an equivalence on $K$-theory with respect to the $h^-$-equivalences for similar reasons.
 The desired zig-zag is then formed by the equivalences arising from \eqref{eq:coefficients-deloop-1}, \eqref{eq:coefficients-deloop-2}, \eqref{eq:coefficients-deloop-3} and \eqref{eq:coefficients-deloop-4}.
\end{proof}

Since there is a weak equivalence from the shifted spectrum $\{ \DD'(G,W,G/H)_{n-1} \}_n$, where we set $\DD'(G,W,G/H)_{-1} = *$, to $\Omega\DD'(G,W,G/H)$, Lemma \ref{lemma:coefficients-deloop} provides us with a zig-zag of natural weak equivalences
\begin{equation}\label{eq:coefficients-deloop-zigzag}
 \Omega\DD'(G,W,G/H) \simeq \KK^{-\infty}(\cR^G_f(W,\JJ^{dis}(G/H) \cap (G \times G/H \times \{1\})),h).
\end{equation}
In order to prove Theorem \ref{thm_coefficients}, it is therefore sufficient to identify the latter $Or(G)$-spectrum.

\begin{lemma}\label{lemma:coefficients-induction-quotient}
 There is a zig-zag of weak equivalence of $Or(G)$-spectra
 \begin{equation*}
  \KK^{-\infty}(\cR^G_f(W,\JJ^{dis}(-) \cap (G \times - \times \{1\})),h) \simeq \Aa^{-\infty}_W(-).
 \end{equation*}
\end{lemma}
\begin{proof}
 Let $\hat{\cR} \subset \cR^G_f(W,(\JJ^{dis}(G/H) \cap (G \times G/H \times \{1\})))$ denote the full Waldhausen subcategory of those objects $(Y,\kappa)$ for which the set of cells $\kappa^{-1}(\{1_G\} \times \{H\} \times \{1\})$ intersects every $G$-orbit of cells.
 
 We claim that the inclusion functor $\hat{\cR} \hookrightarrow \cR^G_f(W,(\JJ^{dis}(G/H) \cap (G \times G/H \times \{1\})))$ is an exact equivalence. What we need to show is that every object is isomorphic to some object in $\hat{\cR}$. Let $(Y,\kappa) \in \cR^G_f(W,(\JJ^{dis}(G/H) \cap (G \times G/H \times \{1\})))$. Due to the $G$-compact support condition, we can find a set of representatives $R$ for the $G$-cells of $Y$ such that $\kappa(R) \subset F_1 \times F_2 \times \{1\}$ for some finite sets $F_1 \subset G$, $F_2 \subset G/H$. Multiplying by appropriate group elements, we can assume without loss of generality that $\kappa(R) \subset F \times \{ H \} \times \{1\}$ for some finite set $F \subset G$. Let $c \colon F \to \{ 1_G \}$ be the unique function. By requiring $G$-equivariance, $c$ induces a $G$-equivariant function $\kappa_c \colon \cells Y \to G \times G/H \times \{1\}$. Since there are only finitely many equivariant cells in $Y$, the labelled $G$-CW-complex $(Y,\kappa_c)$ satisfies bounded control over $G$. By construction, $(Y,\kappa_c)$ is an object of $\hat{\cR}$. The identity map on $Y$ defines an isomorphism $(Y,\kappa) \cong (Y,\kappa_c)$. This proves that the inclusion functor is an equivalence.

 Next, we define an exact functor $Q \colon \hat{\cR} \to \cR_f(W^{op} \times_G G/H, \fT(*))$. Let $(Y,\kappa) \in \hat{\cR}$. Define $Y_H \subset Y$ to be the $H$-invariant subcomplex given by the set of cells $\kappa^{-1}(H \times \{H\} \times \{1\})$. Then $H \backslash Y_H$ is naturally a retractive space over $H \backslash \res_H^G W \cong W^{op} \times_G G/H$. Set $Q(Y) := H \backslash Y_H$.
 
 We claim that this functor is also an equivalence of Waldhausen categories. The following argument is similar to \cite[Lemma 2.1.3]{Waldhausen1985}. 
 
 Let $(X, \kappa) \in \cR_f(W^{op} \times_G G/H, \fT(*))$. Let $\pi \colon W \times G/H \to W^{op} \times_G G/H$ denote the $G$-equivariant map sending $(w,gH)$ to $(g^{-1}w,H)$. By pulling back along $\pi$, we obtain a retractive space $\tilde{X}$ relative $W \times G/H$. Define $\Phi(X)$ as the pushout
 \begin{equation*}
  \begin{tikzpicture}
  \matrix (m) [matrix of math nodes, column sep=3.5em, row sep=2em, text depth=.5em, text height=1em]
   {W \times G/H & \tilde{X} \\
    W            & \Phi(X)   \\};
  \path[->]
   (m-1-1) edge node[above]{} (m-1-2) edge (m-2-1)
   (m-1-2) edge (m-2-2)
   (m-2-1) edge node[above]{$\Phi(s)$} (m-2-2);
  \end{tikzpicture}
 \end{equation*}
 The retraction of $\tilde{X}$ induces a retraction $\Phi(r)$ on $\Phi(X)$. Note that there is a canonical bijection $\cells \Phi(X) \xrightarrow{\sim} \cells \tilde{X}$.
 The projection map $\tilde{X} \to W \times G/H \to G/H$ induces a $G$-equivariant function $\tilde{\kappa} \colon \cells \tilde{X} \to G/H$ with the property that, if $e, e'$ are cells in $\tilde{X}$ such that $e' \subset \langle e \rangle$, then $\tilde{\kappa}(e) = \tilde{\kappa}(e')$.
 
 Choose a set of representatives $S$ for the $G$-orbits of cells in $\Phi(X)$ such that $\tilde{\kappa}(e) = H$ for all $e \in S$. Define the $G$-equivariant function
 \begin{equation*}
  \Phi(\kappa) \colon \cells \Phi(X) \to G \times G/H \times \{1\}
 \end{equation*}
 by $\Phi(\kappa)(e) := (1_G,H,1)$ for all $e \in S$ and extending $G$-equivariantly. This turns $\Phi(X)$ into an object $(\Phi(X),\Phi(\kappa)) \in \hat{\cR}$.  As $W$ is a free $G$-CW complex, $\Phi(Q(Y))$ is canonically isomorphic to $Y$.  Since $Q(\Phi(X))$ is canonically isomorphic to $X$, this shows that $Q$ is essentially surjective and fully faithful. This finishes the proof of the lemma.
\end{proof}

Combining Lemmas \ref{lemma:discrete-control} and \ref{lemma:coefficients-induction-quotient} with the zig-zag \eqref{eq:coefficients-deloop-zigzag}, we obtain the zig-zag of weak equivalences of $Or(G)$-spectra
\begin{equation*}
 \begin{split}
  \Omega\DD(G,W,-)
  &\simeq \Omega\DD(G,W,-) \\
  &\simeq \KK^{-\infty}(\cR^G_f(W,\JJ^{dis}(-) \cap (G \times - \times \{1\})),h) \\
  &\simeq \Aa^{-\infty}_W(-)
 \end{split}
\end{equation*}
whose existence we claimed in Theorem \ref{thm_coefficients}.

As explained in \cite{DL1998}, any $Or(G)$-spectrum $\EE$ gives rise to a $G$-homology theory
$\HH^G(-;\EE)$. By considering the map induced by the projection $X \to G/G$, one obtains
for every $G$-CW-complex $X$ a \emph{Davis--L\"uck assembly map}
\begin{equation*}
 \alpha_X \colon \HH^G(X;\EE) \to \EE(G/G).
\end{equation*}
The upshot of our discussion is that we have constructed a model for the assembly map
associated to the $Or(G)$-spectrum $\Aa^{-\infty}_W$:

\begin{corollary}\label{cor_identificationassemblymap}
 Let $W$ be a free $G$-CW-complex. Then the following holds:
 \begin{enumerate}
  \item The connecting map $\Omega\DD(G,W,X) \to \TT(G,W,X)$ is equivalent
  to the equivariant $A$-theory assembly map
  \begin{equation*}
   \alpha_{X,W} \colon \HH^G(X;\Aa^{-\infty}_W) \to \Aa^{-\infty}_W(G/G) \simeq \Aa^{-\infty}(G \backslash W).
  \end{equation*}
  \item The assembly map $\alpha_{X,W}$ is a weak equivalence if and only if $\FF(G,W,X)$ is weakly contractible.
 \end{enumerate}
\end{corollary}
\begin{proof}
 The first claim follows from Theorem \ref{thm_coefficients} by the argument given in \cite[Section 6.2]{BFJR2004}.
 The second part of the corollary is then evident from the homotopy fiber sequence
 $\TT(G,W,X) \to \FF(G,W,X) \to \DD(G,W,X)$.
\end{proof}

\section{The isomorphism conjecture for Dress--Farrell--Hsiang groups}\label{sec_outline}
Recall (e.g.~from \cite[Conjecture 113]{LR2005}) the statement of the isomorphism conjecture for $A$-theory:

\begin{conjecture}[$A$-theoretic \fic]\label{conj_fic}
 Let $\cF$ be a family of groups and let $G$ be a countable discrete group. Then for every free $G$-CW-complex $W$ the assembly map
 \begin{equation*}
  \alpha_{\cF,W} \colon \HH^G(E_\cF G;\Aa^{-\infty}_W) \to \HH^G(G/G;\Aa^{-\infty}_W) \cong \Aa^{-\infty}(G \backslash W)
 \end{equation*}
 is a weak equivalence, where $E_\cF G$ is the classifying space of $G$ for the family $\cF$.
\end{conjecture}

Whenever Conjecture \ref{conj_fic} holds for some group $G$, we say that $G$ \emph{satisfies the $A$-theoretic \fic\ with respect to $\cF$}.
For the special case that $\cF = \cV\cC yc$ is the family of virtually cyclic groups, we will also say that $G$ \emph{satisfies the $A$-theoretic \ffjc}.

Due to Corollary \ref{cor_identificationassemblymap}, the \fic\ is equivalent to the weak contractibility of the spectra $\FF(G,W,E_\cF G)$
introduced in the previous section. Thus, the $A$-theoretic isomorphism conjecture becomes accessible via the methods employed in \cite{BLR2008, BL2012a, BL2012b}
for the algebraic $K$-theory and $L$-theory of group rings. Our goal is to establish an analog of the main result of \cite{BL2012b}.

Let us recall the definition of Dress--Farrell--Hsiang groups.

\begin{definition}\label{def_dressgroup}
 Let $D$ be a finite group. We call $D$ a \emph{Dress group} if there are primes $p$ and $q$
 and subgroups $P \unlhd C \unlhd D$ such that $P$ is a $p$-group, $C/P$ is cyclic and
 $D/C$ is a $q$-group.
\end{definition}

Recall the definition of the \emph{$\ell^1$-metric} on a simplicial complex. If $X$ is a simplicial complex and $\xi = \sum_x \xi_x \cdot x$, $\eta = \sum_x \eta_x \cdot x$ are points in $X$, this metric is given by $d^{\ell^1}(\xi,\eta) = \sum_x \abs{\xi_x - \eta_x}$. All simplicial complexes we consider are equipped with this metric.

We call a generating set $S$ of a group $G$ \emph{symmetric} if $s \in S$ implies $s^{-1} \in S$.

\begin{definition}\label{def_dfhgroup}
 Let $G$ be a group and $S$ be a symmetric, finite generating set of $G$. Let $\cF$ be a family of subgroups of $G$.
 
 Call $(G,S)$ a \emph{Dress--Farrell--Hsiang group with respect to $\cF$} if there exists $N \in \NN$
 such that for every $\epsilon > 0$ there is an epimorphism $\pi \colon G \twoheadrightarrow F$ onto a finite group $F$
 such that the following holds: For every Dress group $D \leq F$, there are
 a $\overline{D} := \pi^{-1}(D)$-simplicial complex $E_D$ of dimension at most $N$ whose isotropy groups lie in $\cF$, and
 a $\overline{D}$-equivariant map $\phi_D \colon G \to E_D$ such that $d^{\ell^1}(\phi_D(g),\phi_D(g')) \leq \epsilon$
 whenever $g^{-1}g' \in S$.
\end{definition}

A slightly stricter version of this definition appeared previously in \cite[Definition 3.1]{W-TransferReducibilityOfFHGroups}.
The notion of Dress--Farrell--Hsiang groups generalizes that of Farrell--Hsiang groups from \cite[Definition 1.1]{BL2012b} and \cite[Definition 2.14]{BFL2014}.
For examples, we refer to Section \ref{sec_applications} and \cite{W-TransferReducibilityOfFHGroups}.

\begin{theorem}\label{thm_fhmethod}
 Let $G$ be a discrete group. Suppose that there are a symmetric, finite generating set
 $S \subset G$ and a family of subgroups $\cF$ of $G$ such that $(G,S)$ is a Dress--Farrell--Hsiang group with respect to $\cF$.
 Then $G$ satisfies the \fic{} in $A$-theory \ref{conj_fic} with respect to $\cF$, i.e., the assembly map
 \begin{equation*}
  \HH^G(E_\cF G;\Aa^{-\infty}_W) \to \Aa^{-\infty}(G \backslash W)
 \end{equation*}
 is a weak equivalence for every free $G$-CW-complex $W$.
\end{theorem}

Choosing $W$ to be the universal cover of a given connected CW-complex whose fundamental group is $G$, Theorem \ref{thm_fhmethod} implies Theorem \ref{intro_fhmethod}.

Before we can turn to the proof of Theorem \ref{thm_fhmethod}, we need to extend the definition of the obstruction category $\cR^G_f(W,\JJ(X))$. Recall the definition of the coarse structure $\JJ(M,X) = (M \times X \times [1,\infty[,\fC(M,X),\fS(M,X))$ from Definition \ref{def_continuouscontrol}.
Let $(M_k)_{k \in \NN}$ be a sequence of metric spaces with a free, isometric $G$-action, and suppose that $X$ is a $G$-CW-complex.
Equip $G$ with a proper, left invariant metric. Then we define a coarse structure 
\begin{equation*}
 \JJ((M_k)_k,X) = \big(\coprod_{k \in \NN} M_k \times X \times [1,\infty[, \fC((M_k)_k,X), \fS((M_k)_k,X)\big) 
\end{equation*}
as follows:
\begin{enumerate}
 \item A set $C$ lies in $\fC((M_k)_k,X)$ if it is of the form $C = \coprod_k C_k$ with $C_k \in \fC(M_k,X)$,
  and it additionally satisfies the following \emph{uniform metric control condition}: There is some $R > 0$
  such that for all pairs $((m,x,t),(m',x',t')) \in C$ we have $d(m,m') < R$ (i.e., the bound does not depend on $k$).
 \item A set $S$ lies in $\fS((M_k)_k,X)$ if it is of the form $S = \coprod_k S_k$ with $S_k \in \fS(M_k,X)$.
\end{enumerate}
We consider the Waldhausen category $\cR^G_f(W,\JJ((M_k)_k,X))$. Note that this
is a subcategory of $\prod_{k \in \NN} \cR^G_f(W,\JJ(M_k,X))$ in a natural way, and that we will therefore
typically write objects and morphisms as sequences $(Y_k)_k$ and $(f_k)_k$.
Moreover, the category
\begin{equation*}
 \finprod \cR^G_f(W,\JJ(M_k,X)) := \colim_l \prod_{k = 1}^l \cR^G_f(W,\JJ(M_k,X))
\end{equation*}
of eventually trivial sequences is a Waldhausen subcategory of $\cR^G_f(W,\JJ((M_k)_k,X))$.
With this additional notation at our disposal, the proof of Theorem \ref{thm_fhmethod} proceeds as follows.

Suppose that $(G,S)$ is a Dress--Farrell--Hsiang group with respect to $\cF$. Pick $N$ be as in Definition \ref{def_dfhgroup} so that for every $k \geq 1$, there are a finite group $F_k$, an epimorphism $\pi_k \colon G \twoheadrightarrow F_k$ and
a family of maps $(\phi_D \colon G \to E_D)_{D \in \cD_k}$ such that
\begin{enumerate}
 \item the space $E_D$ is a $\overline{D} := \pi_k^{-1}(D)$-simplicial complex of dimension at most $N$ whose stabilizers lie in $\cF$,
 \item the map $\phi_D$ is $\overline{D}$-equivariant and $d^{\ell^1}(\phi_D(g),\phi_D(g')) \leq \frac{1}{k}$
  whenever $g^{-1}g' \in S$, 
\end{enumerate}
where $\cD_k$ denotes the family of Dress subgroups of $F_k$.
Then the proof is organized around a sequence of diagrams, indexed over $j \in \NN$, of the following form (details follow below):
\begin{equation*}
 \begin{tikzpicture}
  \matrix (m) [matrix of math nodes, column sep=.5em, row sep=2em, text depth=.5em, text height=1em]
  { & & \finprod \cR^G_f(W,\JJ(E_k \times G,E_\cF G)) \\
   \cR^G_f(W,\JJ((T_k \times G)_k,E_\cF G)) & & \cR^G_f(W,\JJ((E_k \times G)_k,E_\cF G)) \\
   & \cR^G_f(W,\JJ(G,E_\cF G)) & \\};
  \path[->]
  (m-2-1) edge node[above]{$((\phi_k)_k)_*$} (m-2-3) edge node[above right]{$P_j$} (m-3-2)
  (m-2-3) edge node[above left]{$Q_j$} (m-3-2)
  (m-1-3) edge (m-2-3);
 \draw[dashed, ->] (m-3-2) to[out=180,in=285] node[below]{$\transfer$} (m-2-1);
 \end{tikzpicture}
\end{equation*}
Define $T_k := \coprod_{D \in \cD_k} G / \overline{D}$, equipped with the discrete metric which assigns distance $\infty$ to any two points which are not equal. Define $E_k := \coprod_{D \in \cD_k} G \times_{\overline{D}} E_D$, equipped with the diagonal $G$-action. We consider the metric $k \cdot d^{\ell^1}$ on $E_k$. Equip $G$ with the word metric given by $S$. The products $T_k \times G$ and $E_k \times G$ become metric spaces by summing up the metrics on the two factors.

Define $P_j$ to be the projection functor which takes the inclusion into
the full product category, projects onto the $j$-th component, and then applies the functor induced by
the projection $T_j \times G \times E_\cF G \times [1,\infty[ \to G \times E_\cF G \times [1,\infty[$.
Define $Q_j$ analogously, and let the unlabelled arrow be the canonical inclusion functor.

To show that the $K$-theory of the obstruction category $\cR^G_f(W,\JJ(G,E_\cF G))$ is trivial,
the following input is required:
\begin{enumerate}
 \item\label{proof_almostequivariantmaps} There is a sequence of $G$-equivariant maps
$\phi_k \colon T_k \times G \to E_k \times G$ inducing the functor $((\phi_k)_k)_*$ such that $Q_j \circ ((\phi_k)_k)_* = P_j$ for all $j$. This is Lemma \ref{lem_contractingmaps} below.
 \item\label{proof_transfer} For each $n$, there is a \emph{transfer functor}
  \begin{equation*}
   \transfer \colon \cR^G_f(W,\JJ(G,E_\cF G)(n)) \to \cR^G_f(W,\JJ((T_k \times G)_k,E_\cF G)(n))
  \end{equation*}
  (the dashed arrow in the above diagram) such that $P_j \circ \transfer$ induces the identity in $K$-theory for all $j$;
  in fact, there is even an appropriate map on non-connective $K$-theory, but we do not need to know that.
  This is covered in Section \ref{sec_transfer}, see Corollary \ref{cor:uniformtransfer:deloop} in particular.
 \item\label{proof_squeezing} The canonical inclusion functor
  \begin{equation*}
   \finprod \cR^G_f(W,\JJ(E_k \times G,E_\cF G)) \to \cR^G_f(W,\JJ((E_k \times G)_k,E_\cF G))
  \end{equation*}
  induces a weak equivalence in non-connective $K$-theory. This follows from Theorem \ref{thm_squeezing} below.
\end{enumerate}
Using the fact that $K_{-n}(\cR^G_f(W,\JJ(X))) \cong K_1(\cR^G_f(W,\JJ(X)(n+1))$ (Proposition \ref{prop_connectivitystructuremaps}),
a diagram chase shows that $K_n(\cR^G_f(W,\JJ(X))) = 0$ for all $n \in \ZZ$ under these assumptions (see \cite[Section 4]{BL2012b}).
We remark that the only part of the proof which uses the presence of the classifying space $E_\cF G$
is (\ref{proof_squeezing}); the other two parts still work if we replace $E_\cF G$ with an arbitrary $G$-CW-complex.

\begin{lemma}[{cf.\ \cite[Section 7]{BL2012b}}]\label{lem_contractingmaps}
 Let $X$ be a $G$-CW-complex. For each $D$, the $\overline{D}$-equivariant map $\phi_D$ gives rise to a $G$-equivariant map
 \begin{equation*}
  \begin{split}
   \tilde{\phi_D} \colon G / \overline{D} \times G &\to G \times_{\overline{D}} E_D \\
   (\gamma\overline{D},g) &\mapsto (\gamma, \phi_D(\gamma^{-1}g)).
  \end{split}
 \end{equation*}
 Then the equivariant maps
 \begin{equation*}
  \begin{split}
   \phi_k \colon T_k \times G \times X \times [1,\infty[ &\to E_k \times G \times X \times [1,\infty[ \\
   (\gamma\overline{D},g,x,t) &\mapsto (\tilde{\phi_D}(\gamma\overline{D},g),g,x,t)
  \end{split}
 \end{equation*}
 induce an exact functor
 \begin{equation*}
  ((\phi_k)_k)_* \colon \cR^G_f(W,\JJ((T_k \times G)_k,X)) \to \cR^G_f(W,\JJ((E_k \times G)_k,X))
 \end{equation*}
 such that $Q_j \circ ((\phi_k)_k)_* = P_j$ for all $j$.
\end{lemma}
\begin{proof}
 If the given maps induce a well-defined functor, this functor has the required property.
 So we have to check that composing with the maps $\phi_k$ preserves the uniform metric control condition.
 Let $k$ be arbitrary. Suppose that $d_{T_k \times G}((\gamma\overline{D},g),(\gamma'\overline{D'},g')) < R$.
 Then $d_G(g,g') < R$ and $d_{T_k}(\gamma\overline{D},\gamma'\overline{D'}) < R$, which implies that
 $\gamma\overline{D} = \gamma'\overline{D'}$. Hence, $\overline{D} = \overline{D'}$ and there is some
 $\delta \in \overline{D}$ such that $\gamma' = \gamma\delta$. Moreover, we can find
 $m < R$ and $s_1, \dots, s_m \in S$ such that $g' = gs_1 \dots s_m$. 
 It follows that
 \begin{equation*}
  \begin{split}
   d_{G \times_{\overline{D}} E_D}^{\ell^1}&((\gamma,\phi_D(\gamma^{-1}g)),(\gamma',\phi_D(\gamma'^{-1}g')) \\
   &= d_{G \times_{\overline{D}} E_D}^{\ell^1}((\gamma,\phi_D(\gamma^{-1}g)),(\gamma\delta,\phi_D(\delta^{-1}\gamma^{-1}g')) \\
   &= d_{G \times_{\overline{D}} E_D}^{\ell^1}((\gamma,\phi_D(\gamma^{-1}g)),(\gamma,\phi_D(\gamma^{-1}g')) \\
   &= d_{E_D}^{\ell^1}(\phi_D(\gamma^{-1}g),\phi_D(\gamma^{-1}g')) \\
   &\leq \frac{m}{k} < \frac{R}{k}
  \end{split}
 \end{equation*}
 due to the $S$-equivariance of $\phi_D$ up to $\frac{1}{k}$. We conclude that
 \begin{equation*}
  \begin{split}
   d_{E_k \times G}&((\tilde{\phi_D}(\gamma\overline{D},g),g),(\tilde{\phi_D}(\gamma'\overline{D'},g'),g') \\
   &= d_{E_k \times G}((\tilde{\phi_D}(\gamma\overline{D},g),g),(\tilde{\phi_D}(\gamma\overline{D},g'),g') \\
   &< d_G(g,g') + k \cdot \frac{R}{k} \\
   &< 2R,
  \end{split}  
 \end{equation*}
 so uniform metric control is preserved.
\end{proof}

\section{The $A$-theoretic Swan group}\label{sec_swangroup}
In the linear setting, the transfer functors mentioned in the previous section
are defined via the action of the \emph{Swan group} of $G$ on the $K$-theory of the obstruction category.
This group arises as the Grothendieck group of the category of integral, finite-rank $G$-representations,
and the action is induced by tensoring such a representation with geometric modules. See \cite[Section 5]{BL2012b}.

To establish the existence of a transfer functor, we will need a non-linear analog of this action.
For this purpose, recall the notion of \emph{biexact functor} from \cite[p.342]{Waldhausen1985}:
If $\cC_1$, $\cC_2$ and $\cC_3$ are Waldhausen categories, a \emph{biexact functor} is a functor
\begin{equation*}
 \wedge \colon \cC_1 \times \cC_2 \to \cC_3
\end{equation*}
with the following properties:
\begin{enumerate}
 \item[(E1)] The functor is \emph{exact in the first variable}, i.e., for all $A_2 \in \cC_2$
  the functor $- \wedge A_2 \colon \cC_1 \to \cC_3$ is exact.
 \item[(E2)] The functor is \emph{exact in the second variable}, i.e., for all $A_1 \in \cC_1$
  the functor $A_1 \wedge - \colon \cC_2 \to \cC_3$ is exact.
 \item[(TC)] The functor satisfies the ``more technical condition'' that for every pair of cofibrations
  $A_1 \rightarrowtail A_1'$ in $\cC_1$ and $A_2 \rightarrowtail A_2'$ in $\cC_2$, the canonical morphism
  $(A_1 \wedge A_2') \cup_{A_1 \wedge A_2} (A_1' \wedge A_2) \to A_1' \wedge A_2'$ is a cofibration in $\cC_3$.
\end{enumerate}
As explained in {\it loc.\ cit.}, such a functor induces pairings on homotopy groups
\begin{equation*}
 K_i(\cC_1) \times K_j(\cC_2) \to K_{i+j}(\cC_3)
\end{equation*}
for all $i,j \in \NN$.

Define $\rep(G)$ to be the category of pointed (right) $G$-CW-complexes whose underlying CW-complex is finite; the morphisms of this category are those maps which are pointed, equivariant and cellular.
This category can be equipped with a Waldhausen structure in which the cofibrations are the morphisms
isomorphic to a cellular inclusion, and the weak equivalences are the morphisms which are homotopy equivalences
\emph{in the non-equivariant sense}. We denote the subcategory of these weak equivalences by $h\rep(G)$.

\begin{definition}\label{def_swangroup}
 Define the \emph{$A$-theoretic Swan group} $\swana(G)$ to be
 \begin{equation*}
  \swana(G) := K_0(\rep(G),h).
 \end{equation*}
\end{definition}

Explicitly, $\swana(G)$ is generated by $h$-equivalence classes of objects in $\rep(G)$,
subject to the condition that $[D_0] + [D_2] = [D_1]$ whenever there is a cofibration sequence
$D_0 \rightarrowtail D_1 \twoheadrightarrow D_2$ in $\rep(G)$. As $-[D] = [\Sigma D]$, every element $s \in \swana(G)$ can be written as $s = [D]$ for some object $D \in \rep(G)$.

We can extend the abelian group structure
on $\swana(G)$ to a ring structure using the smash product. The proof of the following proposition amounts to a number of well-known facts about the smash product of CW-complexes.

\begin{proposition}\label{prop_swana-ring}
 The functor
 \begin{equation*}
  \wedge \colon \rep(G) \times \rep(G) \to \rep(G), \quad (D,D') \mapsto D \wedge D'
 \end{equation*}
 is biexact. The functors $- \wedge S^0$ and $S^0 \wedge -$ are naturally equivalent to the identity functor,
 and the square
 \begin{equation*}
  \commsquare{\rep(G) \times \rep(G) \times \rep(G)}{\rep(G) \times \rep(G)}{\rep(G) \times \rep(G)}{\rep(G)}{\id \times \wedge}{\wedge \times \id}{\wedge}{\wedge}
 \end{equation*}
 commutes up to natural isomorphism.
 \qed
\end{proposition}

Thus, $\swana(G)$ becomes a ring under the product $[D] \cdot [D'] := [D \wedge D']$.
The main point about $\swana(G)$ is that it admits an action on $A$-theory.
Suppose that $(Y \leftrightarrows W, \kappa)$ is a labelled $G$-CW-complex and retractive space relative $W$;
let $s \colon W \to Y$ be the structural inclusion and $r \colon Y \to W$ be the structural retraction.
Then we can form the pushout
\begin{equation*}
 \commsquare{(* \times Y) \cup (D \times W)}{W}{D \times Y}{D \wedge_W Y}{r\cup \pr_W}{\inc \cup (\id \times s)}{D \wedge_W s}{}
\end{equation*}
to obtain a $G$-CW-complex under $W$. We equip the product $D \times Y$ with the diagonal action
$g \cdot (d,y) := (d g^{-1}, gy)$. By the universal property of the pushout, every map of retractive spaces $f \colon Y_1 \to Y_2$ induces a map $D \wedge_W f \colon D \wedge_W Y_1 \to D \wedge_W Y_2$. In particular, we can equip $D \wedge_W Y$ with a structural retraction $D \wedge_W r \colon D \wedge_W Y \to D \wedge_W W \cong W$.
Regarding $D$ as a $G$-CW-complex relative to the basepoint, we let $\cells D$ denote the set of relative cells in $D$.
Then $D \wedge_W Y$ becomes a labelled $G$-CW-complex via the control map
\begin{equation*}
 D \wedge_W \kappa \colon \cells (D \wedge_W Y) \cong \cells D \times \cells Y \xrightarrow{\pr} \cells Y \xrightarrow{\kappa} Z.
\end{equation*}
Our goal is to show that this pairing defines a biexact functor. To do this, we need the controlled version
of a well-known statement about homotopy equivalences of free $G$-CW-complexes.

Let $\fZ$ be a coarse structure and consider the category of controlled retractive spaces $\cR^G(W,\fZ)$.
Observe that we have a notion of control for non-equivariant maps between labelled $G$-CW-complexes.

\begin{lemma}\label{lem_easycontrolledwhitehead}
 Let $Y_1$ and $Y_2$ be objects in $\cR^G(W,\fZ)$.
 Suppose $f \colon Y_1 \to Y_2$ is a morphism in $\cR^G(W,\fZ)$ such that there are a non-equivariant
 controlled map $\overline{g} \colon Y_2 \to Y_1$ as well as non-equivariant controlled homotopies
 $\overline{H} \colon \id_{Y_1} \simeq_\fZ \overline{g}f$ and $\overline{K} \colon \id_{Y_2} \simeq_\fZ f\overline{g}$.
 
 Then $f$ is an $h$-equivalence, i.e., there are a $G$-equivariant, controlled inverse $g$ and $G$-equivariant, controlled homotopies $gf \simeq \id_{Y_1}$ and $fg \simeq \id_{Y_2}$.
\end{lemma}
\begin{proof}
 The proof works as in the uncontrolled case, cf.~\cite[Proposition II.2.7]{tomDieck-transformationgroups}.
\end{proof}

\begin{proposition}\label{prop_atheory-swanamodule}
 The smash product $\wedge_W$ over $W$ induces a biexact functor
 \begin{equation*}
  \begin{split}
   \wedge_W \colon \rep(G) \times \cR^G(W,\fZ) &\to \cR^G(W,\fZ) \\
   (D,(Y \leftrightarrows W, \kappa)) &\mapsto (D \wedge_W Y \leftrightarrows W, D \wedge_W \kappa)
  \end{split}  
 \end{equation*}
 which preserves the property of being finite.
 
 The functor $S^0 \wedge_W -$ is naturally equivalent to the identity functor,
 and the diagram
 \begin{equation*}
  \commsquare{\rep(G) \times \rep(G) \times\cR^G(W,\fZ)}{\rep(G) \times \cR^G(W,\fZ)}{\rep(G) \times \cR^G(W,\fZ)}{\cR^G(W,\fZ)}{\wedge \times \id}{\id \times \wedge_W}{\wedge_W}{\wedge_W}
 \end{equation*}
 commutes up to natural isomorphism.
\end{proposition}
\begin{proof}
 Observe that $D \wedge_W Y$ contains only free $G$-cells because $Y$ is assumed to be free (relative $W$).
 Every cell $(e_D,e_Y)$ of $D \wedge_W Y$ is labelled by the same point in $Z$ as $e_Y$, so it is immediate that
 the support and control conditions are preserved by $\wedge_W$. Moreover, for any subset $A \subset Z$ we have
 $(D \wedge_W \kappa)^{-1}(A) = \cells D \times \kappa^{-1}(A)$; since $D$ is finite, $D \wedge_W Y$ lies in $\cR^G_f(W,\fZ)$  whenever $Y$ is an object in $\cR^G_f(W,\fZ)$.
 
 Let us turn to exactness in the first variable. Fix $Y \in \cR^G(W,\fZ)$. We have $* \wedge_W Y = W$.
 Let $i \colon D \rightarrowtail D'$ be a cofibration in $\rep(G)$. Then $i \wedge_W Y \colon D \wedge_W Y \to D' \wedge_W Y$ is also a cofibration, because the same holds for cellular inclusions.
 
 Suppose that $D$ is the pushout of $D_2 \leftarrow D_0 \rightarrowtail D_1$ in $\rep(G)$.
 Then $D \times Y$ is also the pushout of $(D_2 \leftarrow D_0 \rightarrowtail D_1) \times Y$,
 and similarly for $(* \times Y) \cup (D \times W)$. Since pushouts commute with each other, we see that $D \wedge_W Y$  is also the pushout of $(D_2 \wedge_W Y) \leftarrow (D_0 \wedge_W Y) \rightarrowtail (D_1 \wedge_W Y)$.

 The interesting part of the argument is to show that $- \wedge_W Y$ preserves $h$-equivalences.
 Suppose that $\delta \colon D \xrightarrow{\sim} D'$ is a weak equivalence, i.e., there is a
 non-equivariant map $\overline{\delta} \colon D' \to D$ such that $\delta\overline{\delta}$
 and $\overline{\delta}\delta$ are (non-equivariantly) homotopic to the identity map.
 Taking smash products with $\id_Y$ and the constant homotopy on $Y$, we observe
 that $\delta \wedge_W Y$ is a morphism in $\cR^G(W,\fZ)$ which is an $h$-equivalence in $\cR(W,\fZ)$,
 i.e., upon forgetting all $G$-actions. By Lemma \ref{lem_easycontrolledwhitehead}, $\delta \wedge_W Y$ is an $h$-equivalence in $\cR^G(W,\fZ)$.
  
 Exactness in the second variable is similar, but easier. To show condition (TC), one has to show that
 for $D' \in \rep(G)$, $Y' \in \cR^G(W,\fZ)$ and subcomplexes $D \subset D'$, $Y \subset Y'$,
 the complex $(D \wedge_W Y') \cup (D' \wedge_W Y)$ is naturally a subcomplex of $D' \wedge_W Y'$,
 which is the case.
 
 Finally, $S^0 \wedge_W Y \cong Y$, and associativity of the pairing follows again from the fact
 that pushouts commute with each other.
\end{proof}

As explained at the beginning of this section, the biexact functor $- \wedge_W -$ from Proposition \ref{prop_atheory-swanamodule} turns $K_i(\cR^G_f(W,\fZ),h)$ and $K_i(\cR^G_{fd}(W,\fZ),h)$ into $\swana(G)$-modules for all $i \in \ZZ$ (using that $K_{-i}(\cR^G_f(W,\fZ),h) := K_1(\cR^G_f(W,\fZ(i+1),h)$, $i \geq 0$).

\begin{remark}\label{rem_bivarianta}
 Let us digress for a moment to outline the connection between the pairing induced by the biexact functor
 $\wedge_W$ and bivariant $A$-theory \cite[Section 4]{Williams1998} (see also \cite[Section 3]{RS2014}).
 For the purpose of this remark, we relax the definition of retractive spaces to allow for spaces which are not CW-complexes.

 Let $p \colon V_1 \to V_2$ be a fibration.
 Then the category $\cR(p)$ consists of those retractive spaces $(Y,r,s)$ over $V_1$ such that the composition $p \circ r$ is a fibration,
 and for every $v \in V_2$ the (homotopy) fiber $\cF_v(p \circ r)$ of $p \circ r$ at $v$ is finitely dominated in $\cR(\cF_v(p))$.
 Note that $\cR(V \to *)$ is simply the category of (finitely dominated) retractive spaces over $V$.
 
 For two composable fibrations $p \colon V_1 \to V_2$ and $q \colon V_2 \to V_3$, there is defined an exact functor
 \begin{equation*}
  \cR(q) \times \cR(p) \to \cR(q \circ p)
 \end{equation*}
 given on objects by first pulling back along $p$, then taking the external smash product and finally pulling back once more along the diagonal map $\Delta \colon V_1 \to V_1 \times V_1$.
 
 Let $W$ be a free $G$-CW-complex (or more generally, a principal $G$-bundle). Let $V := * \times_G W = G \backslash W$ denote the quotient.
 Taking quotients with respect to the $G$-action defines a functor $\cR^G_{fd}(W) \to \cR_{fd}(V) = \cR(V \to *)$ which induces an equivalence in $K$--theory \cite[Lemma~2.1.3]{Waldhausen1985}.
 
 Moreover, there exists an exact functor $F \colon \rep(G) \to \cR(V \xrightarrow{\id} V)$ sending $D$ to $D \times_G W$; since $D$ comes equipped with a base point, there is an induced section to the canonical retraction map $D \times_G W \to * \times_G W = V$.
 
 Combining these functors, we obtain a diagram, commutative up to natural isomorphism,
 \begin{equation*}
  \begin{tikzpicture}
   \matrix (m) [matrix of math nodes, column sep=2em, row sep=2em, text depth=.5em, text height=1em]
   {\rep(G) \times \cR^G(W)                                            & \cR^G(W)         \\
    \cR(\id_{V}) \times \cR(* \times_G W \to *) & \cR(V \to *) \\};
   \path[->]
   (m-1-1) edge node[above]{$\wedge_W$} (m-1-2) edge (m-2-1)
   (m-1-2) edge (m-2-2)
   (m-2-1) edge (m-2-2);
  \end{tikzpicture}
 \end{equation*}
 which relates the action of $\swana(G)$ on $A$-theory to the bivariant theory.
 
 In fact, Malkiewich and Merling have shown that the functor $F$ induces an equivalence in $K$-theory for $W = EG$ \cite[Proposition~3.7]{MM-EquivariantA}; they have the standing assumption that the group $G$ is finite, but this specific part of the argument works for arbitrary discrete groups.
 Hence, the action of $\swana(G)$ on $A(BG)$ coincides with the one of the ``upside-down-$A$-theory" of $BG$.
 
 This ends the digression.
\end{remark}

We need to consider functoriality of $\swana$ in $G$ to some extent. For any group homomorphism $\phi \colon H \to G$, restriction defines an exact functor
\begin{equation*}
 \res_\phi \colon \rep(G) \to \rep(H), \quad D \mapsto \res_\phi D.
\end{equation*}
If $H$ is a subgroup of $G$ and $[G \colon H] < \infty$, we can also define an exact induction functor
\begin{equation*}
 \ind_H^G \colon \rep(H) \to \rep(G), \quad D \mapsto D \wedge_H (G_+).
\end{equation*}
Note that this functor does not preserve the unit object $S^0$; in fact, $\ind_H^G S^0 = (H \backslash G)_+$.

Let us also consider the case of $A$-theory. Abbreviate the category $\cR^G(W,\fB(G)(n))$ by $\cR^G(W,n)$.
There, we have for an arbitrary subgroup $H \leq G$ an induction functor
\begin{equation*}
 \ind_H^G \colon \cR^H(\res_G^H W,n) \to \cR^G(W,n), \quad (Y,\kappa) \mapsto (\ind_H^G Y, \ind_H^G \kappa),
\end{equation*}
where $\ind_H^G Y$ is defined as the pushout
\begin{equation*}
 \begin{tikzpicture}
  \matrix (m) [matrix of math nodes, column sep=5em, row sep=2em, text depth=.5em, text height=1em]
  {G \times_H (\res_G^H W) & W \\
   G \times_H Y            & \ind_H^G Y \\};
  \path[->]
  (m-1-1) edge node[above]{$(g,w) \mapsto g \cdot w$} (m-1-2) edge node[left]{$G \times_H s$} (m-2-1)
  (m-1-2) edge node[right]{$\ind_H^G s$} (m-2-2)
  (m-2-1) edge (m-2-2);
 \end{tikzpicture}
\end{equation*}
and the control map $\ind_H^G \kappa$ is given by
\begin{equation*}
 \begin{split}
  \ind_H^G \kappa \colon \cells (\ind_H^G Y) = G \times_H (\cells Y) &\to G \times_H (\RR^n \times H) \cong \RR^n \times G \\
  (g,e) &\mapsto (\kappa(e),g).
 \end{split}
\end{equation*}
Suppose $[G \colon H] < \infty$. Choose a normalized set-theoretic section $\sigma$ to the projection map $G \to H \backslash G$. Then $\sigma$ induces an $H$-equivariant map $p_\sigma \colon G \to H, g \mapsto g \sigma(Hg)^{-1}$ which defines a morphism of coarse structures $\fB(G)(n) \to \fB(H)(n)$; we retain bounded control because $H \backslash G$ is finite. Then define an exact functor
\begin{equation*}
 \res_G^H \colon \cR^G(W,n) \to \cR^H(\res_G^H W,n), \quad (Y,\kappa) \mapsto (\res_G^H Y, (\RR^n \times  p_\sigma) \circ \kappa).
\end{equation*}
Note that this functor also preserves finiteness. The functor $\res^H_G$ depends on the choice of $\sigma$, but a different choice of $\sigma$ yields a naturally isomorphic functor. Hence, we will suppress $\sigma$ in what follows.

The restriction and induction functors are related in the expected way:

\begin{lemma}[Frobenius Reciprocity]\label{lem_frobeniusreciprocity}
 Let $G$ be a group and let $H \leq G$ be a subgroup of finite index. Then we have for all $s \in \swana(H)$, $t \in \swana(G)$ and $a \in K_i(\cR^G_f(W,n))$
 \begin{equation*}
  \begin{split}
   \ind_H^G(s) \cdot t &= \ind_H^G(s \cdot \res_G^H t), \\
   \ind_H^G(s) \cdot a &= \ind_H^G(s \cdot \res_G^H a).
  \end{split}
 \end{equation*}
 More precisely, there are natural equivalences
 \begin{equation*}
  \wedge \circ (\ind_H^G \times \id) \xrightarrow{\sim} \ind_H^G \circ \wedge \circ (\id \times \res_G^H) \colon \rep(H) \times \rep(G) \to \rep(G)
 \end{equation*}
 and
 \begin{equation*}
  \wedge_W \circ (\ind_H^G \times \id) \xrightarrow{\sim} \ind_H^G \circ \wedge_W \circ (\id \times \res_G^H) \colon \rep(H) \times \cR^G(W,\fZ) \to \cR^G(W,\fZ).
 \end{equation*}
\end{lemma}
\begin{proof}
 Let $D \in \rep(H)$ and $D' \in \rep(G)$. Then the first equivalence is implemented by the $G$-equivariant homeomorphism
 \begin{equation*}
  \begin{split}
   (D \wedge_H G_+) \wedge D' &\xrightarrow{\cong} (D \wedge \res_G^H D') \wedge_H G_+ \\
   ((d,g),d') & \mapsto ((d,d'g^{-1}),g).
  \end{split}
 \end{equation*}
 For $D \in \rep(H)$ and $Y \in \cR^G(W,\fZ)$, the $G$-equivariant homeomorphism
 \begin{equation*}
  \begin{split}
   (D \wedge_H G_+) \wedge_W Y &\xrightarrow{\cong} \ind_H^G (D \wedge_{\res_G^H W} \res_G^H Y) \\
   ((d,g),y) &\mapsto (g^{-1},(d,gy))
  \end{split}
 \end{equation*}
 yields the second equivalence.
\end{proof}

\begin{theorem}\label{thm_inductionswana}
 Let $G$ be a finite group. The homomorphism
 \begin{equation*}
  \sum_{H} \ind_H^G \colon \bigoplus_{H \leq G \text{ Dress}} \swana(H) \to \swana(G)
 \end{equation*}
 is a surjection.
\end{theorem}
\begin{proof}
 By Frobenius Reciprocity \ref{lem_frobeniusreciprocity}, it suffices to show that $1_G = [S^0]$ lies in the image of the homomorphism.
 Since we can filter $D \in \rep(G)$ by its skeleta and suspension of objects corresponds to taking
 inverses in $K_0$, the class of $D$ equals its (equivariant) Euler characteristic. Hence, if $[S^0] = [D_+]$ for some finite $G$-CW-complex $D$ which has no $G$-fixed-point, then $[S^0]$ is a sum of elements which are induced from proper subgroups. 
 
 If $G$ is not a Dress group, then $G$ acts on a finite, contractible CW-complex $D$ without $G$-fixed points by a theorem of Oliver \cite{Oliver1975}. Then $[D_+] = [S^0]$ in $\swana(G)$. The claim follows by induction.
\end{proof}

\begin{remark}\label{rem_inductionswana-improvement}
 For the sake of completeness, note that Oliver's theorem \cite{Oliver1975} even says that a finite group acts without a global fixed-point on a finite, contractible CW-complex if and only if the group is not Dress.

 One can do slightly better than the induction argument in the proof of Theorem \ref{thm_inductionswana}.
 As shown in \cite[Corollary 2.10]{W-TransferReducibilityOfFHGroups}, Oliver's theorem implies the existence of
 a finite, contractible $G$-CW-complex all whose stabilizers are Dress groups.
\end{remark}

\begin{corollary}\label{cor_inductiona}
 Let $G$ be a finite group, and let $W$ be a $G$-CW-complex. Then the homomorphism
 \begin{equation*}
  \sum_H \ind_H^G \colon \bigoplus_{H \leq G \text{ Dress}} K_i(\cR^H_f(\res_G^H W,n)) \to K_i(\cR^G_f(W,n))
 \end{equation*}
 is surjective for all $i \in \ZZ$.
\end{corollary}
\begin{proof}
 Immediate from Frobenius Reciprocity \ref{lem_frobeniusreciprocity} and Theorem \ref{thm_inductionswana}.
\end{proof}

We will also be able to describe the kernel of the surjection in Corollary \ref{cor_inductiona} once we have shown Theorem \ref{thm_fhmethod} (see Theorem \ref{thm_FJCFiniteGroups}).
As a second application, we obtain a variant of Swan's induction theorem \cite[Cor.~4.2]{Swan1960}.
Recall that the Swan group $\swan(G)$ is the Grothendieck group of integral, finite-rank $G$-representations.

\begin{corollary}\label{cor:variant-swan-induction}
 Let $G$ be a finite group. Then the unit element $1_G = [\ZZ] \in \swan(G)$ can be written as a sum of permutation modules,
 \begin{equation*}
  1_G = \sum_{i = 1}^k n_i \cdot \big[ \ZZ[G/H_i] \big],
 \end{equation*}
 where each $H_i$ is a Dress group and $n_i \in \ZZ$.
\end{corollary}
\begin{proof}
 We define a linearization homomorphism
 \begin{equation*}
  \swana(G) \to \swan(G), \quad [D] \mapsto \sum_{k = 0}^\infty (-1)^k [\tilde{C}_k(D)],
 \end{equation*}
 where $\tilde{C}_*$ denotes the reduced cellular chain complex. This is a well-defined ring homomorphism. Then the claim follows from the proof of Theorem \ref{thm_inductionswana}.
\end{proof}

Corollary \ref{cor:variant-swan-induction} differs from Swan's theorem \cite[Cor.~4.2]{Swan1960} in that we obtain a description of $1_G$ in terms of permutation modules instead of arbitrary representations, at the expense of considering a larger family of subgroups.

\section{The transfer functor}\label{sec_transfer}
We proceed to construct the transfer functors $\transfer$ from Section \ref{sec_outline}. This uses the action of $\swana$ induced by ``$\wedge_W$'' on $A$-theory from the previous section. The proof proceeds as in \cite[Section 6]{BL2012b}.

Let $G$ be a countable discrete group. Let $\pi \colon G \to Q$ be a surjective group homomorphism, and let
$H \leq Q$ be a subgroup of finite index. Then we define a biexact functor
\begin{equation*}
 \begin{split}
  T_{\pi,H} \colon \rep(H) \times \cR^G_f(W,\JJ(X)) &\to \cR^G_f(W,\JJ(X)) \\
  (D,(Y,\kappa)) &\mapsto \res_\pi(\ind_H^Q D) \wedge_W (Y,\kappa).
 \end{split}
\end{equation*}
Recall the coarse structure $\JJ(X) = \JJ(G,X)$ from Definition \ref{def_continuouscontrol}. Let $\overline{H} := \pi^{-1}(H)$, and equip $G / \overline{H} \times G$ with the metric
\begin{equation*}
 d_{G,\overline{H}}((\gamma_1\overline{H},g_1),(\gamma_2\overline{H},g_2)) := \begin{cases} d_G(g_1,g_2) & \gamma_1\overline{H} = \gamma_2\overline{H}, \\ \infty & \text{otherwise.} \end{cases}
\end{equation*}
Next, we define another functor
\begin{equation*}
 \hat{T}_{\pi,H} \colon \rep(H) \times \cR^G_f(W,\JJ(G,X)) \to \cR^G_f(W,\JJ(G / \overline{H} \times G,X))
\end{equation*}
which lifts $T_{\pi,H}$ along the projection functor induced by $G / \overline{H} \times G \to G$. To do so, we equip $\res_\pi(\ind_H^Q D) \wedge_W Y$ with a different control map whose definition we give next.

The unique map $\cells D \to H \backslash H$ induces a $Q$-equivariant function $\cells (\ind_H^Q D) = (\cells D) \times_H Q \to H \backslash Q$. Restricting the $Q$-actions along $\pi$, we obtain a $G$-equivariant function $c_D' \colon \cells (\res_\pi \ind_H^Q D) \to \overline{H} \backslash G$. Regarding source and target as left $G$-sets by letting $g$ act via $g^{-1}$ on the right, we obtain a map of left $G$-sets. Moreover, we can identify $\overline{H} \backslash G$ with its left $G$-action with $G/\overline{H}$. Then we regard $c_D'$ as a $G$-equivariant function of left $G$-sets
\begin{equation*}
 c_D \colon \cells (\res_\pi \ind_H^Q D) \to G/\overline{H}.
\end{equation*}
For $D \in \rep(H)$ and $(Y,\kappa) \in \cR^G_f(W,\JJ(G,X))$, define
\begin{equation*}
 \hat{T}_{\pi,H}(Y) := \res_\pi (\ind_H^Q D) \wedge_W Y 
\end{equation*}
and its control map
\begin{equation*}
 \hat{T}_{\pi,H}(\kappa) \colon \cells \hat{T}_{\pi,H}(Y) \cong \cells (\res_\pi \ind_H^Q D) \times \cells Y \xrightarrow{c_D \times \kappa} G / \overline{H} \times G \times X \times [1,\infty[.
\end{equation*}
On morphisms, we set $\hat{T}_{\pi,H}(\delta,f) := \res_\pi(\ind_H^Q \delta) \wedge_W f$.

\begin{lemma}\label{lem_simpletransfer}
 This defines a biexact functor
 \begin{equation*}
  \hat{T}_{\pi,H} \colon \rep(H) \times \cR^G_f(W,\JJ(G,X)) \to \cR^G_f(W,\JJ(G / \overline{H} \times G,X))
 \end{equation*}
 with the following properties:
 \begin{enumerate}
  \item\label{item:simpletransfer:1} If $f$ is a morphism in $\cR^G_f(W,\JJ(G,X))$ which is $R$-controlled over $G$ and $\delta$ is any morphism in $\rep(H)$,
   then $\hat{T}_{\pi,H}(\delta,f)$ is $R$-controlled over $G / \overline{H} \times G$.
  \item\label{item:simpletransfer:2} Let $P \colon \cR^G_f(W,\JJ(G / \overline{H} \times G,X)) \to \cR^G_f(W,\JJ(G,X))$ denote the canonical projection functor.
   Then $P \circ \hat{T}_{\pi,H} = T_{\pi,H}$.
 \end{enumerate}
\end{lemma}
\begin{proof}
 Let $\delta \colon D \to D'$ be an arbitrary morphism in $\rep(H)$. Then the induced morphism $\ind_H^G \delta = \delta \wedge_H Q_+$
 has the property that any cell $(e,q) \in \cells D \times_H Q$ is mapped to $D \wedge_H (Hq)_+ \subset D \wedge_H Q_+$.
 It follows that $\hat{T}_{\pi,H}(\delta,f)$ is $0$-controlled over $G / \overline{H}$. Since the control map of $\hat{T}_{\pi,H}(D,Y)$ is defined as a product, $\hat{T}_{\pi,H}(\delta,f)$ is $R$-controlled if $f$ is $R$-controlled. In particular, $\hat{T}_{\pi,H}$ is well-defined.
   
 The equality $P \circ \hat{T}_{\pi,H} = T_{\pi,H}$ is obvious.
\end{proof}

\begin{proposition}\label{prop_boundedtransfer}
 Let $G$ be a countable discrete group and let $\pi \colon G \to F$ be a surjective group homomorphism onto a finite group $F$.
 Suppose that $G$ is equipped with a proper, left invariant metric. Let $\cD$ denote the family of Dress subgroups of $F$.
 Define $M := \coprod_{H \in \cD} G/\overline{H} \times G$; we equip $G/\overline{H} \times G$
 with the metric in which different summands are infinitely far apart, and where each summand carries the metric $d_{G,\overline{H}}$.
 
 Then there is an exact functor $\transfer_\pi \colon \cR^G_f(W,\JJ(X)) \to \cR^G_f(W,\JJ(M,X))$ with the following properties:
 \begin{enumerate}
  \item If $f$ is a morphism which is $R$-controlled over $G$, then $\transfer_\pi(f)$ is $R$-controlled over $M$.
  \item Let $P \colon \cR^G_f(W,\JJ(M,X)) \to \cR^G_f(W,\JJ(X))$ denote the functor induced by the projection map
   $M \to G$. Then $P \circ \transfer_\pi$ induces the identity map on $K$-groups.
 \end{enumerate}
\end{proposition}
\begin{proof}
 Using Theorem \ref{thm_inductionswana}, we can find a sequence $(D_H)_{H \in \cD}$ with $D_H \in \rep(H)$ such that
 \begin{equation}\label{eq:boundedtransfer}
  \sum_{H \in \cD} [\ind_H^F D_H] = 1_F \in \swana(F).
 \end{equation}
 Define the transfer by
 \begin{equation*}
  \transfer_\pi(Y,\kappa) := \bigvee_{H \in \cD} \hat{T}_{\pi,H}(D_H,(Y,\kappa)),
 \end{equation*}
 where we regard $\hat{T}_{\pi,H}(D_H,(Y,\kappa))$ as an object over $M \times X \times [1,\infty[$ via the natural inclusion
 $G / \overline{H} \times G \times X \times [1,\infty[ \subset M \times X \times [1,\infty[$. Similarly, we set
 $\transfer_\pi(f) := \bigvee_H \hat{T}_{\pi,H}(\id_{D_H},f)$ for morphisms.
 
 As a consequence of Lemma \ref{lem_simpletransfer}\eqref{item:simpletransfer:1}, this functor preserves $R$-controlled morphisms. Moreover, we have
 \begin{equation*}
  P \circ \transfer_\pi = P \circ \big( \bigvee_{H \in \cD} \hat{T}_{\pi,H}(D_H,-) \big) \cong \bigvee_{H \in \cD} P \circ \hat{T}_{\pi,H}(D_H,-) = \bigvee_{H \in \cD} T_{\pi,H}(D_H,-).
 \end{equation*}
 Using the action of $\swana(G)$ on $K_i(\cR^G_f(W,\JJ(X))$ and the identity \eqref{eq:boundedtransfer}, we conclude that for $a \in K_i(\cR^G_f(W,\JJ(X))$,
 \begin{equation*}
  \begin{split}
   K_i(P \circ \transfer_\pi)(a)
   &= \sum_{H \in \cD} [\res_\pi(\ind_H^F D_H)] \cdot a
   = \Big(\sum_{H \in \cD} [\res_\pi(\ind_H^F D_H)] \Big) \cdot a \\
   &= \res_\pi \Big(\sum_{H \in \cD} [\ind_H^F D_H] \Big) \cdot a
   = \res_\pi(1_F) \cdot a \\
   &= 1_G \cdot a = a,
  \end{split}
 \end{equation*}
 so $P \circ \transfer_\pi$ induces the identity map as claimed.
\end{proof}

\begin{corollary}\label{cor_boundeduniformtransfer}
 Let $G$ be a countable discrete group. For every $k \in \NN$, let $\pi_k \colon G \twoheadrightarrow F_k$ be an epimorphism onto a finite group.
 Let $\cD_k$ be the family of Dress subgroups of $F_k$, and define $T_k := \coprod_{H \in \cD_k} G/\overline{H} \times G$. Recall the definition of the coarse structure $\JJ((T_k)_k,X)$ from Section \ref{sec_outline}.
 
 Then there is an exact functor
 \begin{equation*}
  \transfer \colon \cR^G_f(W,\JJ(X)) \to \cR^G_f(W,\JJ((T_k)_k,X))
 \end{equation*}
 such that each composition $P_k \circ \transfer$ of $\transfer$ with the functor $P_k \colon \cR^G_f(W,\JJ((T_k)_k,X)) \to \cR^G_f(W,\JJ(X))$ from Section \ref{sec_outline} induces the identity on $K$-groups.
\end{corollary}
\begin{proof}
 Define $\transfer := (\transfer_{\pi_k})_{k \in \NN}$ and use Proposition \ref{prop_boundedtransfer}.
\end{proof}

\begin{corollary}\label{cor:uniformtransfer:deloop}
 Assume we are in the same situation as in Corollary \ref{cor_boundeduniformtransfer}. For every $n \in \NN$, there is an exact functor
 \begin{equation*}
  \transfer \colon \cR^G_f(W,\JJ(X)(n)) \to \cR^G_f(W,\JJ((T_k)_k,X)(n))
 \end{equation*}
 such that each composition $P_k \circ \transfer$ of $\transfer$ with the functor $P_k \colon \cR^G_f(W,\JJ((T_k)_k,X)(n)) \to \cR^G_f(W,\JJ(X)(n))$ from Section \ref{sec_outline} induces the identity on $K$-groups.
\end{corollary}
\begin{proof}
 The definitions above generalize to $\cR^G_f(W,\JJ(X)(n))$. The statements follow from the case $n=0$ because the $\RR^n$-coordinate remains untouched.
\end{proof}

Corollary \ref{cor:uniformtransfer:deloop} even provides us with a sequence of functors which induces a map on non-connective algebraic $K$-theory spectra. This map splits the map induced by each functor $P_k$ up to homotopy.

\section{The ``squeezing'' theorem}\label{sec_squeezing}
The main result of this section is the following analog of \cite[Theorem 7.2]{BLR2008},
which is the final ingredient for the proof of Theorem \ref{thm_fhmethod}.
We freely use the notation from Section \ref{sec_outline}.

\begin{theorem}[Squeezing Theorem]\label{thm_squeezing}
 Let $G$ be a countable discrete group, and let $\cF$ be a family of subgroups of $G$.
 Let $(E_k)_k$ be a sequence of $G$-simplicial complexes whose isotropy lies in $\cF$.
 Suppose that there is some $N$ such that the dimension of $E_k$ is at most $N$ for all $k$.
 Equip $E_k$ with the metric $k \cdot d^{\ell^1}$. Then the inclusion functor induces a weak equivalence
 \begin{equation*}
  \KK^{-\infty}\Big(\finprod \cR^G_f(W,\JJ(E_k \times G,E_\cF G))\Big) \xrightarrow{\sim} \KK^{-\infty}(\cR^G_f(W,\JJ((E_k \times G)_k,E_\cF G))).
 \end{equation*}
\end{theorem}

For the purposes of this section, abbreviate
\begin{equation*}
 \begin{split}
  \cB_{fin}((M_k)_k) &:= \finprod \cR^G_f(W,\JJ(M_k,E_\cF G)), \\
  \cB((M_k)_k) &:= \cR^G_f(W,\JJ((M_k)_k,E_\cF G))
 \end{split}
\end{equation*}
for any sequence $(M_k)_k$ of metric spaces with free, isometric $G$-action (in our case $M_k = E_k \times G$).
Observe that $\cB_{fin}((M_k)_k)$ can be described as the full subcategory of objects in $\cB((M_k)_k)$ with support on a finite disjoint union.

Let $Y = (Y_k)_k$ be an object in $\cB((M_k)_k)$. For $K \in \NN$, define $(Y_k)_{k > K}$ to be the sequence
$(X_k)_k$ with $X_k = *$ for $k \leq K$ and $X_k = Y_k$ for $k > K$.
Define $h^{fin}\cB((M_k)_k)$ to be the category of those morphisms $f = (f_k)_k \colon (Y^1_k)_k \to (Y^2_k)_k$
for which there is some $K > 0$ such that the induced morphism $(f_k)_{k > K} \colon (Y^1_k)_{k > K} \to (Y^2_k)_{k > K}$
is an $h$-equivalence in $\cB((M_k)_k)$. Note that this is a stronger condition than requiring
$f_k$ to be a controlled homotopy equivalence for all $k > K$.
Using the Modified Fibration Theorem \ref{prop_modifiedfibrationtheorem}, there is a homotopy fiber sequence
\begin{equation*}
 hS_\bullet\cB((M_k)_k)^{h^{fin}} \to hS_\bullet\cB((M_k)_k) \to h^{fin}S_\bullet\cB((M_k)_k).
\end{equation*}
It is straightforward to check that this homotopy fiber sequence can be delooped, and that the Approximation Theorem applies
to the inclusion $\cB_{fin}((M_k)_k) \hookrightarrow \cB((M_k)_k)^{h^{fin}}$.
We conclude that there is a homotopy fiber sequence
\begin{equation}\label{squeezing.cofiniteequivalences}
 \KK^{-\infty}\big( \cB_{fin}((M_k)_k),h \big) \to \KK^{-\infty}\big( \cB((M_k)_k),h \big) \to \KK^{-\infty}\big( \cB((M_k)_k),h^{fin} \big).
\end{equation}
Consequently, it suffices to show that $\KK^{-\infty}( \cB((E_k \times G)_k),h^{fin} )$ is weakly contractible in order to prove Theorem \ref{thm_squeezing}.
As in \cite{BLR2008}, the proof is by induction on $N$. We need the following lemma.

\begin{lemma}\label{lem.squeezingsimplices}
 Suppose that $(\Delta_k)_k$ is a sequence of $G$-simplicial complexes of the form
 \begin{equation*}
  \Delta_k = \coprod_{i \in I_k} G/H_i \times \Delta^N
 \end{equation*}
 such that $H_i \in \cF$ for all $i$. Equip $\Delta_k$ with the metric which assigns distance $\infty$ to points in different path components,
 and equals $k \cdot d^{\ell^1}$ for points on the same simplex.
 
 Then $\KK^{-\infty}(\cB_{fin}((\Delta_k \times G)_k),h)$, $\KK^{-\infty}(\cB((\Delta_k \times G)_k),h)$ and $\KK^{-\infty}(\cB((\Delta_k \times G)_k),h^{fin})$ are all weakly contractible.
\end{lemma}
\begin{proof}
 It is shown in \cite[proof of Proposition 7.4]{BLR2008} that there is a sequence of maps on the underlying control spaces such that Theorem \ref{thm_eilenbergswindle} applies.
\end{proof}

\begin{corollary}\label{cor:squeezing-induction-start}
 Theorem \ref{thm_squeezing} holds for $N = 0$.
\end{corollary}
\begin{proof}
 Since each $E_k$ is $0$-dimensional, it is a disjoint union of transitive $G$-sets, $E_k = \coprod_{i \in I_k} G/H_i$ with $H_i \in \cF$. Define $\Delta_k$ to be the simplicial complex $E_k$, equipped with the metric from Lemma \ref{lem.squeezingsimplices}. There is an exact functor $F \colon \cB((\Delta_k \times G)_k) \to \cB((E_k \times G)_k)$. We claim that this functor induces a weak equivalence
 \begin{equation*}
  \KK^{-\infty}(F) \colon \KK^{-\infty}(\cB((\Delta_k \times G)_k),h^{fin}) \xrightarrow{\sim} \KK^{-\infty}(\cB((E_k \times G)_k,h^{fin}).
 \end{equation*}
 Obviously, $F$ maps $h^{fin}$-equivalences to $h^{fin}$-equivalences. We claim that $F$ satisfies the Approximation Property.
 
 Let $f \colon (Y^1_k)_k \to (Y^2_k)_k$ be a morphism in $\cB((\Delta_k)_k)$ such that $F(f)$ is an $h^{fin}$-equivalence. Since we require uniform metric control, there is some $K > 0$ such that $(f_k)_{k > K}$ is an $h$-equivalence which is $0$-controlled over $(E_k)_k$. We can assume that $(f_k)_{k > K}$ has an inverse which is $0$-controlled over $(E_k)_k$, and that we can find homotopies between the compositions which are $0$-controlled over $(E_k)_k$ as well. Hence, $f$ is also an $h^{fin}$-equivalence in $\cB((\Delta_k \times G)_k)$.
 
 For the second part of the Approximation Property, let $Y^1 = (Y^1_k)_k \in \cB((\Delta_k \times G)_k)$, $Y^2 = (Y^2_k)_k \in \cB((E_k \times G)_k)$, and let $f = (f_k)_k \colon F((Y^1_k)_k) \to (Y^2_k)_k$ be a morphism in $\cB((E_k \times G)_k)$. Then there is some $K > 0$ such that $(Y^2_k)_{k > K}$ and $(f_k)_{k > K}$ are $0$-controlled over $(E_k)_k$. Define $Y = (Y_k)_k$ via $Y_k := Y^1_k$ for $k \leq K$ and $Y_k := Y^2_k$ for $k > K$. Then $f$ factors canonically as $Y^1 \to Y \to Y^2$, where the first morphism is $0$-controlled over $(E_k)_k$ and the latter morphism is an $h^{fin}$-equivalence. Since $Y$ is also $0$-controlled over $(E_k)_k$, this proves the Approximation Property.
 
 Hence, $\KK^{-\infty}(F)$ is a weak equivalence by the Approximation Theorem. The claim follows from Lemma \ref{lem.squeezingsimplices}.
\end{proof}

Suppose now that Theorem \ref{thm_squeezing} holds for $N$, and let $(E_k)_k$ be a sequence of $G$-simplicial complexes of dimension at most $N+1$. Consider for each $k$ the pushout diagram
\begin{equation}\label{squeezing.attachingsimplices}
 \commsquare{\coprod_{i \in I_k^N} G/H_i \times \partial\Delta^{N+1}}{\skel{E_k}{N}}{\coprod_{i \in I_k^N} G/H_i \times \Delta^{N+1}}{E_k}{}{}{}{}
\end{equation}
describing the attachment of the $(N+1)$-simplices of $E_k$.

\begin{lemma}\label{lem.squeezingexcision}
 Let $N \geq 0$. The commutative square of non-connective $K$-theory spectra
 \begin{equation*}
  \commsquare{\KK^{-\infty}(\big(\cB((\coprod_{i \in I_k^N} G/H_i \times \partial\Delta^{N+1} \times G)_k,h^{fin}\big)}{\KK^{-\infty}\big((\skel{E_k}{N} \times G)_k,h^{fin}\big)}{\KK^{-\infty}\big(\cB((\coprod_{i \in I_k^N} G/H_i \times \Delta^{N+1} \times G)_k,h^{fin}\big)}{\KK^{-\infty}\big((E_k \times G)_k,h^{fin}\big)}{}{}{}{}
 \end{equation*}
 induced by diagram (\ref{squeezing.attachingsimplices}) is a homotopy pullback square of spectra.
\end{lemma}

Lemma \ref{lem.squeezingexcision} provides the induction step: The top left and top right corners of the square from Lemma \ref{lem.squeezingexcision} are weakly contractible by the induction hypothesis. The bottom left corner is weakly contractible by Lemma \ref{lem.squeezingsimplices}. Hence, the bottom right corner is also weakly contractible, and Theorem \ref{thm_squeezing} follows.

In the rest of this section, we prove Lemma \ref{lem.squeezingexcision}.

\begin{lemma}\label{lem.squeezingcofibre}
 Let $(M_k)_k$ be a sequence of metric spaces with free, isometric $G$-action, and let $X_k \subset M_k$ be $G$-invariant, closed subspaces.
 
 Define $X := \coprod_k X_k \times G \times E_\cF G \times [1,\infty[$.
 Let $h^X\cB((M_k)_k)$ be the subcategory of controlled homotopy equivalences away from $X$.
 Let $h^{X,fin}\cB((M_k)_k)$ denote the subcategory of those morphisms $f \colon (Y^1_k)_k \to (Y^2_k)_k$ for which
 there is some $K \in \NN$ such that the induced morphism $(f_k)_{k > K} \colon (Y^1_k)_{k > K} \to (Y^2_k)_{k > K}$ is an $h^X$-equivalence.
  
 Then there is a homotopy fiber sequence
 \begin{equation*}
  \KK^{-\infty}(\cB((X_k)_k),h^{fin}) \to \KK^{-\infty}(\cB((M_k)_k),h^{fin}) \to \KK^{-\infty}(\cB((M_k)_k),h^{X,fin}).
 \end{equation*}
\end{lemma}
\begin{proof}
 Consider the commutative diagram
 \begin{equation*}
  \begin{tikzpicture}
   \matrix (m) [matrix of math nodes, column sep=2em, row sep=2em, text depth=.5em, text height=1em]
   {\KK^{-\infty}( \cB_{fin}((X_k)_k, h )   & \KK^{-\infty}( \cB((X_k)_k, h )   & \KK^{-\infty}( \cB((X_k)_k, h^{fin} )   \\
    \KK^{-\infty}( \cB_{fin}((M_k)_k, h )   & \KK^{-\infty}( \cB((M_k)_k, h )   & \KK^{-\infty}( \cB((M_k)_k, h^{fin} )   \\
    \KK^{-\infty}( \cB_{fin}((M_k)_k, h^X ) & \KK^{-\infty}( \cB((M_k)_k, h^X ) & \KK^{-\infty}( \cB((M_k)_k, h^{X,fin} ) \\};
   \path[->]
    (m-1-1) edge (m-1-2) edge (m-2-1)
    (m-1-2) edge (m-1-3) edge (m-2-2)
    (m-1-3) edge (m-2-3)
    (m-2-1) edge (m-2-2) edge (m-3-1)
    (m-2-2) edge (m-2-3) edge (m-3-2)
    (m-2-3) edge (m-3-3)
    (m-3-1) edge (m-3-2)
    (m-3-2) edge (m-3-3);
  \end{tikzpicture}
 \end{equation*}
 in which all maps are induced by the appropriate inclusion functors. The left and middle columns are homotopy fiber sequences by Theorem \ref{thm_fibresequence}.
 The top and middle rows are instances of the homotopy fiber sequence \eqref{squeezing.cofiniteequivalences}.
 By a straightforward modification of the argument for \eqref{squeezing.cofiniteequivalences}, the bottom row is also a homotopy fiber sequence. Hence, the right column is a homotopy fiber sequence as claimed.
\end{proof}

\begin{proof}[Proof of Lemma \ref{lem.squeezingexcision}]
 Let $\Delta_k := \coprod_{i \in I_k^N} G/H_i \times \Delta^{N+1}$ and $\partial\Delta_k := \coprod_{i \in I_k^N} G/H_i \times \partial\Delta^{N+1}$. Equip $\Delta_k$ with the metric from Lemma \ref{lem.squeezingsimplices}. The inclusion of metric spaces $\partial\Delta_k \times G \subset \Delta_k \times G$ gives rise to a class of weak equivalences $h^{\partial,fin}\cB((\Delta_k \times G)_k)$ and a corresponding homotopy fiber sequence as in Lemma \ref{lem.squeezingcofibre}. Similarly, $\skel{E_k}{N} \times G \subset E_k \times G$ gives rise to a class of weak equivalences $h^{N,fin}\cB((E_k \times G)_k)$ and a corresponding homotopy fiber sequence.
 
 Diagram \eqref{squeezing.attachingsimplices} induces a map between these homotopy fiber sequences. To prove the lemma, it suffices to show that the induced map on the homotopy cofibers
 \begin{equation*}
  \KK^{-\infty}(\cB((\Delta_k \times G)_k), h^{\partial,fin}) \to \KK^{-\infty}(\cB((E_k \times G)_k), h^{N,fin})
 \end{equation*}
 is a weak equivalence. Note that this map is induced by an exact functor $F$, namely the one induced by the characteristic maps of the $(N+1)$-simplices.
 The claim is that the Approximation Theorem applies again, but as for Corollary \ref{cor:squeezing-induction-start}, we have to prove both parts of the Approximation Property.
 
 We start with a preliminary observation: Let $(Y^1,\kappa^1)$ and $(Y^2,\kappa^2)$
 be objects in $\cB((E_k \times G)_k)$, and let $f \colon Y^1 \to Y^2$ be a morphism.
 Suppose that $f$ is $R$-controlled, and let $e$ be a cell in $Y^1_k$ such that the
 $E_k$-component $x$ of $\kappa^1(e)$ is a point in an $(N+1)$-simplex $\sigma$.
 Let $e'$ be a cell in $Y^2_k$ with $e' \subset \gen{f(e)}$, and suppose that the $E_k$-component $y$ of $\kappa^2(e')$ does not lie in $\sigma$.
 Then $k d^{\ell^1}(x,y) \leq R$. According to \cite[Lemma 7.15]{BLR2008}, there is a point $z$ on the boundary of $\sigma$ such that $k d^{\ell^1}(x,z) \leq 2R$.
 Hence, if the distance of $x$ to the boundary of $\sigma$ is greater than $2R$, then for every cell $e'$ in $Y^2_k$ with $e' \subset \gen{f(e)}$,
 the $E_k$-component of $\kappa^2(e')$ also lies in $\sigma$.
 
 Let us now turn to the first part of the Approximation Property. Let $f \colon Y^1 \to Y^2$ be a morphism in $\cB((\Delta_k \times G)_k)$ such that $F(f)$ is an $h^{N,fin}$-equivalence.
 Choose $R > 0$ such that $Y^1$, $Y^2$ and $f$ are all $R$-controlled, and such that further $F(f)$ has a (partially defined) homotopy inverse and homotopies which are also $R$-controlled.
 Let $Y^1_k(6R)$ be the subobject of $Y^1_k$ spanned by those cells $e \in \cells Y^1_k$ such that the $\Delta_k$-component of $\kappa^1(e)$ has distance at least $6R$ to $\partial\Delta_k$.
 If $e$ is any cell in $Y^1_k$, the $E_k$-component of $\kappa^1(e)$ has distance at least $5R$ to the boundary of the $(N+1)$-simplex in which it lies;
 to see this, combine the preliminary observation with the fact that $Y^1$ is $R$-controlled.
 Since $Y^1(6R) := (Y^1_k(6R))_k \subset Y^1$ is cofinal away from $\coprod_k \partial\Delta_k \times G \times E_\cF G \times [1,\infty[$, the inclusion is an $h^\partial$-equivalence. In particular, it is an $h^{\partial,fin}$-equivalence.
 We can similarly define a subcomplex $Y^2(4R) \subset Y^2$, and this inclusion is also an $h^{\partial,fin}$-equivalence.
 
 Since $f$ is $R$-controlled, there is an induced morphism $f' \colon Y^1(6R) \to Y^2(4R)$.
 The morphism $F(f')$ is still an $h^{N,fin}$-equivalence; the inverse and homotopies arise by restricting the inverse and homotopies of $F(f)$ to appropriate cofinal subcomplexes.
 Hence, they are still $R$-controlled. It follows that they do not cross the boundaries of simplices, so they lift to $\cB((\Delta_k \times G)_k)$.
 This shows that $f'$ is an $h^{\partial,fin}$-equivalence.
 
 For the second part of the Approximation Property, let $Y^1 \in \cB((\Delta_k \times G)_k)$, $Y^2 \in \cB((E_k \times G)_k)$, and let $f \colon F(Y^1) \to Y^2$ be a morphism in $\cB((E_k \times G)_k)$.
 The argument is similar to the first part. Again, choose $R > 0$ such that $Y_1$, $Y_2$ and $f$ are all $R$-controlled.
 Then $Y^1(6R)$, defined as before, is a subcomplex of $Y^1$ which is cofinal away from $\coprod_k \partial\Delta_k \times G \times E_\cF G \times [1,\infty[$;
 similarly, $Y^2(4R)$ is a subcomplex of $Y^2$ which is cofinal away from $\coprod_k \skel{E_k}{N} \times G \times E_\cF G \times [1,\infty[$.
 Moreover, $Y^2(4R)$ is supported on the interiors of the $(N+1)$-simplices, and if $e$ is a cell in $Y^2(4R)$, the subcomplex $\langle e \rangle$ spanned by $e$ is based on the same simplex as $e$ by the preliminary observation. Since the characteristic maps of the $(N+1)$-simplices are homeomorphisms on the interiors (and restrict to isometries on individual simplices), we can lift $Y^2(4R)$ to an object in $\cB((\Delta_k \times G)_k)$. Now define $Y$ to be the pushout of $Y^1 \leftarrowtail Y^1(6R) \xrightarrow{f|_{Y^1(6R)}} Y^2(4R)$ in $\cB((\Delta_k \times G)_k)$. Since the inclusion $Y^1(6R) \rightarrowtail Y^1$ is an $h^{\partial,fin}$-equivalence,
 the canonical inclusion $Y^2(4R) \rightarrowtail Y$ is also an $h^{\partial,fin}$-equivalence. Since $F$ is exact, $F(Y)$ is the pushout of $F(Y^1) \leftarrowtail F(Y^1(6R)) \to F(Y^2(4R))$.
 Let $g \colon F(Y) \to Y^2$ be the map induced by the universal property of the pushout. Since
 \begin{equation*}
  \commtriangle{Y^2(4R)}{Y^2}{F(Y)}{\sim}{\sim}{g}
 \end{equation*}
 commutes, $g$ is an $h^{N,fin}$-equivalence by the Saturation Axiom. We conclude that
 \begin{equation*}
  \commtriangle{F(Y^1)}{Y^2}{F(Y)}{f}{}{g}
 \end{equation*}
 is the required factorization. So the Approximation Property holds, and we are done. 
\end{proof}

\section{Applications}\label{sec_applications}
To conclude, we turn to some applications of Theorem \ref{thm_fhmethod}.
As an immediate corollary, we obtain Theorem \ref{intro_finitegroups}, which gives a description of the $A$-theory of spaces with finite fundamental group.

\begin{theorem}\label{thm_FJCFiniteGroups}
 Let $V$ be a connected CW-complex with finite fundamental group $G$. Let $\tilde{V}$ be the universal cover of $V$.
 Denote by $\cD$ the family of Dress subgroups of $G$. Then the Davis--L\"uck assembly map
 \begin{equation*}
  \HH^G(E_\cD G; \Aa^{-\infty}_{\tilde{V}}) \to \Aa^{-\infty}(V)
 \end{equation*}
 is a weak equivalence.
\end{theorem}
\begin{proof}
 The group $G$ is Dress--Farrell--Hsiang with respect to $\cD$: For every $\epsilon > 0$,
 choose $\pi = \id_G$ and let $f_D$ be the projection onto a point for all $D \in \cD$.
 Now apply Theorem \ref{thm_fhmethod} with $W = \tilde{V}$.
\end{proof}

Our ultimate goal is the proof of Theorem \ref{intro_poly-z-groups}. Formally, everything we do is very close to the treatment in \cite{BFL2014}. This involves a rather intricate induction process which relies on a number of inheritance properties of the isomorphism conjecture. These will be established along the way. The reader is encouraged to refer to {\it loc.\ cit.} for definitions.

\begin{proposition}[Transitivity Principle]\label{prop_transitivity}
 Let $\cF_0 \subset \cF_1$ be two families of subgroups of $G$.
 Suppose that $G$ satisfies the \fic\ in $A$-theory with respect to $\cF_1$,
 and that every $H \in \cF_1$ satisfies the \fic\ in $A$-theory
 with respect to $\cF_0|_H := \{ H \cap K \mid K \in \cF_0 \}$.
 
 Then $G$ satisfies the \fic\ in $A$-theory with respect to $\cF_0$. 
\end{proposition}
\begin{proof}
 The proof is analogous to the linear case. However, the published proofs (e.g., \cite[Theorem 2.4]{BL2006}, \cite[Theorem 3.3]{BEL2008})
 all rely on the formalism of equivariant homology theories.
 Since we want to avoid a discussion to which extent the homology theories associated to $A$-theory spectra
 form equivariant homology theories, we give a proof using the language of $Or(G)$-spectra.
 
 Let $\EE$ be an arbitrary $Or(G)$-spectrum, and let $E$ be a $G$-CW-complex. Observe that $G/H \times E$
 is naturally $G$-homeomorphic to $\ind_H^G \res^H_G E = G \times_H \res_G^H E$. Induction defines a functor $\ind_H^G \colon Or(H) \to Or(G)$, so we obtain an $Or(H)$-spectrum $\EE \circ \ind_H^G$. The same arguments as in the proof of Proposition 157 of \cite{LR2005} show that there is a natural isomorphism
 \begin{equation*}
  \HH^H(\res_G^H E;\EE \circ \ind_H^G) \cong \HH^G(\ind_H^G \res^H_G E;\EE).
 \end{equation*}
 Now let $\Aa^{-\infty}_W$ be the $Or(G)$-spectrum from Section \ref{sec_assembly} associated to a free $G$-CW-complex $W$.
 Since $W \times_G (\ind_H^G H/L) \cong \res_G^H W \times_H H/L$, we have
 \begin{equation*}
  \HH^G(G/H \times E;\Aa^{-\infty}_W) \cong \HH^H(\res_G^H E;\Aa^{-\infty}_{\res_G^H W}).
 \end{equation*}
 In particular, the map $\HH^G(G/H \times E_{\cF_0} G;\Aa^{-\infty}_W) \to \HH^G(G/H;\Aa^{-\infty}_W)$ induced by the projection map
 is weakly equivalent to the map
 \begin{equation*}
  \HH^H(\res_G^H E_{\cF_0} G;\Aa^{-\infty}_{\res_G^H W}) \to \HH^H(H/H;\Aa^{-\infty}_{\res_G^H W}) = \Aa^{-\infty}(\res_G^H W/H).
 \end{equation*}
 Since $\res_G^H W$ is a free $H$-CW-complex and $\res_G^H E_{\cF_0} G = E_{\cF_0|_H} H$, this map is an equivalence for all $H \in \cF_1$ by assumption.
 It follows that the maps
 \begin{equation*}
  \HH^G \big(\coprod_i E_{\cF_0} G \times G/H_i \times D^n;\Aa^{-\infty}_W \big) \to \HH^G \big(\coprod_i G/H_i \times D^n;\Aa^{-\infty}_W \big)
 \end{equation*}
 are weak equivalences whenever $H_i$ lies in $\cF_1$ for all $i$ because $G/H_i \times D^n$ is $G$-homotopy equivalent to $G/H_i$
 and the homology theory under consideration commutes with coproducts. By an induction along the skeleta,
 it follows that the projection map $E_{\cF_0} G \times X \to X$ induces an equivalence in $\HH^G(-,\Aa^{-\infty}_W)$
 for every finite-dimensional $G$-CW-complex $X$ whose isotropy groups lie in $\cF_1$. Since homology commutes with filtered colimits, the same holds for all $G$-CW-complexes $X$ whose isotropy groups lie in $\cF_1$.
 
 In particular, we can pick $X = E_{\cF_1} G$. Then $E_{\cF_0} G \times E_{\cF_1} G$ is $G$-homotopy equivalent to $E_{\cF_0} G$,
 so we conclude that the $G$-map $E_{\cF_0} G \to E_{\cF_1} G$ (which is unique up to $G$-homotopy) induces a weak equivalence.
 This implies the claim.
\end{proof}

\begin{proposition}\label{prop_fibering}
 Let $\phi \colon K \to G$ be a group homomorphism. Suppose that $G$ satisfies the \fic\
 in $A$-theory with respect to the family $\cF$ of subgroups of $G$.
 
 Then $K$ satisfies the \fic\ in $A$-theory with respect to the family of subgroups
 \begin{equation*}
  \phi^*\cF := \{ \phi^{-1}(H) \mid H \in \cF \}.
 \end{equation*}
\end{proposition}
\begin{proof}
 Let $\tilde{\EE}$ be a functor from the category of $K$-sets to the category of spectra; let $\EE$ denote
 its restriction to $Or(K)$. It has been shown in the proof of Proposition 4.2 of \cite{MR2294225} that there is for every $G$-CW-complex $X$ a weak equivalence
 \begin{equation}\label{eq:prop-fibering-1}
  \HH^K(\res_\phi X;\EE) \cong \map_G(-,X)_+ \wedge_{Or(G)} \map_K(?,\res_\phi -)_+ \wedge_{Or(K)} \tilde{\EE}(?)
 \end{equation}
 which is natural in $X$.
 Let $G/H \in Or(G)$, and let $K \backslash G / H$ denote the orbit space of $\res_\phi G/H$.
 Subject to a choice of (set-theoretic) section $\sigma \colon K \backslash G /H \to G$ of the obvious projection map,
 there is an isomorphism
 \begin{equation*}
  \begin{split}
   \res_\phi G/H &\cong \coprod_{KgH \in K \backslash G/H} T_\sigma(KgH), \text{ where } \\
   T_\sigma(KgH) &:= K/(K \cap \sigma(KgH)H\sigma(KgH)^{-1}).
  \end{split}
 \end{equation*}
 This isomorphism gives rise to a commutative diagram
 \begin{equation*}
  \commsquare{\big( \bigvee_{KgH \in K \backslash G/H} \map_K(?,T_\sigma(KgH))_+ \big) \wedge_{Or(K)} \EE(?)}{\map_K(?,\res_\phi G/H)_+ \wedge_{Or(K)} \EE(?)}{\bigvee_{KgH \in K \backslash G/H} \EE(T_\sigma(KgH))}{\tilde{\EE}(\res_\phi G/H)}{\cong}{}{}{}
 \end{equation*}
 in which the vertical maps are induced by evaluating $\tilde{\EE}$. The right vertical map is natural in $G/H$.
 The left vertical map is easily seen to be a weak equivalence.
 Whenever $\tilde{\EE}$ commutes with coproducts,, the lower horizontal map is a weak equivalence. In this case, the right vertical map is also a weak equivalence, and we obtain a weak equivalence of $Or(G)$-spectra
 \begin{equation}\label{eq:prop-fibering-2}
  \map_K(?,\res_\phi -)_+ \wedge_{Or(K)} \EE(?) \simeq \tilde{\EE} \circ \res_\phi.
 \end{equation}
 The weak equivalences \eqref{eq:prop-fibering-1} and \eqref{eq:prop-fibering-2} combine to a weak equivalence, natural in $X$,
 \begin{equation*}
  \HH^K(\res_\phi X;\EE) \simeq \HH^G(X;\tilde{\EE} \circ \res_\phi).
 \end{equation*}
 Let $W$ be a free $K$-CW-complex. Since $\Aa^{-\infty}_W$ extends to a functor on all $K$-sets
 and commutes with coproducts (see Lemma \ref{lem_a-coproducts} below), we obtain a natural equivalence
 \begin{equation*}
  \HH^K(\res_\phi E_\cF G;\Aa^{-\infty}_W) \simeq \HH^G(E_\cF G;\Aa^{-\infty}_W \circ \res_\phi).
 \end{equation*}
 Since $\ind_\phi W = W \times_K \res_\phi G$ is a free $G$-CW-complex and $W \times_K \res_\phi G/H \cong \ind_\phi W \times_G G/H$, we have a natural weak equivalence of $Or(G)$-spectra $\Aa^{-\infty}_W \circ \res_\phi \simeq \Aa^{-\infty}_{\ind_\phi W}$. Hence, there is a natural weak equivalence
 \begin{equation*}
  \HH^K(\res_\phi E_\cF G;\Aa^{-\infty}_W) \simeq \HH^G(E_\cF G;\Aa^{-\infty}_{\ind_\phi W}).
 \end{equation*}
 Observe that $\res_\phi E_\cF G = E_{\phi^*\cF} K$. We conclude that the assembly map
 \begin{equation*}
  \HH^K(E_{\phi^*\cF}K;\Aa^{-\infty}_W) \to \HH^K(K/K;\Aa^{-\infty}_W)
 \end{equation*}
 is weakly equivalent to the assembly map
 \begin{equation*}
  \HH^G(E_\cF G;\Aa^{-\infty}_{\ind_\phi W}) \to \HH^G(G/G;\Aa^{-\infty}_{\ind_\phi W}).
 \end{equation*}
 The latter map is assumed to be a weak equivalence, so we are done.
\end{proof}

\begin{lemma}\label{lem_a-coproducts}
 Let $W$ be a CW-complex, and let $W = \coprod_{i \in I} W_i$ be a decomposition of $W$ into subspaces.
 Then the natural map
 \begin{equation*}
  \bigvee_{i \in I} \Aa^{-\infty}(W_i) \to \Aa^{-\infty}(W)
 \end{equation*}
 is a weak equivalence.
\end{lemma}
\begin{proof}
 Let $Y$ be a CW-complex relative $W$ together with a retraction $r \colon Y \to W$.
 Then the partition $W = \coprod_i W_i$ induces a partition of $Y$ into subcomplexes $Y = \coprod_i Y_i$,
 where $Y_i := r^{-1}(W_i)$.
 Similarly, every morphism of retractive spaces $f \colon Y^1 \to Y^2$ over $W$
 decomposes into a coproduct $f = \coprod_i f_i$ since $f$ is compatible with the retractions.
  Restricting to finite objects, this shows that there
 is an isomorphism
 \begin{equation*}
  \colim_{J \subset I \text{ finite}} \prod_{i \in J} \cR_f(W_i,\fT(*)(n)) \xrightarrow{\cong} \cR_f(W,\fT(*)(n)).
 \end{equation*}
 Here we have used that the image of the retraction of a finite object intersects only finitely many path components of $W$.
 It follows that the map of spectra $\bigvee_{i \in I} \Aa^{-\infty}(W_i) \to \Aa^{-\infty}(W)$
 is a levelwise equivalence.
\end{proof}

\begin{corollary}\label{cor_subgroups}
 Let $G$ be a discrete group and let $H \leq G$ be a subgroup.
 If $G$ satisfies the \fic\ in $A$-theory with respect to the family $\cF$, then $H$ satisfies the \fic\ in $A$-theory with respect to $\cF|_H$.
\end{corollary}
\begin{proof}
 Apply Proposition \ref{prop_fibering} to the inclusion $H \hookrightarrow G$.
\end{proof}

\begin{corollary}\label{cor_epis}
 Let $\pi \colon G \to Q$ be a surjective group homomorphism. Suppose that $Q$ satisfies the
 \ffjc\ in $A$-theory, and that for every virtually cyclic subgroup $V \leq Q$,
 the preimage $\pi^{-1}(V)$ satisfies the \ffjc\ in $A$-theory.
 
 Then $G$ satisfies the \ffjc\ in $A$-theory.
\end{corollary}
\begin{proof}
 Note that $\pi^*\cV\cC yc = \{ \pi^{-1}(V) \mid V \leq Q \text{ virtually cyclic} \}$.
 Hence, the claim is a combination of Proposition \ref{prop_fibering} and the Transitivity Principle \ref{prop_transitivity}.
\end{proof}

\begin{corollary}\label{cor_episwithfinitekernel}
 Let $\pi \colon G \to Q$ be a surjective group homomorphism with finite kernel.
 If $Q$ satisfies the \ffjc\ in $A$-theory, then so does $G$.
\end{corollary}
\begin{proof}
 If $V \leq Q$ is virtually cyclic, then $\pi^{-1}(V)$ is also virtually cyclic, and thus $G$ satisfies the conjecture by Corollary \ref{cor_epis}.
\end{proof}

The next two statements and their proofs are analogous to \cite[Section 3.2 \& 3.3]{BFL2014}. We only sketch their proofs and refer to {\it loc.~cit.} for details.

\begin{lemma}[{cf.\ \cite[Lem.~3.15]{BFL2014}}]\label{lem_crystallographic}
 Let $\Gamma$ be a crystallographic group of virtual cohomological dimension $2$ which possesses a normal, infinite cyclic subgroup.
 Then $\Gamma$ satisfies the \ffjc\ in $A$-theory.
\end{lemma}
\begin{proof}
 Do an induction on the order of the smallest finite group $F$ such that there is
 a short exact sequence $1 \to \ZZ^2 \to \Gamma \to F \to 1$. Then the claim follows
 from \cite[Lemma 5.2 and Proposition 5.3]{W-TransferReducibilityOfFHGroups} in conjunction
 with the induction hypothesis and the Transitivity Principle \ref{prop_transitivity}.
\end{proof}

\begin{proposition}[{cf.\ \cite[Section 3.3]{BFL2014}}]\label{prop_vfga}
 Let $\Gamma$ be a virtually finitely generated abelian group. Then $\Gamma$ satisfies the \ffjc\ in $A$-theory.
\end{proposition}
\begin{proof}
 We do an induction on the virtual cohomological dimension of $\Gamma$. If $\vcd(\Gamma) \leq 1$, the group $\Gamma$ is
 virtually cyclic and there is nothing to show. So assume $\vcd(\Gamma) \geq 2$.
 Then do a subinduction on the cardinality of the smallest finite group $F$ such that $\Gamma$ admits an
 epimorphism onto $F$ whose kernel is isomorphic to $\ZZ^{\vcd(\Gamma)}$.
 
 Since $\Gamma$ admits a surjection with finite kernel onto a crystallographic group \cite[Lemma 4.2.1]{Quinn2012},
 we may assume by Corollary \ref{cor_episwithfinitekernel} that $\Gamma$ is crystallographic of the same virtual cohomological dimension. Now fix an epimorphism $p \colon \Gamma \twoheadrightarrow F$ onto a finite group $F$ such that the kernel of $p$
 is isomorphic to $\ZZ^{\vcd(\Gamma)}$ and such that the cardinality of $F$ is minimal among all finite groups
 which admit such an epimorphism. By induction and the Transitivity Principle \ref{prop_transitivity},
 it suffices to show that $\Gamma$ satisfies the \fic\ with respect to the family of all virtually finitely generated abelian subgroups of $\Gamma$ whose virtual cohomological dimension is smaller than that of $\Gamma$, or whose virtual cohomological dimension equals $\vcd(\Gamma)$ and which admit an epimorphism onto a finite group whose cardinality is smaller than that of $F$ such that the kernel is isomorphic to $\ZZ^{\vcd(\Gamma)}$.
 
 Suppose that $\Gamma$ possesses a normal, infinite cyclic subgroup $C \unlhd \Gamma$. We want to apply Corollary \ref{cor_epis}.
 Since $\vcd(\Gamma/C) < \vcd(\Gamma)$, the quotient $\Gamma/C$ satisfies the \ffjc. Let $\pi \colon \Gamma \to \Gamma/C$ be the projection. For every virtually cyclic subgroup $V \leq G/C$, the preimage $\pi^{-1}(V)$ has virtual cohomological dimension $2$.
 Again, $\pi^{-1}(V)$ admits a surjection with finite kernel onto a crystallographic group $K$ \cite[Lemma 4.2.1]{Quinn2012}, so we may assume that $\pi^{-1}(V)$ is crystallographic by Corollary \ref{cor_episwithfinitekernel}. Since $\vcd(\pi^{-1}(V)) = 2$ and there exists a normal, infinite cyclic subgroup, it follows from Lemma \ref{lem_crystallographic} that $\pi^{-1}(V)$ satisfies the \ffjc.
 So Corollary \ref{cor_epis} applies.
 
 Suppose that there is no normal, infinite cyclic subgroup in $\Gamma$. Then $\Gamma$ is a Dress--Farrell--Hsiang group with respect
 to a family containing only groups to which the induction hypothesis applies \cite[Proposition 5.4]{W-TransferReducibilityOfFHGroups}.
 The Transitivity Principle \ref{prop_transitivity} implies that $\Gamma$ satisfies the \ffjc.
\end{proof}

Proposition \ref{prop_vfga} is the stepping stone to proving an even slightly stronger version of Theorem \ref{intro_poly-z-groups}, see Theorem \ref{thm_poly-z-wreath} below. Recall that the (unrestricted) wreath product $G_1 \wr G_2$ of a group $G_1$ with another group $G_2$ is the semidirect product $\big( \prod_{G_2} G_1 \big) \rtimes G_2$, where $G_2$ acts on the left factor by left translations.

\begin{definition}
 Let $\cF$ be a family of groups and let $G$ be a discrete group. We say that $G$ \emph{satisfies the \ficw\ in $A$-theory} with respect to $\cF$ if for every finite group $F$, the wreath product $G \wr F$ satisfies the \fic\ in $A$-theory with respect to $\cF$.
 
 If $\cF$ is the family of virtually cyclic groups, we say that $G$ \emph{satisfies the \ffjcw\ in $A$-theory}.
\end{definition}

\begin{corollary}\label{cor_vfga-wreath}
 Every virtually finitely generated abelian group satisfies the \ffjcw\ in $A$-theory.
\end{corollary}
\begin{proof}
 This is an immediate consequence of Proposition \ref{prop_vfga} since the wreath product of a virtually finitely generated abelian group with a finite group is again virtually finitely generated abelian.
\end{proof}

Let us record some additional inheritance properties of the \ficw. The following results have been worked out in \cite{Kuehl2009}, we collect them here for reference and the convenience of the reader.

\begin{lemma}\label{lem_inherit}
 Let $G$, $G_1$, $G_2$ be discrete groups, and let $\cF$ be a family of groups.
 \begin{enumerate}
  \item\label{inherit_subgroups} Let $H \leq G$ be a subgroup. If $G$ satisfies the \ficw\ with respect to $\cF$, then so does $H$.
  \item\label{inherit_supergroups} Let $H \leq G$ be a subgroup of finite index. If $H$ satisfies the \ficw\ with respect to $\cF$, so does $G$.
  \item\label{inherit_products} If $G_1$ and $G_2$ satisfy the \ffjcw, so does $G_1 \times G_2$.
  \item\label{inherit_transitivity} Suppose $G$ satisfies the \ficw\ with respect to $\cF$, and that every subgroup $H \leq G$ which lies in $\cF$
   satisfies the \ffjcw. If $\cF$ is closed under taking quotients, then $G$ satisfies the \ffjcw.
  \item\label{inherit_epis} Let $\pi \colon G \twoheadrightarrow Q$ be a surjective homomorphism.
   Suppose that $Q$ satisfies the \ffjcw, and that for every virtually cyclic subgroup $V \leq Q$ the preimage $\pi^{-1}(V)$ satisfies the \ffjcw.
   Then $G$ satisfies the \ffjcw.
  \item\label{inherit_episwithfinitekernel} Let $\pi \colon G \twoheadrightarrow Q$ be a surjective homomorphism with finite kernel.
   If $Q$ satisfies the \ffjcw, so does $G$.
 \end{enumerate}
\end{lemma}
\begin{proof}
 Claim \eqref{inherit_subgroups} is a consequence of Corollary \ref{cor_subgroups} since $H \wr F$ is a subgroup of $G \wr F$ for every group $F$.
 
 For claim \eqref{inherit_supergroups}, assume first that $H$ is normal in $G$. Set $F := G/H$.
 Choose a set-theoretic section $s \colon F \to G$ of the projection map $\pi \colon G \twoheadrightarrow F$.
 For $g \in G$ and $f \in F$ define
 \begin{equation*}
  h(g,f) := s(f)^{-1}gs(\pi(g)^{-1}f).
 \end{equation*}
 Then $g \mapsto ((h(g,f))_f,\pi(g))$ defines a monomorphism $G \hookrightarrow H \wr F$.
 Thus, for every finite group $F'$, the wreath product $G \wr F'$ is a subgroup of $(H \wr F) \wr F'$. Since $(H \wr F) \wr F'$ itself embeds into $H \wr (F \wr F')$ \cite[Lemma 1.21]{Kuehl2009}, the claim follows from part \eqref{inherit_subgroups}.
 If $H$ is not normal, part \eqref{inherit_subgroups} allows us to replace $H$ by $\bigcap_{g \in G} gHg^{-1}$.
 
 For part \eqref{inherit_products}, observe that $(G_1 \times G_2) \wr F$ is a subgroup of $(G_1 \wr F) \times (G_2 \wr F) =: \Gamma$.
 By part \eqref{inherit_subgroups}, it suffices to check that the latter group satisfies the \ffjc.
 Consider the projection map $p_1 \colon \Gamma \to G_1 \wr F$. We want to apply Corollary \ref{cor_epis},
 so we need to check that $V \times (G_2 \wr F)$ satisfies the \ffjc\ for every virtually cyclic subgroup $V$ of $G_1 \wr F$.
 This can be done by another application of Corollary \ref{cor_epis}. The target of the projection map $p_2 \colon V \times (G_2 \wr F) \to G_2 \wr F$ satisfies the \ffjc, so the only thing left to verify is that every product $V \times V'$ of virtually cyclic groups
 satisfies the \ffjc. Since the product of two virtually cyclic groups is virtually finitely generated abelian,
 this is true by Proposition \ref{prop_vfga}.
 
 Let us turn to part \eqref{inherit_transitivity}. Consider a wreath product $G \wr F$, where $F$ is finite.
 Our goal is to apply the Transitivity Principle, so we need to check that every subgroup $H \leq G \wr F$
 which lies in $\cF$ satisfies the \ffjc. Let $H$ be such a subgroup. Since $H' := H \cap \big( \prod_F G \big)$ is normal in $H$
 and has finite index, it suffices to show that $H'$ satisfies the \ffjcw\ by part \eqref{inherit_supergroups}.
 Observe that $H' \in \cF$. Let $H_f$ denote the image of $H'$ under the projection map $\big( \prod_F G \big) \to G$ onto the factor indexed by $f \in F$.
 Then $H'$ embeds into $\prod_{f \in F} H_f$. Since $\cF$ is closed unter taking quotients, $H_f$ satisfies the \ffjcw,
 and so does the product $\prod_{f \in F} H_f$ by part \eqref{inherit_products}. Now part \eqref{inherit_subgroups}
 implies that $H'$ satisfies the \ffjcw, and we are done.
 
 For part \eqref{inherit_epis} of the lemma, observe that $\pi$ induces a surjective homomorphism $\pi_F \colon G \wr F \twoheadrightarrow Q \wr F$
 for every finite group $F$. The quotient $Q \wr F$ satisfies the \ffjc\ by assumption. We want to apply Corollary \ref{cor_epis}.
 So let $V \leq Q \wr F$ be virtually cyclic. In order to show that $\pi_F^{-1}(V)$ satisfies the \ffjc,
 it suffices to show that $\tilde{V} := \pi_F^{-1}(V) \cap \big( \prod_F G \big)$ satisfies the \ffjcw.
 Denote by $V_f$ the image of $V \cap \big( \prod_F Q \big)$ under the projection $\prod_F Q  \to Q$ onto the factor indexed by $f \in F$.
 Then $\tilde{V}$ embeds into $\prod_{f \in F} \pi^{-1}(V_f)$. Since $V_f$ is a virtually cyclic subgroup of $Q$,
 the preimage $\pi^{-1}(V_f)$ satisfies the \ffjcw\ by assumption, and hence so does $\prod_{f \in F} \pi^{-1}(V_f)$ by
 part \eqref{inherit_products}. Then $\tilde{V}$ satisfies the \ffjcw\ by part \eqref{inherit_subgroups}.
 
 The last part of the lemma follows from part \eqref{inherit_epis} because the preimage of each virtually cyclic subgroup of $Q$
 is again virtually cyclic (and these satisfy the \ffjcw\ by Corollary \ref{cor_vfga-wreath}.
\end{proof}

In analogy to \cite[Proposition 2.19]{W-FJCSolvableGroups}, we are going to show next that the Dress--Farrell--Hsiang condition \ref{def_dfhgroup} is well-behaved with respect to wreath products with finite groups. Let $G$ be a group, $\cF$ a family of subgroups
and $\Phi$ a finite group. Denote by $\cF^{\wr \Phi}$ the family of subgroups of $G \wr \Phi$ consisting of those
groups which contain a finite-index subgroup of the form $\prod_{\psi \in \Phi} H_\psi$, where each $H_\psi$ lies in $\cF$.

Recall the following construction of the product of simplicial complexes. Let $E_1, \dots, E_k$ be (abstract) ordered simplicial complexes. Then define $E_1 \otimes \dots \otimes E_k$ to be the simplicial complex whose $r$-simplices are ascending chains $(e^0_1,\dots,e^0_k) < \dots < (e^r_1,\dots,e^r_k)$ with respect to the lexicographic ordering such that $\{ e^0_i,\dots,e^r_i \}$ is a simplex in $E_i$ for all $i$. The map $\abs{E_1 \otimes \dots \otimes E_k} \to \abs{E_1} \times \dots \times \abs{E_k}$ induced by the obvious projections $E_1 \otimes \dots \otimes E_k \to E_i$ is a homeomorphism (with respect to the topologies induced by the $\ell^1$-metric).

\begin{proposition}\label{prop_dfhandwreath}
 Let $G$ be a discrete group and let $\cF$ be a family of subgroups. Let $S$ be a finite, symmetric generating set such that
 $(G,S)$ is a Dress--Farrell--Hsiang group with respect to $\cF$.
 
 Then there is for every finite group $\Phi$ a generating set $S^{\wr \Phi}$ of $G \wr \Phi$ such that
 $(G \wr \Phi, S^{\wr\Phi})$ is a Dress--Farrell--Hsiang group with respect to $\cF^{\wr \Phi}$.
\end{proposition}
\begin{proof}
 We start with a preliminary observation. Let $\Phi$ be a finite group, and let $\pi \colon G \twoheadrightarrow F$ be an epimorphism onto some finite group $F$. Then $\pi$ induces a surjective homomorphism $\pi^{\wr \Phi} \colon G \wr \Phi \twoheadrightarrow F \wr \Phi$ given by $((g_\psi)_\psi,\phi) \mapsto ((\pi(g_\psi))_\psi,\phi)$.
 
 Let $H \leq F \wr \Phi$ be a Dress group. Let $\Phi_H$ be the image of $H$ under the canonical projection $F \wr \Phi \twoheadrightarrow \Phi$, and let $H_\xi$ denote the image of $H \cap \big( \prod_\Phi F \big)$ under the map $p_\xi \colon \big( \prod_\Phi F \big) \to F$ given by projection onto the $\xi$-th component.
 Since the class of Dress groups is closed under taking quotients, each $H_\xi$ is a Dress group.
 For each $\phi \in \Phi_H$, pick a preimage $\kappa^\phi = ((\kappa^\phi_\psi)_\psi,\phi) \in H$.
 Choose a section $s \colon \Phi_H \backslash \Phi \to \Phi$ of the obvious projection map such that $s(\Phi_H) = 1$.
 Now define
 \begin{equation*}
  \kappa := \big((\kappa^{s(\Phi_H\psi)\psi^{-1}}_{s(\Phi_H\psi)})_\psi,1 \big) \in F \wr \Phi,
 \end{equation*}
 and let $\hat{H}$ denote the group
 \begin{equation*}
  \Big( \prod_{\Phi_H\psi \in \Phi_H \backslash \Phi} \Big( \prod_{\psi' \in \Phi_H\psi} H_{s(\Phi_H\psi)}  \Big)\Big) \rtimes \Phi_H,
 \end{equation*}
 where $\Phi_H$ acts on the left hand side by permuting the index set of every factor $\prod_{\Phi_H\psi} H_{s(\Phi_H\psi)}$.
 Observe that $\hat{H}$ is naturally a subgroup of $F \wr \Phi$.
 We claim that $\kappa$ subconjugates $H$ into $\hat{H}$.
 
 To see this, compute first for an arbitary element $((\alpha_\psi)_\psi,\phi) \in H$
 \begin{equation*}
  \begin{split}
   \kappa ((\alpha_\psi)_\psi,\phi) \kappa^{-1}
   &= \big( (\kappa^{s(\Phi_H\psi)\psi^{-1}}_{s(\Phi_H\psi)})_\psi,1 \big) ((\alpha_\psi)_\psi,\phi) \big( ([\kappa^{s(\Phi_H\psi)\psi^{-1}}_{s(\Phi_H\psi)}]^{-1})_\psi,1 \big) \\
   &= \big( (\kappa^{s(\Phi_H\psi)\psi^{-1}}_{s(\Phi_H\psi)} \alpha_\psi [\kappa^{s(\Phi_H\phi^{-1}\psi)\psi^{-1}\phi}_{s(\Phi_H\phi^{-1}\psi)}]^{-1})_\psi, \phi \big) \\
   &= \big( (\kappa^{s(\Phi_H\psi)\psi^{-1}}_{s(\Phi_H\psi)} \alpha_\psi [\kappa^{s(\Phi_H\psi)\psi^{-1}\phi}_{s(\Phi_H\psi)}]^{-1})_\psi, \phi \big).
  \end{split}
 \end{equation*}
 In order to show that this element lies in $\hat{H}$, we need to check that for every $\xi \in \Phi$,
 $\kappa^{s(\Phi_H\xi)\xi^{-1}}_{s(\Phi_H\xi)} \alpha_\xi [\kappa^{s(\Phi_H\xi)\xi^{-1}\phi}_{s(\Phi_H\xi)}]^{-1}$
 lies in $H_{s(\Phi_H\xi)}$. Indeed,
 \begin{equation*}
  \begin{split}
   &\kappa^{s(\Phi_H\xi)\xi^{-1}} ((\alpha_\psi)_\psi,\phi) (\kappa^{s(\Phi_H\xi)\xi^{-1}\phi})^{-1} \\
   &\ = \big( (\kappa^{s(\Phi_H\xi)\xi^{-1}}_\psi)_\psi,s(\Phi_H\xi)\xi^{-1} \big) \big((\alpha_\psi)_\psi,\phi\big) \big( ( [\kappa^{s(\Phi_H\xi)\xi^{-1}\phi}_{s(\Phi_H\xi)\xi^{-1}\phi\psi}]^{-1} )_\psi  ,\phi^{-1}\xi s(\Phi_H\xi)^{-1} \big) \\
   &\ = \big( (\kappa^{s(\Phi_H\xi)\xi^{-1}}_\psi \alpha_{\xi s(\Phi_H\xi)^{-1}\psi} [\kappa^{s(\Phi_H\xi)\xi^{-1}\phi}_\psi]^{-1})_\psi, 1 \big).
  \end{split}
 \end{equation*}
 Since this is an element in $H \cap \big( \prod_\Phi F \big)$, we obtain
 \begin{equation*}
  \kappa^{s(\Phi_H\xi)\xi^{-1}}_{s(\Phi_H\xi)} \alpha_\xi [\kappa^{s(\Phi_H\xi)\xi^{-1}\phi}_{s(\Phi_H\xi)}]^{-1}
  = p_{s(\Phi_H\xi)}(\kappa^{s(\Phi_H\xi)\xi^{-1}} ((\alpha_\psi)_\psi,\phi) [\kappa^{s(\Phi_H\xi)\xi^{-1}\phi}]^{-1}) \in H_{s(\Phi_H\xi)}.
 \end{equation*}
 Hence, $\kappa H \kappa^{-1} \subset \hat{H}$.
 
 Since $(G,S)$ is Dress--Farrell--Hsiang, there is some $N$ as in Definition \ref{def_dfhgroup}. Let $\epsilon' > 0$. Let $\pi = \pi_{\epsilon'} \colon G \twoheadrightarrow F$ be some epimorphism satisfying the conditions in Definition \ref{def_dfhgroup}.
 Define $\pi^{\wr \Phi}$ as above. According to our preliminary observation, it suffices to consider subgroups of $F \wr \Phi$ which have the form
 \begin{equation*}
  H = \Big( \prod_{\Phi_H\psi \in \Phi_H \backslash \Phi} \Big( \prod_{\psi' \in \Phi_H\psi} H_{\Phi_H\psi}) \Big) \Big) \rtimes \Phi_H,
 \end{equation*}
 where $\Phi_H$ is some subgroup of $\Phi$ and each $H_{\Phi_H\psi}$ is a Dress subgroup of $F$.
 Define a generating set $S^{\wr \Phi}$ of $G \wr \Phi$ by
 \begin{equation*}
  S^{\wr \Phi} := \{ ((g_\psi)_\psi,\phi) \mid g_\psi \in S \text{ for all } \psi,\ \phi \in \Phi \}.
 \end{equation*}
 For each $\Phi_H\psi \in \Phi_H \backslash \Phi$, choose a $\pi^{-1}(H_{\Phi_H\psi})$-equivariant map
 $f_{\Phi_H\psi} \colon G \to E_{\Phi_H\psi}$ to a $\pi^{-1}(H_{\Phi_H\psi})$-simplicial complex
 of dimension at most $N$ whose stabilizers lie in $\cF$, and such that $d(f_{\Phi_H\psi}(g),f_{\Phi_H\psi}(g')) \leq \epsilon'$
 whenever $g^{-1}g' \in S$. Define
 \begin{equation*}
  \begin{split}
   f_H \colon G \wr \Phi &\to \prod_{\Phi_H\psi \in \Phi_H \backslash \Phi} \ \prod_{\psi' \in \Phi_H\psi} E_{\Phi_H\psi}  =: E_H \\
   ((g_\psi)_\psi,\phi) &\mapsto \big( (f_{\Phi_H\psi}(g_{\psi'})\big)_{\psi' \in \Phi_H\psi} \big)_{\Phi_H\psi \in \Phi_H \backslash \Phi}.
  \end{split}
 \end{equation*}
 We regard $E_H$ as a simplicial complex via the product construction described previously. Let $H$ act on $E_H$ by
 \begin{equation*}
  ((h_{\psi'})_{\psi' \in \Phi_H\psi})_{\Phi_H\psi},\phi) \cdot ((x_{\psi'})_{\psi' \in \Phi_H\psi})_{\Phi_H\psi} := ( ( h_{\psi'} x_{\phi^{-1}\psi'} )_{\psi' \in \Phi_H\psi} )_{\Phi_H\psi}.
 \end{equation*}
 This induces a $(\pi^{\wr \Phi})^{-1}(H)$-action on $E_H$ by restriction, and $f_H$ is $(\pi^{\wr \Phi})^{-1}(H)$-equivariant with respect to this action.
 Observe that the dimension of $E_H$ is bounded by $\abs{\Phi} N$, and that this number only depends on $\Phi$.
 
 Let $x := ((x_{\psi'})_{\psi' \in \Phi_H\psi})_{\Phi_H\psi}$ be a point in $E_H$. Consider the stabilizer $(\pi^{\wr \Phi})^{-1}(H)_x$.
 The intersection $(\pi^{\wr \Phi})^{-1}(H)_x \cap \big( \prod_\Phi G \big)$ is a finite-index subgroup of $(\pi^{\wr \Phi})^{-1}(H)_x$,
 and is equal to $\prod_{\Phi_H\psi \in \Phi_H \backslash \Phi} \prod_{\psi' \in \Phi_H\psi} H_{x_{\psi'}}$. Since each $H_{x_{\psi'}}$ lies in $\cF$, this shows that the stabilizer of $x$ lies in $\cF^{\wr \Phi}$.
 
 What is left to show is that the map $f_H$ has the desired contracting property.
 So let $g = ((g_\psi)_\psi,\phi)$ and $g' = ((g'_\psi)_\psi,\phi')$ be elements in $G \wr \Phi$ such that $g^{-1}g' \in S^{\wr \Phi}$;
 equivalently, $g^{-1}_{\psi}g'_{\psi} \in S$ for all $\psi \in \Phi$.
 For each $\Phi_H\psi \in \Phi_H \backslash \Phi$ and $\psi' \in \Phi_H\psi$, we have
 \begin{equation*}
  d^{\ell^1}_{E_{\Phi_H\psi}}(f_{\Phi_H\psi}(g_{\psi'}),f_{\Phi_H\psi}(g'_{\psi'})) \leq \epsilon'.
 \end{equation*}
 Let $\epsilon > 0$. By Lemma \ref{lem_metricandproducts} below, $d^{\ell^1}_{E_H}(f_H(g),f_H(g')) \leq \epsilon$ as long as $\epsilon'$  was initially chosen to be small enough.
\end{proof}

\begin{lemma}\label{lem_metricandproducts}
 Let $N$, $K \in \NN$. For every $\epsilon > 0$ there is some $\epsilon' > 0$ such that for every sequence $E_1, \dots, E_K$ of (abstract) ordered simplicial complexes, each of which has dimension at most $N$, the following holds:
 
 Let $E := E_1 \otimes \dots \otimes E_K$. For $x \in \abs{E}$, let $(x_1,\dots,x_K)$ denote the image of $x$ under the canonical map $\abs{E} \to \abs{E_1} \times \dots \times \abs{E_K}$. Denote by $d_i$ the $\ell^1$-metric on $\abs{E_i}$, and let $d$ be the $\ell^1$-metric on $\abs{E}$.
 
 Then for all $x,x' \in \abs{E}$, we have $d(x,x') \leq \epsilon$ whenever $d_i(x_i,x'_i) \leq \epsilon'$.
\end{lemma}
\begin{proof}
 The argument is very similar to the one employed in the proof of \cite[Lemma 5.5]{BLRR2014}.
 Since distances with respect to the $\ell^1$-metric are independent of the ambient complex,
 we may assume that $E_i = \Delta^{2N + 1}$. Consider the composition
 \begin{equation*}
  \prod_{1 \leq i \leq K} \abs{\Delta^{2N + 1}} = \abs{E_1} \times \dots \times \abs{E_K} \xrightarrow{\cong} \abs{E} \subset \abs{\Delta^{(2N+2)^K - 1}}
 \end{equation*}
 of the inverse of the canonical homeomorphism with the inclusion into the full simplex.
 Consider the domain of this map as a metric space by taking the metric $d_\Sigma := \sum_i d_i$ and equip the target with its natural $\ell^1$-metric $d_\Delta$. This map is uniformly continuous since the source is compact; hence, there is some $\epsilon'' > 0$ such that $d_\Delta(x,x') \leq \epsilon$ whenever $d_\Sigma((x_1,\dots,x_K),(x'_1,\dots,x'_K)) \leq \epsilon''$.
 Thus, the claim holds for $\epsilon' := \frac{\epsilon''}{K}$.
\end{proof}

\begin{corollary}\label{cor_dfhandwreath}
 Let $G$ be a discrete group and let $\cF$ be a family of groups such that all groups in $\cF$ satisfy the \ffjcw\ in $A$-theory.
 If there is a finite, symmetric generating set $S$ of $G$ such that $(G,S)$ is a Dress-Farrell--Hsiang group with respect to $\cF$,
 then $G$ satisfies the \ffjcw\ in $A$-theory.
\end{corollary}
\begin{proof}
 Let $\Phi$ be a finite group. By Proposition \ref{prop_dfhandwreath}, the wreath product $G \wr \Phi$ is a Dress--Farrell--Hsiang group
 with respect to $\cF^{\wr \Phi}$, so $G \wr \Phi$ satisfies the \fic\ with respect to $\cF^{\wr \Phi}$.
 Since all groups in $\cF$ satisfy the \ffjcw, parts \eqref{inherit_products} and \eqref{inherit_supergroups}
 of Lemma \ref{lem_inherit} imply that all groups in $\cF^{\wr \Phi}$ satisfy the \ffjc.
 Hence, $G \wr \Phi$ satisfies the \ffjc\ by the Transitivity Principle \ref{prop_transitivity}.
\end{proof}

\begin{theorem}\label{thm_specialaffine}
 Let $\Gamma$ be an irreducible special affine group. Then $\Gamma$ satisfies the \ffjcw\ in $A$-theory.
\end{theorem}
\begin{proof}
 By \cite[Theorem 6.1]{W-TransferReducibilityOfFHGroups}, $\Gamma$ is a Dress--Farrell--Hsiang group with respect to the family of virtually finitely generated abelian groups.
 Since we have already shown that all virtually finitely generated abelian groups satisfy the \ffjcw\ in Corollary \ref{cor_vfga-wreath},
 the theorem is an immediate consequence of Corollary \ref{cor_dfhandwreath}.
\end{proof}

\begin{theorem}[cf.~{\cite[Section 5]{BFL2014}}]\label{thm_poly-z-wreath}
 Let $G$ be a virtually poly-$\ZZ$-group. Then $G$ satisfies the \ffjcw\ in $A$-theory.
\end{theorem}
\begin{proof}
 Repeat the argument on page 377 of \cite{BFL2014}, which relies only on the inheritance properties of the conjecture.
\end{proof}

Theorem \ref{intro_poly-z-groups} from the introduction follows as a special case.

\bibliographystyle{halpha}
\bibliography{ControlledATheory_bibliography}

\end{document}